\documentclass[oneside]{amsart}
\usepackage{ifpdf}
\usepackage{amsmath, amsthm, amsfonts, amssymb, amscd}
\usepackage[active]{srcltx}
\usepackage{color}
\usepackage{marginnote}
\usepackage{caption}
\usepackage[T1]{fontenc}

    \ifpdf

      \usepackage[pdftex]{graphicx}  

    \else

      \usepackage[dvips]{graphicx}  

      \newcommand{\href}[2]{#2}

    \fi

\newcommand{\rotmz}{{\rho}_{\mathrm{mz}}}
\newcommand{\roterg}{{\rho}_{\mathrm{erg}}}
\newcommand{\rotmeas}{{\rho}_{\mathrm{m}}}
\newcommand{\wind}{\mc{W}}
     
\renewcommand{\fill}{\operatorname{Fill}}

\newcommand{\ine}{\operatorname{Ine}}
\newcommand{\ess}{\operatorname{Ess}}
\newcommand{\fess}{\mc{C}}
\newcommand{\cV}{\mathcal{V}}
\newcommand{\cH}{\mathcal{H}}
\newcommand{\length}{\operatorname{length}}
\newcommand{\diamup}{\mc{D}}

\newcommand{\orb}{\mc{O}}

\newcommand{\homeo}{\mathrm{Homeo}}

\newcommand{\norm}[1]{\left\|{#1}\right\|}
\newcommand{\snorm}[1]{\|{#1}\|}

\DeclareMathOperator{\inter}{\rm{int}}
\DeclareMathOperator{\bd}{\partial}
\DeclareMathOperator{\cl}{cl}
\DeclareMathOperator{\diam}{\rm{diam}}
\DeclareMathOperator{\conv}{\rm{Conv}}

\DeclareMathOperator{\fix}{\rm{Fix}}

\newcommand{\mc}{\mathcal}

\newcommand{\ol}{\overline}
\renewcommand{\hat}{\widehat}
\newcommand{\til}{\widetilde}

\newcommand{\R}{\mathbb{R}}\newcommand{\N}{\mathbb{N}}
\newcommand{\Z}{\mathbb{Z}}\newcommand{\Q}{\mathbb{Q}}
\newcommand{\T}{\mathbb{T}}
\newcommand{\D}{\mathbb{D}}\newcommand{\A}{\mathbb{A}}
\renewcommand{\SS}{\mathbb{S}}

\newcommand{\sm}{\setminus}

\newcommand{\id}{\mathrm{Id}}
\newcommand{\deck}{\operatorname{deck}}

\newcommand{\ie}{i.e.\ }
\newcommand{\eg}{e.g.\ }

\newtheorem{theorem}{Theorem}[section] 
\newtheorem{corollary}[theorem]{Corollary}
\newtheorem{lemma}[theorem]{Lemma}
\newtheorem{proposition}[theorem]{Proposition}
\newtheorem*{proposition*}{Proposition}

\newtheorem{claim}{Claim}
\newtheorem*{theorem*}{Theorem}
\newtheorem*{claim*}{Claim}

\newtheorem{theoremain}{Theorem}

\newtheorem{corollarymain}[theoremain]{Corollary}

\theoremstyle{definition}

\theoremstyle{remark}
\newtheorem{remark}[theorem]{Remark}

\AtBeginDocument{ \def\MR#1{}} 

\title[Fully essential dynamics]{Fully essential dynamics for area-preserving surface homeomorphisms}
\author{Andres Koropecki}
\address{Universidade Federal Fluminense, Instituto de Matem\'atica e Estat\'\i stica, Rua M\'ario Santos Braga S/N, 24020-140 Niteroi, RJ, Brasil}
\email{ak@id.uff.br}
\author{Fabio Armando Tal}
\address{Instituto de Matem\'atica e Estat\'\i stica, Universidade de S\~ao Paulo, Rua do Mat\~ao 1010, Cidade Universit\'aria, 05508-090 S\~ao Paulo, SP, Brazil}
\email{fabiotal@ime.usp.br}
\thanks{The first author was partially supported by FAPERJ-Brasil and CNPq-Brasil. The second author was partially supported by FAPESP and CNPq-Brasil}

\begin{document}

\begin{abstract} 
We study the interplay between the dynamics of area-preserving surface homeomorphisms homotopic to the identity and the topology of the surface. We define fully essential dynamics and generalize the results previously obtained on strictly toral dynamics to surfaces of higher genus. Non-fully essential dynamics are, in a way, reducible to surfaces of lower genus, while in the fully essential case the dynamics is decomposed into a disjoint union of periodic bounded disks and a complementary invariant externally transitive continuum $C$.  When the Misiurewicz-Ziemian rotation set has non-empty interior the dynamics is fully essential, and the set $C$ is (externally) sensitive on initial conditions and realizes all the rotational dynamics. As a fundamental tool we introduce the notion of homotopically bounded sets and we prove a general boundedness result for invariant open sets when the fixed point set is inessential.
\end{abstract}

\maketitle
\setcounter{tocdepth}{2}
\tableofcontents
\section*{Introduction}

An important aspect of the theory of dynamical systems on manifolds is the study of the restrictions and unique characteristics that the topology and geometry of the ambient space imposes on the dynamics of homeomorphisms and diffeomorphisms of the manifold. Given a homeomorphism $f$ of a topological manifold $M$, there are several ways to try and understand this interplay between the dynamical features of $f$ and the topology, and the foremost instrument is the study of the actions induced by $f$ on the homotopy and homology groups of $M$. In the case of two-dimensional dynamics, this sole instrument can provide a great wealth of information, as is the case where the dynamics is in the homotopy class of a Pseudo-Anosov map. The study of these induced actions, however, is a blunt instrument as it cannot differentiate between homotopic homeomorphisms. In particular, for the relevant case of dynamics homotopic to the identity, this study provides no information (although studying homotopy relative to fixed point sets has been productive, \eg \cite{handel-annulus, llibre-mackay}).

For one and two dimensional manifolds, rotation theory has provided comprehension of the role of the ambient space by means of the notions of rotation vectors and rotation sets. The classical case of the Poincar\'e rotation number for homeomorphisms of the circle and its importance is now well known, and the study of rotation theory for oriented two dimensional manifolds has drawn increased attention in the last decades, with significant progress (see for example \cite{m-z, franks-reali, llibre-mackay, handel-annulus, pollicott, franks-rot, zanata-vertical, davalos}), although most of this progress is restricted to maps of the torus or the annulus. In this work, we intend to further develop the work started in \cite{kt-strictly} and study the relationship, usually under some sort of conservative or nonwandering hypothesis, between dynamics and topology from yet a different view, that of the dynamically essential and fully essential points. We use this notion in an attempt to identify, for a given closed oriented surface $S$ of genus $g$, what are the homeomorphisms homotopic to the identity that interact with the whole homology of $S$, in which case we say that the map has fully essential dynamics, and we are interested in the special phenomena that must always be present for these maps.

In \cite{kt-strictly}, the authors introduced the notion of \emph{strictly toral homeomorphism}, which loosely speaking are maps whose dynamics cannot be understood in terms of the dynamics of a planar or annular homeomorphism; in other words the dynamics ``uses'' all the topology of $\T^2$. Concrete examples of strictly toral dynamics are homeomorphisms which have a rotation set with nonempty interior.
It turns out that strictly toral homeomorphisms have many interesting properties, described in terms of of a set $\ine(f)$ which is a union of homotopically bounded periodic disks (``periodic islands'') and a complementary compact invariant set $\fess(f)$ (``chaotic region'') which is connected, externally transitive, and contains all the interesting rotational dynamics. This description improved considerably one previously studied by J\"ager \cite{jager-elliptic}. The proofs of these properties rely heavily on a boundedness result for invariant open topological disks, a fundamental lemma that has seen other interesting applications \cite{kt-annular, kt-pseudo, kt-transitive}.  To better understand this boundedness result, let us introduce some concepts.

Let $S$ be a closed orientable surface endowed with a Riemannian metric and $\pi\colon \hat{S}\to S$ its universal covering map.  If $U$ is an open topological disk, we denote by $\diamup(U)$ its \emph{covering diameter}, defined as the diameter in $\hat{S}$ of any connected component of $\pi^{-1}(U)$. We say $U$ is homotopically bounded if $\diamup(U)<\infty$.
It was the question, motivated by an example \cite{kt-example}, of when an area-preserving homeomorphism can exhibit homotopically unbounded invariant topological disks that led to the fundamental lemma of \cite{kt-strictly}.

In this article, we extend the results from \cite{kt-strictly} to surfaces of higher genus, and we also improve the results for the case of the torus. As before, an important step is to understand when homotopically unbounded invariant topological disks can arise. Our first result addresses this question under general hypotheses. 
An \emph{inessential} open subset $U\subset S$ is one such that every loop in $U$ is homotopic to a point in $S$. A general set is inessential if it has an inessential neighborhood (see Section \ref{sec:essential} for more details). A \emph{cross-cut} of an open topological disk $U\subset S$ is a simple arc contained in $U$ joining two points of $\bd U$ (and if these points are equal, one assumes that the two open sets bounded by the closure of this arc intersect $\bd U$); see \ref{sec:cross-cuts} for a more precise definition. A set $X$ is wandering if $f^n(X)\cap X=\emptyset$ for all $n>0$.

\begin{theoremain}\label{th:main-disk} For any homomorphism $f\colon S\to S$ homotopic to the identity with an inessential set of fixed points there exists $M>0$ such that any open $f$-invariant topological disk without wandering cross-cuts has covering diameter at most $M$.
\end{theoremain}
This is a generalization of \cite[Theorem B]{kt-strictly} in several aspects. Firstly, it applies to higher genus surfaces; secondly, it assumes a weaker nonwandering-type condition, and finally for the case of $\T^2$ it removes a ``non-annular'' hypothesis that was present in the previous result. Theorem \ref{th:main-disk} also holds for surfaces with boundary; see Theorem \ref{th:main-disk-bd}.

It is clear that the hypothesis on $\fix(f)$ cannot be removed, the identity map providing a trivial counter-example (see \cite{kt-example} for a more interesting one). The hypothesis of having no wandering cross-cuts is a weak form of nonwandering condition. It will always hold if $f$ is area-preserving, or even if $f$ is nonwandering in $U$. The theorem  does not hold if one removes this condition; for instance, the flow on the torus obtained by suspension of a Denjoy example on the circle (with a single wandering interval) has a homotopically unbounded invariant topological disk (namely, the orbit of the wandering interval by the flow) and has no singularities; so its time-$1$ map provides an example where the conclusion of Theorem \ref{th:main-disk} fails (but, of course, this example has wandering cross-cuts).

If $U$ is connected but not necessarily simply connected, we define a notion of covering diameter $\diamup(U)$ which measures how much the set ``deviates'' from a surface with boundary of similar type, and we say that $U$ is homotopically bounded if $\diamup(U)$ is finite. The precise definition of $\diamup(U)$ is given in Section \ref{sec:covering-diameter}, but for the moment we may avoid these details by mentioning that $U$ being homotopically bounded is equivalent to saying that there exists a compact surface with (rectifiable) boundary $S_0\subset U$ such that if $\hat{U}$ is a connected component of $\pi^{-1}(U)$ and $\hat{S}_0$ is the component of $\pi^{-1}(S_0)$ contained in $\hat{U}$, then the distance of any point of $\hat{U}$ to $\hat{S}_0$ is bounded by a uniform constant. For instance Figure \ref{fig:rotbad} illustrates an homotopically unbounded set.

\begin{theoremain}\label{th:main-gen} 
Let $f\colon S\to S$ be an area-preserving homeomorphism homotopic to the identity such that $\fix(f)$ is inessential. Then any connected open $f$-invariant set is homotopically bounded.
\end{theoremain}

In the $C^r$-generic area-preserving setting, a similar (but finer) description was given in \cite{kln}; Theorem \ref{th:main-gen} indicates that this generic description is not too distant from the general case (except when the fixed point set is too large).

Theorem \ref{th:main-gen} is also the main tool that allows us to generalize to surfaces of higher genus virtually all of the results of \cite{kt-strictly} for area-preserving maps. This is done through the notion of dynamically essential and inessential points introduced in that work, which we develop further in this paper.

Given a homeomorphism $f\colon S\to S$, a point $x\in S$ is said to be \emph{dynamically inessential} if there is a neighborhood $V$ of $x$ such that $\orb_f(V):= \bigcup_{n\in \Z} f^n(V)$ is an inessential subset of $S$. The set of all such points is denoted $\ine(f)$. Any $x\in \ess(f):=S\sm \ine(f)$ is called a \emph{dynamically essential} point. 
A fully essential set is one whose complement is inessential. We say that $x$ is dynamically \emph{fully essential} if $\orb_f(V)$ is a fully essential set for every neighborhood $V$ of $x$. We denote the set of all dynamically fully essential points by $\fess(f)$. 

If $f\colon S\to S$ is homotopic to the identity, we say that $f$ has \emph{uniformly bounded displacement} if there exists $M>0$ such that some lift $\hat{f}\colon\hat{S}\to \hat{S}$ of $f$ satisfies $d(\hat{f}^n(z), z)\leq M$ for all $z\in \hat{S}$ and $n\in \Z$.

Finally, an invariant set $K\subset S$ is externally transitive if for any $U,V$ neighborhoods (in $S$) of points of $K$ there exists $n\in \Z$ such that $f^n(U)\cap V\neq \emptyset$. If the set of all $n\in \Z$ with this property has bounded gaps, $K$ is called \emph{externally syndetically transitive}.

\begin{theoremain}\label{th:essine} If $f\colon S\to S$ is an area-preserving homeomorphism homotopic to the identity of a hyperbolic surface, then one of the following properties hold:
\begin{itemize}
\item[(1)] There is $n>0$ such that $\fix(f^n)$ is essential;
\item[(2)] The map $f$ has uniformly bounded displacement;
\item[(3)] There is a homotopically bounded invariant open connected set $U\subset S$ which is essential but not fully essential;
\item[(4)] The set $\fess(f)$ is fully essential, connected, and externally syndetically transitive. Moreover, $\ess(f)=\fess(f)$, and $\ine(f)$ is the union of a family of pairwise disjoint homotopically bounded periodic open topological disks. 
\end{itemize}
\end{theoremain}
We say that $f$ has \emph{fully essential dynamics} when neither case (1), (2) nor (3) from Theorem \ref{th:essine} holds (note that these three cases may overlap).
Also as a consequence of Theorem \ref{th:main-disk} we obtain a slightly improved version of \cite[Theorem A]{kt-strictly}, providing a statement similar to the theorem above for the case of $\T^2$ (see Theorem \ref{th:torus}).

Non-fully essential dynamics may be thought as ``reducible'' dynamics in some sense. Specifically, case (1) of Theorem \ref{th:essine} is a degenerate situation where the set of non-fixed points of $f^n$, where the dynamics is non-trivial, is an open (possibly disconnected) surface of lower genus. Case (2) implies that there is a finite-to-one map from an open set in the plane onto $S$  which semiconjugates the dynamics of a planar homeomorphism with that of $f$ (indeed, in this case there exists a bounded open topological disk in the universal cover of $S$ which is invariant by a lift of $f$ and contains a fundamental domain). Finally, in case (3) one may consider a compact surface with boundary $S_0$ which is a neighborhood of $S\sm U$. Any orbit of $f$ either remains entirely in $U$ (which can be thought as a subset of a surface of lower genus) or in one connected component of $S_0$ (which again is a surface of lower genus, although not necessarily invariant). 

Let us now relate the concept of fully essential dynamics to the rotation theory viewpoint. Generalizing the notion of Poincar\'e rotation number, homological rotation sets can be defined in several ways for surfaces of positive genus as a subset of the first homology group $H_1(S,\R)$; see for instance \cite{pollicott, franks-rot}.  For the case of the torus, the most widely adopted definition is due to Misiurewicz and Ziemian \cite{m-z}. In \cite[Theorem C]{kt-strictly} it was shown that an area preserving homeomorphism of $\T^2$ whose rotation set has nonempty interior has fully essential dynamics.
We discuss the definitions of rotation vectors of points and invariant measures as well as rotation sets in Section \ref{sec:rot}, where we also introduce the Misiurewicz-Ziemian rotation set $\rotmz(f)\subset H_1(S,\R)$ for surfaces of higher genus, which is the definition that most closely resembles the rotation set of a toral homeomorphism. We also define the rotation set over a subset $U\subset S$, denoted $\rotmz(f,U)$.
Many of the dynamical consequences of the rotation set for toral homeomorphisms fail to hold on surfaces of higher genus due to the fact that, being a homological object, the rotation set misses information when the fundamental group is no longer abelian (for example, a rotation along an essential homologically trivial loop is undetected by the rotation set). 
Nonetheless, we are able to fully generalize the results from \cite{kt-strictly}.

We say that an $f$-invariant set $K\subset S$ is \emph{externally sensitive on initial conditions} if there is a constant $c>0$ such that, for any $x\in K$ and any neighborhood $U\subset S$ of $x$ there is $n>0$ such that $\diam(f^n(U))\geq c$.
A rational subspace of $H_1(S,\R)$ is a subspace generated by elements of $H_1(S,\Z)$. The next theorem says that unless the rotation set is considerably small, the dynamics will be fully essential. We note that the hypothesis for the rotation set automatically holds if the rotation set has nonempty interior:

\begin{theoremain}\label{th:interior-essential} If $S$ is hyperbolic and $f\colon S\to S$ is an area-preserving homeomorphism homotopic to the identity. Suppose that $\rotmz(f)$ is not contained in a union of $g+1$ rational proper subspaces of $H_1(S, \R)$. Then:
\begin{itemize}
\item[(1)] $\fess(f)=\ess(f)$, which is a fully essential continuum, and $\ine(f)$ is a disjoint union of periodic bounded disks;
\item[(2)] $\fess(f)$ is externally syndetically transitive and sensitive on initial conditions;
\item[(3)] If $U$ is any neighborhood of  $x\in \fess(f)$, then $\rotmz(f,U)=\rotmz(f)$. Moreover, there is a constant $c>0$ depending only on $\rotmz(f)$ such that 
	$$\liminf_{n\to\infty} \diamup(f^n(U))/n \geq c$$
\end{itemize}
In particular, $f$ has fully essential dynamics. 
\end{theoremain}
In fact, part (3), which essentially says that the ``local'' rotation set near any point of $\fess(f)$ is equal to the rotation set of $f$, holds in general, even when $\rotmz(f)$ has empty interior, and in a stronger ``uniform'' way; see Lemma \ref{lem:localrot-uniform}. 
The constant $c$ in part (3) is in fact the diameter of $\rotmz(f)$ under the stable norm on the first homology group of $S$ (see Section \ref{sec:rot}).

When the dynamics is fully essential, there is an important connection between the set of fully essential points and the realization of rotation vectors by periodic orbits and ergodic measures (again, see Section \ref{sec:rot} for definitions):

\begin{theoremain}\label{th:reali-ess} If $S$ is hyperbolic and $f\colon S\to S$ be an area-preserving homeomorphism homotopic to the identity with fully essential dynamics. Then,
\begin{itemize}
\item[(1)] For each periodic point $p$ of $f$, there is a periodic point of $f$ in $\fess(f)$ with the same rotation vector. Moreover, there exists a periodic point in $\fess(f)$ which is Nielsen equivalent to $p$.
\item[(2)] Every ergodic probability measure with a non-rational rotation vector is supported in $\fess(f)$.
\end{itemize}
\end{theoremain}

As a final application, we give a characterization of transitivity.
\begin{corollarymain}\label{cor:trans} Suppose that $f\colon S\to S$ is an area-preserving homeomorphism homotopic to the identity and assume that $f$ has fully essential dynamics. Then $f$ is transitive if and only if there are no homotopically bounded periodic disks, and moreover in that case $f$ is syndetically transitive.
\end{corollarymain}

As in \cite{kt-strictly}, we expect the bound on the covering diameter of periodic disks in Theorem \ref{th:interior-essential}(1) to be uniform, although even in the case of the torus we are only able to obtain a bound that depends on the period of the disks. However, for $C^{1+\alpha}$-diffeomorphisms of $\T^2$, a theorem of Addas-Zanata provides a uniform bound \cite{zanata-disk}.
We note that, if $f$ is area-preserving, the sole hypothesis that $\fix(f^n)$ is inessential for every $n>0$ already implies that all periodic topological disks are homotopically bounded (due to Theorem \ref{th:main-gen}), but this bound  may fail to be independent of the periods of the disks without additional hypotheses. We present an example of this fact in Section \ref{sec:example}.

Let us say a few words about the proofs, along with an outline of the article.
In Section \ref{sec:prelim} we introduce the preliminary notation and terminology, along with some useful lemmas. In particular, a key concept in this work is the covering diameter which is defined in \ref{sec:covering-diameter}. An important tool in our proofs is the equivariant Brouwer theory developed by Le Calvez \cite{lecalvez-equivariant}. 
Section \ref{sec:linking} introduces the main results about Brouwer-Le Calvez foliations, along with the notion of Linking of open topological disks, which plays a central role in the proof of Theorem \ref{th:main-disk}. We prove two Linking Lemmas which are similar but improved versions of results already used in \cite{kt-strictly, kt-pseudo}.

Section \ref{sec:main-disk} is devoted to the proof of Theorem \ref{th:main-disk}. The proof is comprised by three main parts. The first part is a reduction to the case where the fixed point set is totally disconnected. This is necessary in order to have a convenient setting to apply the Brouwer-Le Calvez theory and linking lemmas. The main idea is to fill-in and collapse the components of $\fix(f)$, but a careful analysis is needed to maintain the hypotheses after this process.
This is done in \S\ref{sec:fix-collapse}. The second part, in \S\ref{sec:contractible}, considers the case where the invariant disk $U$ lifts to a disk invariant by the ``natural'' lift of $f$ to the universal covering (obtained by lifting a given isotopy from the identiy). We call this the \emph{contractible} case. The proof of this part relies on the tools introduced in Section \ref{sec:linking}. The third and final part of the proof concerns the case where $U$ is not invariant by the natural lift. This case is only relevant for the case where $S$ is a hyperbolic surface, and to deal with it we reduce it to a problem on the torus by considering the quotient of the universal cover $\D\to S$ by some adequate deck transformation (which leads to an annulus), compactifying this annulus with two boundary circles and then ``gluing'' those two circles. Then we may apply the results from the previous sections to conclude the proof.
In Section \ref{th:main-disk-bd} we prove a version of Theorem \ref{th:main-disk} for surfaces with boundary, which will be helpful in the following section.

Section \ref{sec:main-gen} contains the proof of Theorem \ref{th:main-gen}, which is done by means of a suitable ``surgery'', which involves modifying $f$ at a neighborhood of each topological end of $U$ to obtain an invariant circle, then ``cutting'' along those  circles and collapsing them to obtain one invariant topological disk for each topological end of $U$. Applying Theorem \ref{th:main-disk} to these disk and translating the results, we conclude that $U$ is homotopically bounded.

Section \ref{sec:essine} studies the key notion of dynamically essential and inessential points and their consequences.
In particular, it includes the proof of Theorem \ref{th:essine}, a version of the same theorem for the torus (Theorem \ref{th:torus}), and a proof of Corollary \ref{cor:trans}.

Section \ref{sec:rot} introduces the notion of rotation vectors and rotation sets for points and for invariant measures. We adopt a definition in the vein of Misiurewicz and Ziemian \cite{m-z}, modifying the definition of pointwise rotation set given by Franks in \cite{franks-rot}. In \S\ref{sec:displacement}, we study the restrictions imposed to the rotation set by the existence of essential but not fully essential invariant open sets, among other properties. In \S\ref{sec:shadow} we obtain a useful general result that shows that for any neighborhood $U$ of a point of $\fess(f)$, the rotational behavior (in homology) of any orbit of $f$ is accompanied by orbits of $U$. This is the main result required to show that the local rotation set at points of $\fess(f)$ is equal to the global rotation set. We also show that the linear growth of $\diamup(f^n(U))$ is bounded below by the diameter of the rotation set. With these results, the proof of Theorem \ref{th:interior-essential} is done in a a few lines.  Finally, in \S\ref{sec:reali-ess} we deal with the realization of rotation vectors by periodic orbits and invariant measures in $\fess(f)$, proving Theorem \ref{th:reali-ess}.

\section{Preliminaries}
\label{sec:prelim}

In this article, $S$ will always denote (unless stated otherwise) a connected orientable surface of nonpositive Euler characteristic $\chi(S)\leq 0$ (in most cases, closed), and $\pi\colon \hat{S}\to S$ denotes its universal covering map. We assume that $S$ is endowed with a Riemannian metric, and $\hat{S}$ is endowed with the lifted metric, so that the group of covering transformations $\deck(\pi)$ consists of isometries. Note that $\hat{S}$ is homeomorphic to $\R$ since $\chi(S)\leq 0$.
The notation $d(\cdot, \cdot)$ will be used for the distance in $\hat{S}$ and also for the induced distance in $S$ (with a small abuse of notation).

An \emph{arc} in $S$ is a continuous map $\sigma\colon I\to S$, where $I$ is an interval (often assumed to be $[0,1]$). When no ambiguity is likely, we abuse the notation and denote its image $\{\sigma(t):t\in I\}$ by $\sigma$ as well.
A \emph{loop} is an arc $\gamma\colon[0,1]\to S$ such that $\gamma(0)=\gamma(1)$. A \emph{simple} arc is an injective arc, and a \emph{simple} loop is a loop which is injective except at the endpoints of its domain. 

Let $\homeo(S)$ be the space of homeomorphisms of $S$. An isotopy $\mc{I}=(f_t)_{t\in [0,1]}$ from $f_0$ to $f_1$ is a continuous arc in $\homeo(S)$ joining $f_0$ to $f_1$.
We denote by $\homeo_0(S)$ the connected component of the identity in $\homeo(S)$, \ie the homeomorphisms isotopic to the identity. By \cite{hamstrom}, if $S$ is hyperbolic (\ie $\chi(S)<0$) the space $\homeo_0(S)$ is homotopically trivial. Thus, if $\mc{I}$ and $\mc{I}'$ are two isotopies from the identity to a homeomorphism $f$, then there is a homotopy with fixed endpoints in $\homeo_0(S)$ from $\mc{I}$ to $\mc{I}'$. We remark that this is no longer true if $S$ is a torus.

The isotopy $\hat{\mc{I}}=(\hat{f}_t)_{t\in [0,1]}$ obtained by lifting $\mc{I}$ with basepoint $\hat{f}_0=\id$ is called the natural lift of $\mc{I}$. The map $\hat{f}=\hat{f}_1$ is called the natural lift of $f$ associated to the isotopy $\mc{I}$. Natural lifts of a homeomorphism are characterized by the property of commuting with all deck transformations.

If $S=\T^2$, the natural lift $\hat{f}$ depends on the chosen isotopy from the identity to $f$. However, if $S$ is hyperbolic (\ie $\chi(S)<0$), it depends only on $f$ and not on the chosen isotopy, so we call $\hat{f}$ \emph{the} natural lift of $f$.

\subsection{Topological ends or ideal boundary}
\label{sec:ends}

We briefly describe the compactification by topological ends of an open connected set $U\subset S$, following \cite{richards}. A \emph{boundary representative} of $U$ is a sequence $P_1\supset P_2\supset\cdots$ of connected unbounded (\ie not relatively compact in $U$) open subsets of $U$ such that $\bd_U P_n$ is compact for each $n$ and for any compact set $K\subset U$, there is $n_0>0$ such that $P_n\cap K=\emptyset$ if $n>n_0$ (here we denote by $\bd_U P_n$ the boundary of $P_n$ in $U$). Note that the last condition is equivalent to saying that $\bigcap_{n\in \N} \cl_U P_n=\emptyset$. Two boundary representatives $(P_i)_{i\in \N}$ and $(P_i')_{i\in \N}$ are said to be equivalent if for any $n>0$ there is $m>0$ such that $P_m\subset P_n'$, and vice-versa. The \emph{ideal boundary} $\mc{E}$ of $U$ is defined as the set of all equivalence classes of boundary representatives, which are called topological ends of $U$ or ideal boundary points. The space $U_* = U\cup \mc{E}$ may be endowed with the topology generated by open sets in $U$ together with sets of the form $V \cup V'$, where $V$ is an open set in $U$ such that $\bd_U V$ is compact, and $V'$ denotes the set of elements of $\mc{E}$ which have some boundary representative $(P_i)_{i\in \N}$ such that $P_i\subset V$ for all $i\in \N$. Defined in this way, $U_*$ is a compact space called the \emph{ends compactification} or \emph{ideal completion} of $U$. 

The fact that $U\subset S$ which is a closed orientable surface implies that $U_*$ is a closed orientable surface as well (of genus at most equal to the genus of $S$). Any homeomorphism $f\colon U\to U$ extends to a homeomorphism $f_*\colon U_*\to U_*$, which preserves orientation if $f$ does. The set $\mc{E}\subset U_*$ is always compact and totally disconnected.

A regular topological end is an isolated element of $\mc{E}$. For such an end $\mathfrak{e}$ one can always find a topological disk $D$ bounded by a simple loop $\gamma$ such that $\ol{D}\cap \mc{E}=\{\mathfrak{e}\}$. An annulus $A=D\sm \{\mathfrak{e}\}\subset U$ obtained in this form is called a \emph{collar} of the topological end $\mathfrak{e}$ in $U$.

We may always find an increasing sequence of compact surfaces with boundary $K_n\subset U$ such that $\bigcup_n K_n=U$. If the number of connected components of $U\sm K_n$ is finite (and so eventually constant equal to some number $m$) then $U$ has exactly $m$ topological ends and each component of $U\sm K_n$ is (if $n$ is large enough) a collar of some topological end.

\subsection{Essential, inessential and filled sets}\label{sec:essential}

Assume $S$ is closed (and $\chi(S)\leq 0$). If $U\subset S$ is an open set, we say that $U$ is \emph{inessential} if every loop contained in $U$ is homotopically trivial in $S$. This is equivalent to saying that the inclusion map $U\to S$ is homotopic to a constant, or that $U$ is contained in a disjoint union of open topological disks.
If $E\subset S$ is an arbitrary set, we say that $E$ is inessential if some open neighborhood $U$ of $E$ is inessential. 
An \emph{essential} set is one that is not inessential, and a \emph{fully essential} set is a set whose complement is inessential.
An essential open or closed set always has an essential connected component, and if the set is fully essential this component is unique. We note however that the notion of being essential for arbitrary sets is more subtle; for instance it is not difficult to construct an essential set whose every connected component is inessential. 
 
Given a set $E\subset S$, its $\emph{filling}$ is the set $\fill(E)$ consisting of the union of $E$ with all connected components of $S\sm E$ which are inessential. A \emph{filled} set is one that is equal to its own filling. Some basic properties of the filling are:
\begin{itemize}
\item $\fill(\fill(E))=\fill(E)$
\item If $E$ is connected then so is $\fill(E)$;
\item if $E$ is inessential, then so is $\fill(E)$;
\item if $E$ is fully essential, then $\fill(E)=S$;
\item if $E$ is $f$-invariant by a homeomorphism $f\colon S\to S$, then so is $\fill(E)$.
\end{itemize}

Whenever one has a filled compact set which is inessential, one may collapse its connected components to obtain a totally disconnected set (see Proposition \ref{pro:collapse}). This will be useful to simplify the fixed point set in our proofs.

If $S=\R^2$, we define $\fill(E)$ as the union of $E$ with all bounded connected components of its complement. In this setting, it is also true that $\fill(\fill(E))=\fill(E)$, and if $E$ is invariant by a given homeomorphism then so is $\fill(E)$. Note that a compact filled set in $\R^2$ is a non-separating planar continuum. 
\begin{remark}\label{rem:Q}
We will often make use of the following fact: if $U\subset S$ is open and fully essential, given any bounded open topological disk $D\subset \hat{S}$, there exists a topological disk $Q$ bounded by a simple loop $\gamma$ such that $\pi(\gamma)\subset U$ and $D\subset Q$. One way of doing this is as follows: being open and fully essential, $U$ contains a fully essential subset $\gamma_0$ which is the image of a loop, which we may assume (by choosing it polygonal, for instance) is locally simple and has finitely many self-intersections. 
This implies that $S\sm \gamma_0$ has finitely many connected components, each of which is a topological disk bounded by a loop contained in $\gamma_0$. As a consequence, the set $\mc{C}$ consisting of closures of connected components of $\hat{S}\sm \pi^{-1}(\gamma_0)$ covers $\hat{S}$ and is locally finite. Since $\ol{D}$ is bounded, $\ol{D}$ is covered by finitely many elements $V_1, \dots, V_m$ of $\mc{C}$. Then we may choose $Q$ to be the connected component of the interior of $\fill(\bigcup_{i=1}^m V_m)$ containing $D$.
\end{remark}

\begin{proposition}\label{pro:ends} If $g\geq 1$ is the genus of $S$ and $U\subset S$ is a filled connected essential open set, then $U$ has at most $2g$ topological ends and $S\sm U$ has at most $g$ connected components. In addition, if $f\colon S\to S$ is homotopic\footnote{Recall that in a closed orientable surface, two homeomorphisms are homotopic if and only if they are isotopic \cite{epstein}.} to the identity and $f(U)=U$, then each topological end of $U$ is fixed.
\end{proposition}
\begin{proof}
Given any filled essential surface with boundary $U_0\subset U$, if $b$ is the number of boundary components of $U_0$ and $g_0$ is its genus, then one has $2-2g_0-b = \chi(U_0)\geq \chi(S) = 2-2g$, so $b<=2g-2g_0\leq 2g$. In addition, since $U_0$ is filled, if $V$ is a connected component of $S\sm U_0$ with genus zero, then $V$ has at least two boundary components and therefore $S\sm V$ has genus strictly smaller than $S$. Hence each connected component of $S\sm U_0$ contributes at least $1$ to the genus of $S$, and we deduce that there are at most $g$ connected components in $S\sm U_0$.
Since $U$ can be written as an increasing union of filled surfaces with boundary $(U_k)$, each with at most $2g$ boundary components, the first claim follows. In addition, since $S\sm U$ is obtained by a decreasing intersection of the corresponding components of $S\sm U_k$, which are at most $g$ for each $k$, it follows that $S\sm U$ has at most $g$ components. 

If $m\leq 2g$ is the number of topological ends of $U$, the surface $U_0$ may be chosen such that $U\sm U_0$ is a disjoint union of $m$ topological annuli (one collar of each end of $U$). Let $U_*=U\sqcup\{\mathfrak{e}_1,\dots, \mathfrak{e}_m\}$ denote the ends compactification of $U$, which is a closed orientable surface. Each connected component of $U_*\sm U_0$ is a closed topological disk $D_i$ containing exactly one end $\mathfrak{e}_i$. The map $f$ extends to an orientation-preserving homeomorphism $f_*$ permuting the topological ends of $U$. Suppose that some end is not fixed, for instance $f_*(\mathfrak{e}_1)=\mathfrak{e}_2$. Then there exists a positively oriented simple loop $\gamma$ contained in $D_1$ and bounding a disk $D_1'\subset D_1$ such that $f_*(D_1')\subset D_2$. Since $\gamma$ and $f(\gamma)$ are isotopic in $S$ and disjoint, they bound some topological annulus $A\subset S$ such that they are isotopic in $\ol{A}$. Since $\bd A\subset U$, either $A\subset U$ or $S\sm A\subset U$. If $A\subset U$, then $U$ has exactly two ends which are permuted by $f_*$ and $\gamma$ is isotopic to $f(\gamma)$ in $U$, so these two loops bound disjoint disks $D_1$ and $f_*(D_1)$ in $U^*$. This implies that the loops $\gamma$ and $f(\gamma)$ have reversed orientation in $U_*\simeq \mathbb{S}^2$, a contradiction.
If, on the other hand, $S\sm U\subset A$, then some essential connected component of $\bd U$ separates $A$, and so $D_1\sm \{\mathfrak{e}_1\}$ and $f(D_1)\sm \{\mathfrak{e}_2\}$ are contained in $A$, and by a similar argument this contradicts the preservation of orientation of $f_*$ (since $D_1$ is on the right-hand side of $\gamma$ and $f(D_1)$ is on the left-hand side of $f(\gamma)$, or vice-versa).
\end{proof}

\subsection{Covering distance}\label{sec:covering-diameter}

In this subsection we assume that $S$ is a compact surface with or without boundary, endowed with a Riemannian metric. No assumption on $\chi(S)$ is made.
Let $W\subset S$ be an arcwise connected set. The \emph{covering distance} $d_W(x,y)$ between two points $x,y\in W$ is the smallest possible length of a rectifiable arc in $S$ joining $x$ to $y$ and homotopic with fixed endpoints (in $S$) to some arc contained in $W$. Note that if $W\subset W'$, then $d_{W}(x,y)\geq d_{W'}(x,y)$ for any $x,y\in W$.

An alternative but equivalent definition of this distance is as follows: if $\hat{W}$ is a connected component of $\pi^{-1}(W)$, then $\pi|_{\hat{W}}\colon \hat{W}\to W$ is a covering map. The covering distance is precisely the distance induced in $W$ by projecting the distance from the covering; in other words: 
$$d_W(x,y)=\inf\bigg\{d(\hat{x},\hat{y}) :\hat{x}\in \pi^{-1}(x)\cap \hat{W},\, \hat{y}\in \pi^{-1}(y)\cap \hat{W}\bigg\}.$$
Note that the infimum is always attained, and it is independent of the chosen connected component $\hat{W}$, as any other component is mapped to $\hat{W}$ by a deck transformation (which is an isometry of $\hat{S}$).

We denote by $\diam_D(X)$ the diameter of a set $X$ with respect to a metric $D$.
The \emph{covering diameter} of $W$ is $\diamup(W)=\diam_{d_W}(W)$, \ie 
$$\diamup(W) = \sup \{d_W(x,y): x,y\in W\}.$$
We will say that $W$ is homotopically bounded if $\diamup(W)$ is finite. 
Note that $\diamup(W)$ depends on the underlying surface $S$ (and its metric). If it is necessary to avoid ambiguity we use the notation $\diamup_S(W)$ to emphasize this.
In the case that $W$ is inessential, $\diamup(W)$ is the same as the diameter in $\hat{S}$ of any connected component of $\pi^{-1}(W)$. Note that the covering diameter of any compact surface with rectifiable boundary is finite.
\begin{remark}
The covering distance depends on the underlying metric of $S$; however since $S$ is compact, any two Riemannian metrics are equivalent and thus the covering distances and diameters are equivalent. In other words, there is a constant $c\geq 1$ such that if $d_W'$ denotes the covering metric on $W$ induced by a different underlying metric on $S$, then $c^{-1}d_W(x,y)\leq d_W'(x,y)\leq cd_W(x,y)$ for all $x,y\in W$. In particular, if $\diamup'(W)$ denotes the covering diameter in the new metric, then $c^{-1}\diamup'(W) \leq \diamup(W)\leq c\diamup(W)$. 
\end{remark}

\begin{proposition}\label{pro:diam-partition} If $W_1,\dots W_k$ are arcwise connected sets and $W=\bigcup_{i=1}^k W_i$ is arcwise connected, then $\diamup(W)\leq \sum_{i=1}^k \diamup(W_i)$.
\end{proposition}
\begin{proof} The definition of covering distance easily implies that $d_W(x,y) \leq d_{W_i}(x,y)$ if $x,y\in W_i$. Thus the diameter of $W_i$ using the $d_W$ metric is bounded above by its diameter using the $d_{W_i}$ metric, which is $\diamup(W_i)$, and the claim follows directly form this fact.
\end{proof}

\begin{proposition}\label{pro:diam-fill} If $U\subset S$ is an open connected set and $\chi(S)\leq 0$, then  $$\diamup(\fill(U))\geq \diamup(U).$$ 
\end{proposition}
\begin{proof} It follows from the fact that any arc contained in $\fill(U)$ is homotopic with fixed endpoints in $S$, to an arc contained in $U$.
\end{proof}

The next proposition follows easily from the definitions:
\begin{proposition}\label{pro:diam-sub} If $S'\subset S$ is a compact surface with boundary, then for any arcwise connected set $W\subset S'$, $\diamup_{S'}(W)\geq \diamup_{S}(W)$.
\end{proposition}

In the opposite direction we have the following:
\begin{proposition}\label{pro:cover-disk} If $U\subset S$ is an open homotopically bounded topological disk, and $D\subset U$ is a closed topological disk with rectifiable boundary then $U\sm D$ is homotopically bounded in $S\sm \inter{D}$.
\end{proposition}
\begin{proof} If $S$ is a sphere the claim is obvious, so we may assume that $\hat{S}$ is homeomorphic to a subset of $\R^2$.

Let $U_0$ be a connected component of $\pi^{-1}(U)$, which is bounded by our hypothesis, and let $D_0$ be the connected component of $\pi^{-1}(D)$ in $U_0$. Note that since $U$ is an open  topological disk, for any deck transformation $T\neq \id$ one has $TU_0\cap \ol{U_0}=\emptyset$. Moreover, if $K=\fill(\ol{U_0})$, then  $TU_0\cap K=\emptyset$, since otherwise $TU_0$ would be contained in the union of all bounded connected components of $\hat{S}\sm \ol{U_0}$, which would imply $TK= \fill(\ol{TU_0})\subset K$ contradicting the boundedness of $K$. In particular $K$ is disjoint from $\pi^{-1}(D)\sm D_0$, and since $K$ is filled there exists a closed topological disk $W$, which may be chosen with rectifiable boundary, such that $K\subset W$ and $W$ is disjoint from $\pi^{-1}(D)\sm D_0$.
Since $W\sm \inter D_0$ is an annulus bounded by two rectifiable loops, there is a constant $M>0$ such that any pair of points of $W$ can be joined by an arc in $W$ of length at most $M$. Given $x,y\in U\sm D$, choose the corresponding lifts $\hat{x}, \hat{y}\in U_0\sm D_0$, and let $\gamma$ be an arc in $W$ joining $\hat{x}$ to $\hat{y}$ of length at most $M$. We claim that there exists an arc $\gamma'$ in $U_0\sm D_0$ joining $\hat{x}$ to $\hat{y}$ and homotopic with fixed endpoints in $W$ to $\gamma$. Indeed, $U_0\sm D_0$ is a topological annulus which is essential in $W$, so if $\sigma$ is any arc in $U_0\sm D_0$ joining $\hat{y}$ to $\hat{x}$, then the loop $\gamma*\sigma$ is homotopic to some loop $\eta$ contained in $U_0\sm D_0$ with base point $\hat{x}$. This implies that $\gamma'=\eta*\sigma^{-1}$ is homotopic with fixed endpoints in $W$ to $\gamma$. 
Since $\pi(W)\subset S\sm \inter D$, this implies that the covering diameter in $S\sm \inter D$ of $U\sm D$ is at most $M$, as claimed.
\end{proof}

\subsection{Some properties of hyperbolic surfaces}\label{sec:hyperbolic}

In this subsection we state some technical lemmas about isotopies on hyperbolic surfaces that will be useful ahead; but first we recall some general facts about hyperbolic surfaces that we will need. For details, see \cite{casson} or \cite{farb-margalit}.

Given an essential loop $\gamma:[0,1]\to S$, an \emph{extended lift} of $\gamma$ is an arc $\Gamma\colon \R \to \D$ obtained by concatenation of the arcs $\hat{\gamma}^n$, where $\hat{\gamma}\colon [0,1]\to \hat{D}$ is any lift of $\gamma$.

Suppose that $S$ is a hyperbolic closed surface (\ie $\chi(S)<0$). Then we may identify its universal covering $\hat{S}$ with the Poincar\'e disk $\D$ endowed with the hyperbolic metric. Any nontrivial deck transformation $T\in \deck(\pi)$ is a hyperbolic isometry, and extends to the ``boundary at infinity'' $\bd \D$ to a map which has exactly two fixed points. These fixed points are the endpoints of some $T$-invariant geodesic $\Gamma_T$ of $\D$. For any $z\in \ol{\D}$, the sequence  $T^n(z)$ converges to one endpoint of $\Gamma_T$ as $n\to -\infty$ and to the other one as $n\to \infty$. Any subarc of $\Gamma$ joining a point $z$ to $Tz$ projects into $S$ to an essential loop $\gamma_T$ which is the unique geodesic in its free isotopy class. If $\gamma$ is a loop in $S$ freely homotopic to $\gamma_T$, then any extended lift $\Gamma$ remains a bounded distance away from $\Gamma_T$ (as a subset of $\D$), and therefore it has the same endpoints in $\bd \D$ as $\Gamma_T$. 

These facts imply that two extended lifts of an essential loop coincide if and only if they share the same endpoints in $\bd \D$. In addition, if $R\in \deck(\pi)$ is a deck transformation that commutes with $T$, then $\Gamma_R=\Gamma_T$, and the group of all deck transformations fixing the same two points of $\bd \D$ is cyclic (generated by $T$ if we assume that $\gamma_T$ is in the homotopy class of a simple loop).

If $\hat{f}$ is the natural lift of a homeomorphism $f\colon S\to S$ isotopic to the identity, then $\hat{f}$ extends continuously to $\ol{\D}$ fixing every point of $\bd \D$. 

\begin{lemma}\label{lem:annulus-lift} Suppose $S$ is closed and hyperbolic, and $f\colon S\to S$ is a homeomorphism isotopic to the identity. If $A\subset S$ is a topological annulus such that $A$ contains an essential simple loop $\gamma$ and its image by $f$, then any connected component $\hat{A}$ of $\pi^{-1}(A)$ contains a connected component of $\pi^{-1}(\gamma)$ and its image by the natural lift of $f$.
\end{lemma}
\begin{proof}
Since $A$ is an annulus, every extended lift of an essential loop in $A$ has the same pair of endpoints in $\bd \D$. Let $\Gamma\subset \hat{A}$ be the extended lift of $\gamma$ in $\hat{A}$. Since $f(\gamma)\subset A$, there is an extended lift $\Gamma'$ of $f(\gamma)$ in $\hat{A}$, and $\Gamma' = T\hat{f}(\Gamma)$ for some deck transformation $T$. Since $\hat{f}$ extends to $\bd \D$ fixing every point, it follows that $T\hat{f}(\Gamma)$ and $\hat{f}(\Gamma)$ are two extended lifts of the same loop sharing the same endpoints, therefore they coincide.
\end{proof}

\begin{lemma}\label{lem:hyp-isotopy} Suppose $S$ is closed and hyperbolic and let $(f_t)_{t\in [0,1]}$ be an isotopy from the identity to $f=f_1$. Let $U\subset S$ be a connected open set, and suppose that $U_0\subset U$ is a filled connected open set such that $f(U_0)\subset U$ and some essential loop in $U_0$ is homotopic in $U_0$ to its own image. Then for every $x\in U_0$ the arc $(f_t(x))_{x\in [0,1]}$ is isotopic with fixed endpoints in $S$ to an arc contained in $U$.
\end{lemma}
\begin{proof}
If $\gamma$ is an essential loop in $U_0$ and $f(\gamma)$ is homotopic to $\gamma$ in $U_0$, then there is a topological annulus $A\subset U_0$ containing both $\gamma$ and $f(\gamma)$.
Let $\hat{U}$ be a connected component of $\pi^{-1}(U)$. Then by Lemma \ref{lem:annulus-lift}, any connected component $\hat{U}_0$ of $\pi^{-1}(U_0)$ contained in $\hat U$ intersects $\hat{f}(\hat{U}_0)$, where $\hat{f}$ is the natural lift of $f$. Since $\hat{f}(\hat{U}_0)$ is contained in some connected component of $\pi^{-1}(U)$, it follows that $\hat{f}(\hat{U}_0)\subset \hat{U}$. Thus, if $x\in U_0$, the lift of $(f_t(x))_{t\in [0,1]}$ with basepont $\hat{x}\in \hat{U}$ is homotopic to an arc contained in $\hat{U}$ connecting $\hat{x}$ to $\hat{f}(x)$, and the result follows by projecting this homotopy.
\end{proof}

\begin{lemma}\label{lem:periodic-inter} Suppose $S$ is closed and hyperbolic, $f\colon S\to S$ is a homeomorphism isotopic to the identity, and $U\subset S$ a connected essential open set such that $f^n(U)=U$ for some $n>0$. Then $f(U)\cap U\neq \emptyset$. Moreover, every connected component of $\pi^{-1}(U)$ intersects its own image by the natural lift of $f$. 
\end{lemma}
\begin{proof} It suffices to consider the case where $U$ is filled. Applying Lemma \ref{lem:annulus-lift} to $f^n$, we see that if $\hat{U}$ is a connected component of $\pi^{-1}(U)$ and $\hat{f}$ is the natural lift of $f$, then $\hat{f}^n(\hat{U})$ intersects $\hat{U}$, so $\hat{f}^n(\hat{U})=\hat{U}$. Suppose that $\hat{f}(\hat{U})\cap \hat{U}=\emptyset$. If $\Gamma\subset \hat{U}$ is an extended lift of some simple loop in $U$ which is essential in $S$, then $\hat{f}(\Gamma)$ is disjoint from $\Gamma$ and they connect the same two points of $\bd \D$. From the fact that $\hat{f}$ preserves orientation and (its extension) fixes $\bd \D$, we see that one connected component $W$ of $\D\sm \Gamma$ satisfies $\hat{f}(W)\subset W \subset \hat{f}^{-1}(W)$ (See Figure \ref{fig:W}). This implies that $\hat{U}$ is contained in the topological disk $D = \hat{f}^{-1}(W)\sm \hat{f}(\ol{W})$ bounded by $\hat{f}(\Gamma)$ and $\hat{f}^{-1}(\Gamma)$. But $\hat{f}^n(\hat{U})\subset \hat{U} \subset D$ whereas, if $n>1$, $\hat{f}^n(D)\subset \hat{f}(W)$ which is disjoint from $\hat{U}$. Thus $\hat{f}(\hat{U})\cap \hat{U}\neq \emptyset$.
\end{proof}

\begin{figure}
\centering
\begin{minipage}{.45\textwidth}
  \centering
  \includegraphics[width=.9\linewidth]{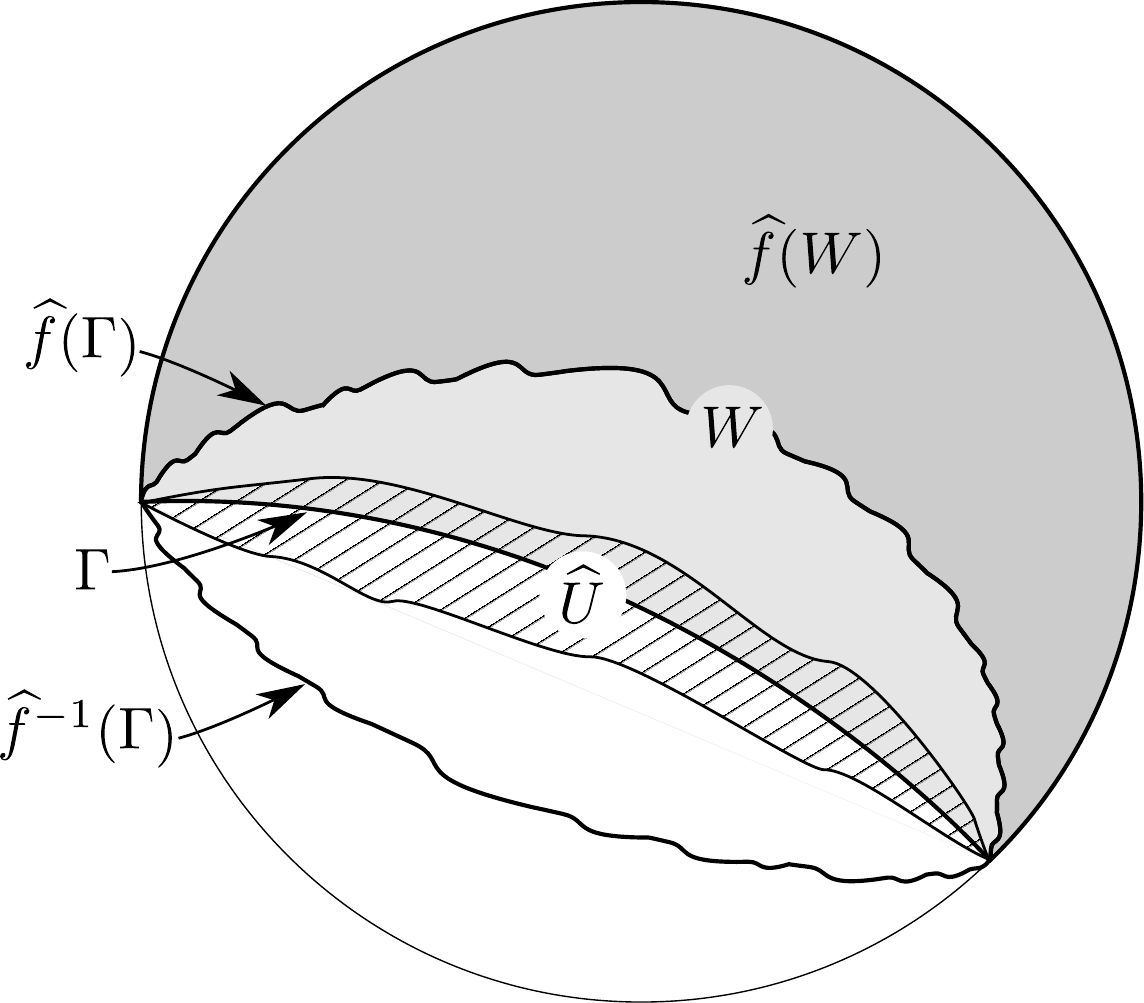}
\caption{}
	\label{fig:W}
\end{minipage}%
\hspace{.05\textwidth}
\begin{minipage}{.45\textwidth}
  \centering
  \includegraphics[width=.9\linewidth]{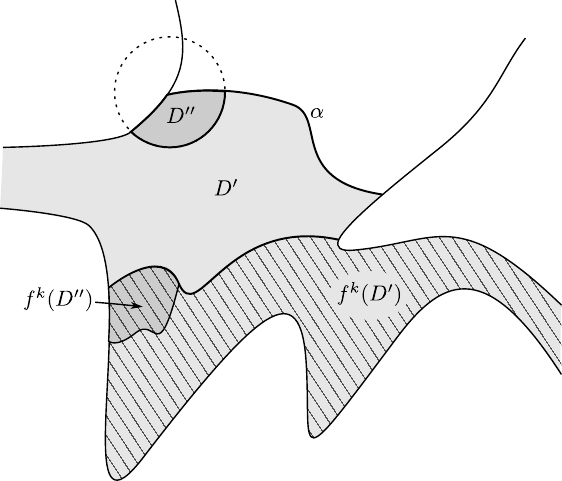}
  \caption{}
    \label{fig:prime-end}
\end{minipage}
\end{figure}

\subsection{Cross-cuts, cross-sections}\label{sec:cross-cuts}

Let $U\subset S$ be an open topological disk. 
A \emph{cross-cut} of $U$ is an arc $\sigma\colon (0,1)\to U$ such that $\lim_{t\to 0^+}\sigma(t)\in \bd U$ and $\lim_{t\to 1^-}\sigma(t)\in \bd U$ with the additional property that both connected components of $U\sm \sigma$ have more than one boundary point in $\bd U$ (this condition amounts to excluding the possibility of $\sigma$ being a closed loop bounding a disk entirely contained in $U$). For instance, any arc in $U$ joining two distinct points of $\bd U$ (not including the endpoints) is a cross-cut.
A \emph{cross-section} of $U$ is any one of the two connected components of $U\sm \sigma$. 

Note that the notion of cross-cut depends on the ambient surface: $\sigma$ is a cross-cut of $U$ as a subset of $S$, but it is no longer a cross-cut of $U$ in the surface $S\sm \{p,q\}$, where $p,q$ are the endpoints of $\sigma$.
We will often make use of the following simple observation: if $U\subset S$ is an open topological disk and $\hat{U}\subset \hat{S}$ is a connected component of $\pi^{-1}(U)$, then any cross-cut of $\hat{U}$ projects to a cross-cut of $U$. In particular if $U$ is invariant by a homeomorphism $f\colon S\to S$ and has no wandering cross-cuts, and if $\hat{f}$ is a lift that leaves $\hat{U}$ invariant, then $\hat{U}$ has no wandering cross-cuts by $\hat{f}$.

Note that any cross-section contains cross-cuts, and so whenever there is a wandering cross-section there is also a wandering cross-cut. We state the following partial converse of that fact for future reference:

\begin{proposition}\label{pro:nwcc-nwcs} Let $f\colon S\to S$ be an orientation-preserving homeomorphism with an invariant open topological disk $U$. If there exists a wandering cross-cut of $U$ which has a non-periodic endpoint, then there exists a wandering cross-section. 
\end{proposition}
\begin{proof} Let $\alpha$ be a wandering cross-cut with some non-periodic endpoint, and let $D$ be the component of $U\sm \alpha$ (cross-section) which contains $f(\alpha)$, so that $f(D)\cap D\neq \emptyset$, and let $D'$ the remaining component. If $D'$ is wandering, we are done. Otherwise, let $k>0$ be the smallest integer such that $f^k(D')\cap D'\neq \emptyset$. Since $f^k(\bd_U D')$ is disjoint from $\bd_U D'$, it follows that either $D'\subset f^k(D')$, or $f^k(D')\subset D'$, or $D\subset f^k(D')$. 
The latter case is not possible: if $D\subset f^k(D')$, then since $f(D)\cap D\neq \emptyset$ we have $f^k(D')\cap f(f^k(D'))\neq \emptyset$, and therefore $D'\cap f(D')\neq \emptyset$. This implies that $k=1$; but then $D\subset f(D')$, and so $f(\alpha)=\bd_U f(D')$ is disjoint from $D$, contradicting our assumption. Thus, the only possibilities are $f^k(D')\subset D'$ or $D'\subset f^k(D')$. Assume $f^k(D')\subset D'$ (the other case is analogous). Then $W=D'\sm \ol{f^k(D')}$ is a wandering open set, and if $x\in \bd U$ is an endpoint of $\alpha$ which is not periodic, a small enough disk $B$ around $x$ intersects $W$ and is disjoint $\ol{f^k(D')}$ from which it is easy to construct a cross-section $D''$ contained in $W$ (see Figure \ref{fig:prime-end}, or see \cite{kln} for more details).
\end{proof}

\subsection{A lemma from Brouwer theory}

We will use the following classical result from Brouwer theory; see for instance \cite{brown,franks-gen}.
\begin{lemma}\label{lem:brouwer-free} If $f\colon \R^2\to \R^2$ is an orientation preserving homeomorphism without fixed points and $\gamma$ is a compact arc such that $f(\gamma)\cap\gamma=\emptyset$, then $f^n(\gamma)\cap \gamma=\emptyset$ for all $n\neq 0$.
\end{lemma}

\section{Linking lemmas}
\label{sec:linking}

In this section we recall some facts about dynamically transverse foliations introduced by Le Calvez \cite{lecalvez-equivariant} and a result of Jaulent \cite{jaulent} which allows us to apply those results in our setting. Then we define linking numbers of open invariant disks with respects to singularities of the foliation and show a linking lemma which plays a key role in the proof of our main result. Similar results were used in \cite{kt-strictly} and \cite{kt-pseudo}, but we need more refined versions here.

\subsection{Oriented foliations, limit sets}

If $\Gamma$ is a regular orbit of a topological flow on $\SS^2$, the \emph{filled} $\omega$-limit of $\Gamma$ is the set $\ol{\omega}(\Gamma)$ consisting of the union of $\omega(\Gamma)$ with all connected components of the complement of $\omega(\Gamma)$ which are disjoint from $\Gamma$. Note that $\omega(\Gamma)$ is always compact and invariant by the flow. 
The \emph{filled} $\alpha$-limit $\ol{\alpha}(\Gamma)$ is defined analogously, replacing $\omega$ by $\alpha$ in the definition.
The filled $\alpha$ and $\omega$-limits are always connected and non-separating.

An oriented foliation $\mc{F}$ with singularities on $S$ consists of a closed set $X\subset S$ together with a nonsingular oriented foliation on $S\sm X$. Any such foliation coincides with the orbits of a topological flow on $S$ whose singularities are precisely the elements of $X$ (see, for instance \cite{flow-foliation}).
Thus we may define the $\omega$ and $\alpha$ limit sets of leaves of the foliation and their filled versions using the corresponding flow.  

If $\Gamma$ is an orbit of a topological flow on $\R^2$, we may define its $\omega$ and $\alpha$-limit sets and their filled counterparts by considering the flow induced on the one-point compactification $\R^2\sqcup\{\infty\}\simeq \SS^2$ and removing the point $\infty$ from the corresponding sets. Note that in this case, the filled $\alpha$ or $\omega$-limit sets may fail to be non-separating (in the case that they contain $\infty$ after the compactification), but if either of the sets is compact then it will be non-separating as well.

An arc $\gamma$ is \emph{positively transverse} to an orbit $\Gamma$ of a topological flow (or a leaf of an oriented foliation) if $t\mapsto \gamma(t)$ locally crosses $\Gamma$ from left to right. If $\gamma$ is topologically transverse to every orbit of the flow (or every leaf of the foliation) that it intersects, we say that $\gamma$ is positively transverse to the flow (or foliation).

\subsection{Brouwer-Le Calvez Foliations}
Let $S$ be an orientable surface (not necessarily compact), and  $\pi\colon \hat{S}\to S$ the universal covering of $S$. Let $\mc{I}=(f_t)_{t\in [0,1]}$ be an isotopy from $f_0=\id_S$ to some homeomorphism $f_1=f$. The \emph{natural lift} of $\mc{I}$ is the isotopy $\hat{\mc{I}}=(\hat{f}_t)_{t\in [0,1]}$ obtained by lifting the isotopy $\mc{I}$ starting with $\hat{f}_0=\id_{\hat{S}}$. The map $\hat{f}=\hat{f}_1$ is a lift of $\hat{f}$ which commutes with every deck transformation. Any such map is called a \emph{natural lift} of $f$.

Let $X= \fix(\mc{I}):= \{x\in S: f_t(x)=x\text{ for all } t\in [0,1]\}$ be the fixed point set of the isotopy, and assume $\mc{F}$ is an oriented foliation of $S$ with singularities at $X$.
We say that the isotopy $\mc{I}$ is \emph{dynamically transverse} to $\mc{F}$ if for each $x\in S$, the arc $(f_t(x))_{x\in [0,1]}$ is homotopic with fixed endpoints in $S\sm X$, to an arc that is positively transverse to $\mc{F}$. In this case, it is also said that $\mc{F}$ is dynamically transverse to $\mc{I}$, and we call $(\mc{I}, \mc{F})$ a 
\emph{Brouwer-Le Calvez pair} for $f$. For further details see \cite{lecalvez-equivariant}.

If $\hat{X}=\pi^{-1}(X)$, then the isotopy $\hat{\mc{I}}$ fixes $\hat{X}$ pointwise. If $\mc{F}$ is dynamically transverse to $\mc{I}$, then the lifted foliation $\hat{\mc{F}}$ (with singularities in $\hat{X}$) of $\hat{S}$ is also dynamically transverse to $\hat{\mc{I}}$. 
Thus, $(\hat{\mc{I}}, \hat{\mc{F}})$ is a Brouwer-Le Calvez pair for $\hat{f}$, which we call the natural lift of $(\mc{I}, \mc{F})$ to the universal covering of $S$.

When $S\sm \fix(\mc{I})$ is connected, the dynamically transverse foliation $\mc{F}$ has the property that if $\til{\mc{F}}$ is its lift to the universal covering of $S\sm \fix(\mc{I})$ and $(\til{f}_t)_{t\in [0,1]}$ is the natural lift of the isotopy $\mc{I}|_{S\sm \fix(\mc{I})}$ then $\til{\mc{F}}$ is a foliation by \emph{Brouwer lines} for $\til{f}=\til{f}_1$. This means that each leaf $\til{\Gamma}$ of $\til{\mc{F}}$ is a properly embedded line which separates the covering space into two connected components: the \emph{left side}, which contains $\til{f}^{-1}(\Gamma)$, and the \emph{right side} which contains $\til{f}(\Gamma)$. Note that $\til{f}$ is a \emph{Brouwer homeomorphism}, \ie it preserves orientation and has no fixed points.

We will use the following result from \cite{jaulent} (stated in a more general way there):
\begin{proposition}\label{pro:jaulent} If $S$ is an orientable surface (not necessarily compact) and $f\colon S\to S$ is a homeomorphism isotopic to the identity and $\hat{f}\colon \hat{S}\to \hat{S}$ is a natural lift of $f$, then there exists a closed set $X_*\subset \fix(f)$ and a Brouwer-Le Calvez pair $((f_t^*)_{t\in [0,1]}, \mc{F}_*)$ associated to $f|_{S\sm X_*}$ in $S\sm X_*$ such that $\mc{F}_*$ has no singularities, and for each $z\in S\sm X_*$ the lift of $(f_t^*(z))_{t\in[0,1]}$ with base point $\hat{z}$ has endpoint $\hat{f}(\hat{z})$. 
\end{proposition}

\begin{remark}It would be useful if the isotopy fixed $X_*$ pointwise, so it would provide a Brouwer-Le Calvez pair for $f$; this is possible using a recently announced result by F. B\'eguin, S. Crovisier and F. Le Roux which improves Jaulent's result, but we will not make use of that fact, since we will ensure that $X_*$ is totally disconnected (in which case the isotopy extends naturally).
Note also that the last part of Proposition \ref{pro:jaulent} implies that the Brouwer-Le Calvez pair may be lifted to a Brouwer-Le Calvez pair in $\hat{S}\sm \pi^{-1}(X_*)$ associated to $\hat{f}|_{\hat{S}\sm \pi^{-1}(X_*)}$. 
\end{remark}

We will need the following technical fact, which says that in a Brouwer-Le Calvez pair, if a point $x$ is close enough to some leaf of the foliation then the arc obtained by following the isotopy from the preimage of $x$ to the image of $x$ necessarily crosses the given leaf.

\begin{proposition}\label{pro:perto-corta} If $S$ is an orientable surface and $(\mc{I}, \mc{F})$ is a Brouwer-Le Calvez pair with $\mc{I}=(f_t)_{t\in [0,1]}$ and $f=f_1$, then every leaf $\Gamma$ of $\mc{F}$ has a neighborhood $V_{\Gamma}$ with the property that if $z\in V_{\Gamma}$ and $\gamma$ is any arc joining $f^{-1}(z)$ to $f(z)$  homotopic with fixed endpoints in $S\sm \fix(\mc{I})$ to the concatenation of $(f_t(f^{-1}(z)))_{t\in [0,1]}$ with $(f_t(z))_{t\in [0,1]}$, then $\gamma$ intersects $\Gamma$.
\end{proposition}
\begin{proof} Let $X=\fix(\mc{I})$. It suffices to prove the proposition when $N=S\sm X$ is connected, since otherwise we may work with the connected component of $S\sm X$ containing $\Gamma$. Let $\tau\colon \til{N}\to N$ be the universal covering of $N$. Then as we noted earlier, the natural lift $((\til{f}_t)_{t\in [0,1]}, \til{\mc{F}})$ of the Brouwer-Le Calvez pair restricted to $N$ has the property that every leaf of $\til{\mc{F}}$ is a Brouwer line. In particular if  $\til{\Gamma}$ is a lift of $\Gamma$, then $\til{\Gamma}$ separates $\til{f}_1(\til{\Gamma})$ from $\til{f}_1^{-1}(\til{\Gamma})$. This implies that there exists a neighborhood $\til{V}_{\til{\Gamma}}$ of $\til{\Gamma}$ such that $\til{f}_1(\til{V}_{\til{\Gamma}})$ and $\til{f}_1^{-1}(\til{V}_{\til{\Gamma}})$ are separated by $\til{\Gamma}$. See Figure \ref{fig:Vgamma}. Letting $V_\Gamma=\tau(\til{V}_{\til{\Gamma}})$, if $z\in V_\Gamma$ and $\gamma$ is an arc as in the statement of the proposition, then $\gamma$ lifts to some arc $\til{\gamma}$ joining a point $\til{f}^{-1}(\til{z})\in \til{f}^{-1}(\til{V}_{\til{\Gamma}})$ to $\til{f}(\til{z})\in \til{f}(\til{V}_{\til{\Gamma}})$, so $\til{\gamma}$ must intersect $\til{\Gamma}$. This implies that $\gamma$ intersects $\tau(\til{\Gamma})=\Gamma$ as required.
\end{proof}
\begin{figure}[ht]
\includegraphics[width=0.8\linewidth]{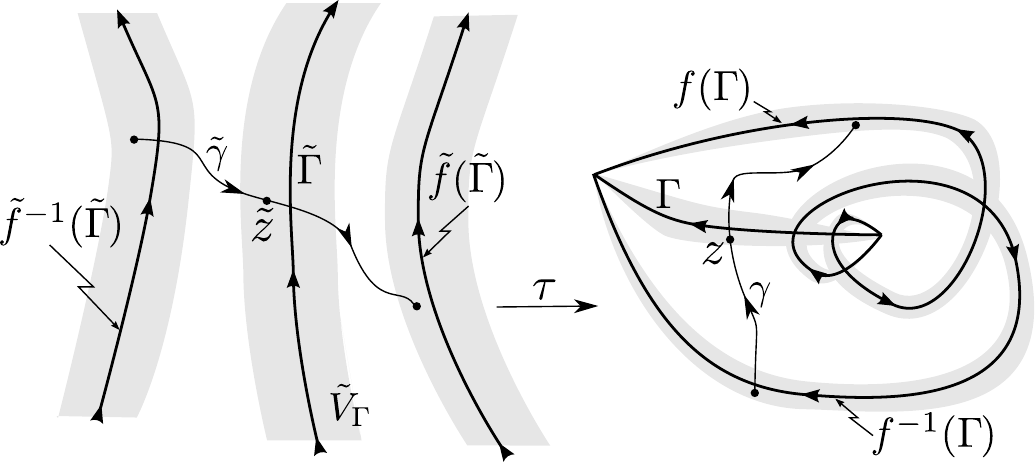}
\caption{Proof of Proposition \ref{pro:perto-corta}.}
\label{fig:Vgamma}
\end{figure}

\subsection{Linking numbers}

If $\gamma$ is a loop in $\R^2$ and $p$ is a point in the complement of $\gamma$, we write $\wind(\gamma, p)$ for the winding number of $\gamma$ around $p$.
If $\gamma$ is a loop in $\SS^2$ and $p,q$ are points in its complement, we write $\wind(\gamma, p, q)$ for the winding number of $\gamma$ around $p$ in the plane $\SS^2\sm \{q\}$, or equivalently, the algebraic intersection number $\sigma\wedge \gamma$, where $\sigma$ is any arc joining $p$ to $q$. Note that $\wind(\gamma, p, p) = 0$, and if $r\in \SS^2$ is in the complement of $\gamma$, then $\wind(\gamma, p, q) = \wind(\gamma, p, r)+\wind(\gamma, r, q)$. 

We state the following lemma for future reference. The proof is straightforward (see, for instance, \cite[Proposition 9]{tal2}).

\begin{lemma}\label{lem:loop-trap} If $\gamma$ is a loop positively transverse to an orbit $\Gamma$ of a topological flow of $\SS^2$ intersecting the orbit at least once, then $\ol{\omega}(\Gamma)$ and $\ol{\alpha}(\Gamma)$ belong to different connected components of $\SS^2\sm \gamma$. Moreover, there is $k\neq 0$ such that $\wind(\gamma, p, q) = k$ for each $p\in \ol{\omega}(\Gamma)$ and $q\in \ol{\alpha}(\Gamma)$.
\end{lemma}

For the remainder of this subsection, fix a Brouwer-Le Calvez pair $(\mc{I}, \mc{F})$ of an orientation preserving homeomorphism $f$ of the sphere $\SS^2$, where $\mc{I}=(f_t)_{t\in [0,1]}$ and let $X=\fix(\mc{I})$. 

If $p,q \in X$ and $z\in \fix(f)\sm \{p,q\}$, we define the \emph{linking number} $L_{\mc{I}}(z,p,q) = \wind(\gamma_z, p, q)$ where $\gamma_z$ is the loop $(f_t(z))_{t\in [0,1]}$. More generally, if $U\subset \SS^2$ is an $f$-invariant open topological disk and $p,q\in X\sm U$, we define the linking number $L_{\mc{I}}(U, p, q)$ as follows: choose any $z\in U$, any arc $\sigma$ in $U$ joining $f(z)$ to $z$, and let $\gamma$ be the loop obtained by concatenation of the arc $(f_t(z))_{t\in [0,1]}$ with $\sigma$. Then set
 $$L_{\mc{I}}(U, p, q) = \wind(\gamma, p, q).$$
It is easy to see that this number does not depend on the choice of either $z$ or $\sigma$ (see \cite[\S 5.3]{kt-strictly}).

If $r\notin U$ is another point of $X$, the properties of winding number lead to 
$$L_{\mc{I}}(U, p, q)=L_{\mc{I}}(U, p, r) + L_{\mc{I}}(U, r, q).$$

We may also define the linking for fixed points as follows: if $r\in\fix(f)$, then $L_{\mc{I}}(r, p, q)=\wind(\gamma, p, q)$, where $\gamma=(f_t(r))_{t\in [0,1]}$. Clearly, if $r\in U$ then $L_{\mc{I}}(r, p, q)=L_{\mc{I}}(U, p, q).$
In particular, we note that if $U\cap X\neq \emptyset$ then $L_{\mc{I}}(U, p, q)=0$.

Another useful property that follows from the definition is that if $\mc{I}^n$ is the isotopy from $\id$ to $f^n$ obtained by successive concatenation of $\mc{I}$, then $L_{\mc{I}^n}(U, p, q)  = nL_{\mc{I}}(U, p,q)$.

Note that for any leaf $\Gamma$ of $\mc{F}$ each set $\ol{\alpha}(\Gamma)$ and $\ol{\omega}(\Gamma)$ contains at least one point of $X$ due to the Poincar\'e-Bendixson theorem.
\begin{lemma}[Linking lemma]\label{lem:linking-sphere} Let $\Gamma$ be a leaf of $\mc{F}$, and assume that $V_\Gamma$ is a neighborhood of $\Gamma$ as in the statement of \ref{pro:perto-corta}. Then, for any $f$-invariant open topological disk  $U\subset \SS^2$ intersecting $V_\Gamma$, one of the following properties hold:
\begin{itemize}
\item[(1)] $\ol{\alpha}(\Gamma)\cap X\subset U$;
\item[(2)] $\ol{\omega}(\Gamma)\cap X\subset U$;
\item[(3)] 
\begin{itemize}
\item[(a)] $\ol{\alpha}(\Gamma)\cap X$ and $\ol{\omega}(\Gamma)\cap X$ are disjoint from $U$ and there is $k\neq 0$ such that $L(U, p,q) = k$ for all $p\in \ol{\alpha}(\Gamma)\cap X$ and $q\in \ol{\omega}(\Gamma)\cap X$, or
\item[(b)] $\Gamma\cap U\neq \emptyset$, and each connected component of $U\cap \Gamma$  is wandering. 
\end{itemize}
\end{itemize}
\end{lemma}
Whenever property (3a) above holds we say that $U$ is \emph{linked} to $\Gamma$. Thus, either all singularities in the filled $\omega$-limit set of $\Gamma$ are in $U$, or a similar property holds for the $\alpha$-limit set, or $\Gamma$ intersects $U$ and all connected components of this intersection are wandering, or $U$ is linked to $\Gamma$.

\begin{proof}
Assume that neither (1) nor (2) hold, so that there exist $p\in \ol{\alpha}(\Gamma)\cap X\sm U$ and $q\in \ol{\omega}(\Gamma)\cap X\sm U$. 
For $z\in S\sm X$, let $\gamma_z$ denote the arc $(f_t(z))_{t\in [0,1]}$, and let $\til{\gamma}_z$ denote an arc positively transverse to the foliation $\mc{F}$ and homotopic with fixed endpoints in $S\sm X$ to $\gamma_z$. Given $z\in U\cap V_\Gamma$, we know from the property of $V_\Gamma$ that the arc $\til{\gamma}_{f^{-1}(z)}*\til{\gamma}_z$ intersects $\Gamma$ (see Proposition \ref{pro:perto-corta}). Suppose first that $\Gamma$ does not intersect $U$. Then we may choose an arbitrary arc $\sigma$ in $U$ connecting $f(z)$ to $f^{-1}(z)$ and we have that $\beta = \til{\gamma}_{f^{-1}(z)}*\til{\gamma}_z*\sigma$ is a loop positively transverse to $\Gamma$ and has at least one intersection with $\Gamma$. From Lemma \ref{lem:loop-trap} we conclude that $\wind(\beta,p,q):=k'\neq 0$. But $\beta$ is homotopic to $\gamma_{f^{-1}(z)}*\gamma_z*\sigma$ in $\mathbb{S}^2\sm X$, and therefore $L_{\mc{I}^2}(U, p,q) = \wind(\gamma_{f^{-1}(z)}*\gamma_z*\sigma, p,q) = \wind(\beta,p,q) = k'\neq 0$. Thus $L_{\mc{I}}(U, p, q)= k'/2:=k\neq 0$. This also implies that no element of $X$ can be in $U$, as that would mean that $L_{\mc{I}}(U, p, q)=0$. Since the number $k'$ is independent on the choice of $p$ and $q$, and we showed that no element of $X$ may be in $U$, we conclude that (3a) holds.

Now suppose that $\Gamma$ intersects $U$, and assume that (3b) does not hold, so some component $C$ of $\Gamma\cap U$  is nonwandering. Thus there is $n>0$ and some point $z\in C$ such that $f^n(z)\in C$. Let $\sigma$ be the subarc of $C$ connecting $f^n(z)$ to $z$, and define $\beta = \til{\gamma}_z*\til{\gamma}_{f(z)}*\cdots*\til{\gamma}_{f^{n-1}(z)}*\sigma$. The only part of $\beta$ that is not positively transverse to $\mc{F}$ is $\sigma$, which is contained in a leaf; but a small perturbation in a tubular neighborhood of $\mc{F}$ around $C$ leads to a loop $\beta'$ which is homotopic to $\beta$ in $\mathbb{S}^2\sm X$, has the same basepoint $z\in C$, and is positively transverse to $\mc{F}$. Applying Lemma \ref{lem:loop-trap} we conclude that $\wind(\beta,p,q)=k\neq 0$, and since $\beta$ is homotopic to $\gamma_z*\gamma_{f(z)}*\cdots*\gamma_{f^{n-1}(z)}*\sigma$ in $\mathbb{S}^2\sm X$, we deduce that $L_{\mc{I}^n}(U, p,q) = \wind(\beta,p,q) :=k'\neq 0$. Thus $L_{\mc{I}}(U, p, q) = k:= k'/n\neq 0$, and as before this implies that $U$ is disjoint from $X$; since the number $k'$ is independent on the choice of $p,q$, we conclude again that (3a) holds.
\end{proof}

\subsection{Linking in the plane}

Assume now that $f\colon \R^2\to \R^2$ is an orientation-preserving homeomorphism and $(\mc{I}, \mc{F})$ is a Brouwer-Le Calvez pair for $f$, with $X=\fix(\mc{I})$ totally disconnected. We may consider the map $f'\colon \SS^2 = \R^2\sqcup\{\infty\}\to \SS^2$ induced by $f$ in the one-point compactification of $\R^2$ by fixing $\infty$, and the corresponding isotopy $\mc{I}'$ fixing $\infty$ pointwise. The foliation $\mc{F}'$ of $\SS^2$ with singularities at $X\sqcup\{\infty\}$ is then dynamically transverse, so that $(\mc{I}', \mc{F}')$ is a Brouwer-Le Calvez pair.

If $p\in X$ and $z\in \fix(f)\sm p$, we may define the linking number $L_{\mc{I}}(z,p)$ as the corresponding liking number $L_{\mc{I}'}(z,p,\infty)$ on $\mathbb{S}^2$, and similarly if $U\subset \R^2$ is an open $f$-invariant topological disk and $p\in X\sm U$, we define $L_{\mc{I}}(U, p)$ as the corresponding linking number $L_{\mc{I}'}(U, p,\infty)$ on $\mathbb{S}^2$.
In this setting we have the following:

\begin{lemma}[Linking lemma, planar version]\label{lem:linking-plane} Let $\Gamma$ be a leaf of $\mc{F}$ and $V_\Gamma$ a neighborhood of $\Gamma$ as in Proposition \ref{pro:perto-corta}. Suppose that $U\subset \R^2$ is an $f$-invariant open topological disk intersecting $V_\Gamma$. Then one of the following properties hold:
\begin{itemize}
\item[(1)] $\ol{\alpha}(\Gamma)$ is compact, and either
\begin{itemize}
\item[(a)] $\emptyset\neq \ol{\alpha}(\Gamma)\cap X\subset U$, or
\item[(b)] $\ol{\alpha}(\Gamma)\cap X$ is disjoint from $U$, and there is $k\neq 0$ such that $L_{\mc{I}}(U, p) = k$ for all $p\in \ol{\alpha}(\Gamma)\cap X$.
\end{itemize}
\item[(2)] $\ol{\omega}(\Gamma)$ is compact, and either
\begin{itemize}
\item[(a)] $\emptyset\neq \ol{\omega}(\Gamma)\cap X\subset U$, or
\item[(b)] $\ol{\omega}(\Gamma)\cap X$ is disjoint from $U$, and there is $k\neq 0$ such that $L_{\mc{I}}(U, p) = k$ for all $p\in \ol{\omega}(\Gamma)\cap X$.
\end{itemize}
\item[(3)] $\Gamma\cap U\neq \emptyset$ and each of its connected components is wandering.
\end{itemize}
\end{lemma}
\begin{proof}
Applying Lemma \ref{lem:linking-sphere} to the one-point compactification of $\R^2$, since $\infty$ does not belong to $U$, if cases (1) or (2) from said proposition hold then (1a) or (2a) above hold. If (3a) from Lemma \ref{lem:linking-sphere} holds, noting that
$$0\neq k=L_{\mc{I}'}(U, p,q) = L_{\mc{I}'}(U, p, \infty) + L_{\mc{I}'}(U, \infty, q)$$
for any $p\in \ol{\alpha}(\Gamma)\cap X$ and $q\in \ol{\omega}(\Gamma)\cap X$, we have that one of the two terms on the right hand side is nonzero. Assume for instance that $L_{\mc{I}'}(U, p, \infty) = k'\neq 0$. Then replacing $p$ by a different point of $\ol{\alpha}(\Gamma)\cap X$ does not change this value, so that $L_{\mc{I}}(U,p) = k'$ for all $p\in \ol{\alpha}(\Gamma)\cap X$, and in particular $\infty$ does not belong to that set (so the compactness of $\ol{\alpha}(\Gamma)$ follows at once). Hence, (1b) holds.

In the case that $L_{\mc{I}'}(U, \infty, q) = k'\neq 0$, we have $L_{\mc{I}}(U, q) = L_{\mc{I}'}(U, q, \infty) = -k'\neq 0$ and this number does not depend on the choice of $q\in \ol{\omega}(\Gamma)$, so (2b) holds.
\end{proof}

\section{Boundedness of invariant disks: proof of Theorem \ref{th:main-disk}}
\label{sec:main-disk}

\subsection{Fixed points in $U$}

We begin by noting that the condition of having no wandering cross-cuts implies that there is a fixed point in $U$.
\begin{proposition}\label{pro:nwcc-fix} Let $f\colon S\to S$ be an orientation-preserving homeomorphism of an orientable surface $S$ (not necessarily closed), and let $U\subset S$ be an $f$-invariant open topological disk without wandering cross-cuts such that $\bd U \sm \fix(f)\neq \emptyset$. Then $f$ has a fixed point in $U$.
\end{proposition}
\begin{remark} The example in Figure \ref{fig:example-nofix} shows that the condition that $\bd U$ has at least one non-fixed point is necessary.
\end{remark}
\begin{figure}[ht]
\includegraphics[height=4cm]{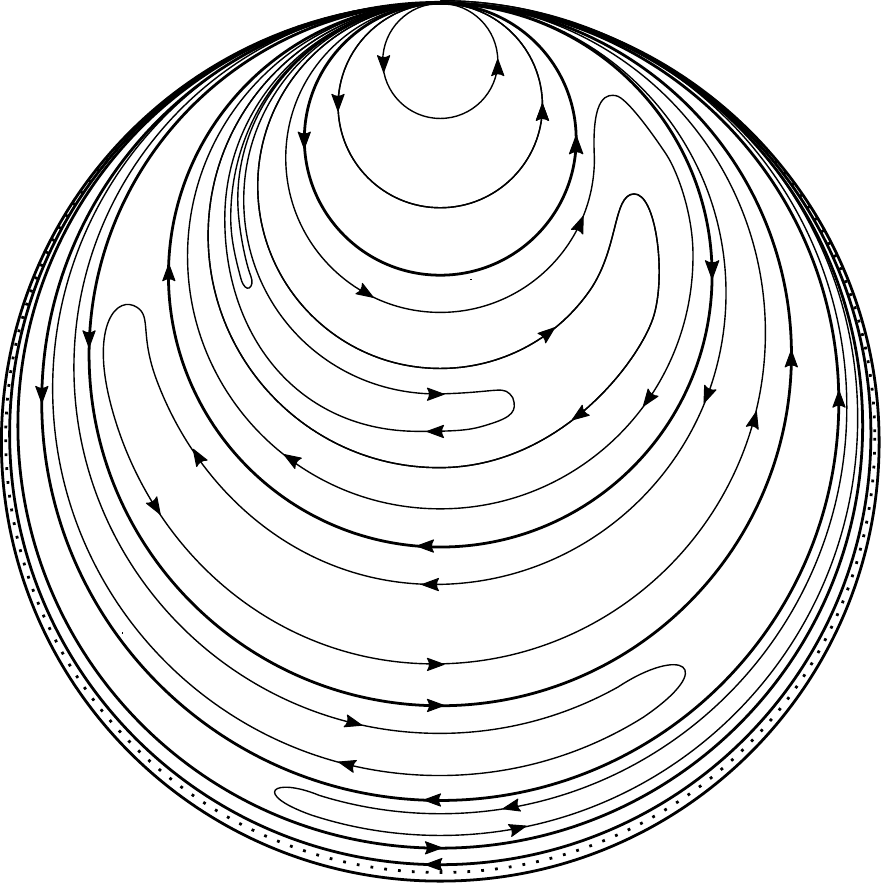}
\caption{Example with no wandering cross-cuts and no fixed points in the disk. The boundary of the disk is fixed 
pointwise.}
\label{fig:example-nofix}
\end{figure}
\begin{proof}
Intersecting $U$ with a small circle around a non-fixed point $z\in \bd U$, we may find a cross-cut $C$ of $U$ such that $f(C)\cap C=\emptyset$. From our hypothesis, $C$ must be nonwandering, so there is a smallest $n>1$ such that $f^n(C)\cap C\neq \emptyset$. Choosing a compact subset $C'$ of $C$ such that $f^n(C')\cap C'\neq \emptyset$ and choosing an appropriate neighborhood $D$ of $C'$ (contained in $U$) we obtain a free disk $D$ such that $f^n(D)\cap D\neq \emptyset$. This implies that $f|_U$ has a fixed point due to a classical lemma of Franks (see \cite[Proposition 1.3]{franks-gen}). 
\end{proof}

\subsection{A bound on the fixed points in $U$}\label{sec:fix-collapse}
In this subsection, $S$ will always denote a closed orientable surface with $\chi(S)\leq 0$.

To prove Theorem \ref{th:main-disk}, we will need to guarantee that a certain set of fixed points is totally disconnected. If the set in question is compact and inessential, we may do that by ``collapsing'' the filling of the given set, with the aid of the next proposition, which is stated in \cite[Proposition 2.6]{kt-strictly} for the case $S=\T^2$, but the exact same proof applies for an arbitrary surface:
\begin{proposition}\label{pro:collapse} Let $K$ be a compact inessential filled set and $f\colon S\to S$ a homeomorphism such that $f(K)=K$. Then there is a continuous surjection $h\colon S\to S$ and a homeomorphism $f'\colon S \to S$ such that:
\begin{itemize}
\item $h$ is homotopic to the identity;
\item $hf = f'h$;
\item $K' = h(K)$ is totally disconnected;
\item $h|_{S\sm K}\colon S\sm K \to S\sm K'$ is a homeomorphism.
\end{itemize}
\end{proposition}

In general, if $f$ has arbitrarily large invariant open topological disks, we expect that the same would be true for $f'$; but this could fail to be the case if the set $K$ separates those disks into many uniformly bounded ones. The next proposition will allow us to address this issue in the proof of Theorem \ref{th:main-disk}.

\begin{proposition} \label{pro:slice} For any orientation preserving homeomorphism $f\colon S\to S$ with an inessential set of fixed points, there exists a constant $M_0$ such that if $U$ is any open $f$-invariant topological disk without wandering cross-cuts, then:
\begin{itemize} 
\item The covering distance in $U$ of any pair of fixed points is bounded by $M_0$. In other words, if $\hat{U}$ is a connected component of $\pi^{-1}(U)$, then $\pi^{-1}(\fix(f))\cap \hat{U}$ has diameter at most $M_0$. 
\item If $M>2M_0$ is such that $\diamup(U)>M$, then there exists a (necessarily invariant) connected component $U_0$ of $U\sm \fix(f)$ such that $\diamup(U_0)>M/2 - M_0$. 
\end{itemize}
\end{proposition}

\begin{proof}
The fact that $\fix(f)$ is inessential implies that we can find an closed topological disk $Q_0\subset \hat{S}$ such that $\pi(Q_0)=S$ and $\bd Q_0$ is the image of a simple rectifiable loop $\hat{\gamma}$ such that $\gamma:=\pi(\hat{\gamma})$ is disjoint from $\fix(f)$ (see Remark \ref{rem:Q}). Using $Q_0$, we can obtain a sequence $(Q_n)_{n\in \N}$ of open topological disks such that $Q_n$ is contained in the interior of $Q_{n+1}$,  $\pi(\bd Q_n)\subset \gamma$, and $\bigcup_{n\in \Z} Q_n = \hat{S}$. Indeed, one may for instance define recursively $Q_{n+1}$ as the filling of the union of all sets of the form $TQ_0$ which intersect the neighborhood of size $1$ of $Q_n$ (since $Q_n$ is compact, there are finitely many such sets). Each point in the boundary of $Q_{n+1}$ is contained in $T\hat{\gamma}$ for some $T\in \deck(\pi)$, so $\pi(\bd Q_{n+1})\subset \gamma$.

Note that if $C\subset U$ is the image of a cross-cut which separates two fixed points of $f$ in $U$, then $f(C)$ intersects $C$. Indeed, suppose that $f(C)$ is disjoint from $C$. The set $U\sm C$ has two connected components $D_1, D_2$, and we assume $D_1$ is the component containing $f(C)$. Note that $f(D_1)$ intersects $D_1$ and does not contain $D_2$ (because there is a fixed point in $D_2$). This means that $f(D_1)$ is disjoint from $D_2$, since otherwise the boundary of $f(D_1)$ in $U$ (which is $f(C)$) would contain a point in $D_2$. Thus $f(D_1\cup C) \subset D_1$. This is a contradiction, as it implies that $C$ is a wandering cross-cut.

Since $\gamma$ has finite length and no point of $\gamma$ is fixed by $f$, there exists a constant $c>0$ (independent of $U$) such that every sub-arc of $\gamma$ with diameter smaller than $c$ is free (disjoint from its own image). By our previous observation, this implies that any cross-cut of $U$ contained in $\gamma$ and separating the set of fixed points of $f$ in $U$ has lenght least $c$. In particular, there cannot be more than $N = \text{length}(\gamma)/c$ pairwise disjoint cross-cuts of $U$ contained in $\gamma$, each separating the fixed point set of $f$ in $U$ (note that $N$ is independent of $U$ as well).

If $U$ contains no fixed points, there is nothing to be done. Now assume $U$ has a fixed point, and let $\hat{U}$ be a connected component of $\pi^{-1}(U)$ such that $\hat{U}\cap Q_0$ contains some element $z_0 \in \pi^{-1}(\fix(f))$. We claim that $\pi^{-1}(\fix(f))\cap \hat{U}\subset Q_{N+1}$. Indeed, for a given $n\in \N$, suppose there is $z_1\in \hat{U}\sm Q_n$ such that $\pi(z_1)$ is fixed by $f$. For each $i\in \{0,\dots, n-1\}$, The fact that $z_0 \in Q_i$ and $z_1\notin Q_i$ implies that there is a cross-cut $\hat{\alpha}_i$ of $\hat{U}$ contained in $\bd Q_i$ which separates $z_0$ from $z_1$ (See Figure \ref{fig:prop33}). Since $\pi|_{\hat{U}}$ is injective, it follows that $\alpha_i=\pi(\hat{\alpha}_i)$ is a cross-cut of $U$ contained in $\gamma$ and separating the fixed point $\pi(z_0)$ from $\pi(z_1)$. Moreover, since the $\hat{\alpha}_i$'s are pairwise disjoint and contained in $\hat U$, the set $\{\alpha_i: 0\leq i\leq n-1\}$ is a family of $n$ pairwise disjoint cross-cuts of $U$ contained in $\gamma$, each of which separates the fixed point set of $f$ in $U$. By our choice of $N$, this implies that $n\leq N$. Thus, $\pi^{-1}(\fix(f))\cap \hat{U}\subset Q_{N+1}$ as claimed.

\begin{figure}
\includegraphics[width=0.8\linewidth]{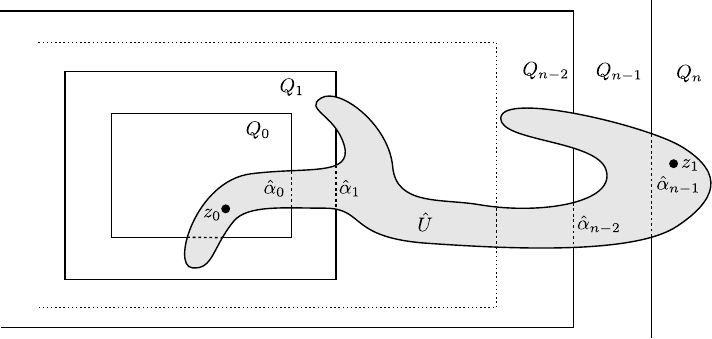}
\caption{The cross-cuts $\hat{\alpha}_i$.}
\label{fig:prop33}
\end{figure}

Note that this implies that $\diam(\pi^{-1}(\fix(f))\cap \hat{U})\leq M_0:= \diam(Q_{N+1})$, a number independent of $U$. This proves the first claim.

If $\diamup(U)>M>2M_0$, there is a point $z\in \hat{U}$ such that $d(z, z_0) > M/2$, and so $d(z,x)\geq M/2 - M_0$ for any $x\in Q_{N+1}$. Thus the connected component $\hat{U}_0$ of $z$ in $\hat{U}\sm \pi^{-1}(\fix(f))$ has diameter at least $M/2 - M_0$.  Note that $\hat{U}_0$ projects to a connected component $U_0$ of $U\sm \fix(f)$, so the second claim  claim follows (the fact that $U_0$ is invariant follows, for instance, from \cite{brown-kister}). 
\end{proof}

\begin{corollary} \label{coro:slice}
Let $f\colon S\to S$ be an orientation preserving homeomorphism with an inessential set of fixed points, $K_0\subset \fix(f)$ a compact set, and $h$ and $f'$ be the corresponding maps from Proposition \ref{pro:collapse} applied to $K=\fill(K_0)$. Then there is a constant $M_1$ such that, if there is an open $f$-invariant topological disk $U$ without wandering cross-cuts and  $\diamup(U)>M>M_1$, then there is also an open $f'$-invariant topological disk $U'$ such that $\diamup(U')> (M-M_1)/2$ and any cross-cut of $U'$ with endpoints in $\bd U' \sm h(K)$ is nonwandering. Moreover, $U'$ is the filling of the image by $h$ of some connected component of $U\sm K$.
\end{corollary}
\begin{proof}
Let $h$ be the map from Proposition \ref{pro:collapse}, and $M_0$ the constant from Proposition \ref{pro:slice}. 
Since $h$ is homotopic to the identity, there is a constant $M_2$ and a lift $\hat{h}$ of $h$ to the universal covering $\hat{S}$ such that $d(\hat{h}(z), z)<M_2$ for all $z\in \hat{S}$. 

Let $M_1 = 2M_0+4M_2$, and suppose $U$ is an $f$-invariant open topological disk without wandering cross-cuts and $\diamup(U)>M>2M_1$. Let $U_0$ be the corresponding connected component of $U\sm \fix(f)$ obtained from Proposition \ref{pro:slice}, so that $\diamup(U_0)>M/2-M_0$.
From our choice of $M_2$, it is easy to verify that $U_0':=h(U_0)$ satisfies $\diamup(U_0)'\geq \diamup(U_0)-2M_2 > M/2 - M_0 - 2M_2 > 0$. The fact that $U_0'$ has positive diameter implies that $U_0$ is disjoint from $K$, since it is already disjoint from $\fix(f)\supset K_0$ and it cannot be contained in an inessential connected component of $S\sm K_0$, as that would imply that $h(U_0) = U_0'$ is a point. 

Let $U_1$ be the connected component of $U\sm K$ containing $U_0$, and $U_1'=h(U_1)$. Since $h|_{S\sm K}$ is a homeomorphism, $U_1'$ is open, inessential and $f'$-invariant. Thus the set $U'=\fill(U_1')$ is an $f'$-invariant open topological disk with $$\diamup(U') = \diamup(U_1')\geq\diamup(U_0')>M/2-M_0-2M_2 \geq (M-M_1)/2.$$

Finally, suppose for a contradiction that there exists a wandering cross-cut $\alpha'$ of $U'$ whose endpoints are not in $h(K)$. Then $\alpha'$ must be contained $U_1'$, and its endpoints are points of $\bd U_1'\sm h(K)$. Since $h$ maps homeomorphically $S\sm K$ onto $S\sm h(K)$ conjugating the restrictions of $f$ and $f'$ to those sets, it follows that $\alpha:=h^{-1}(\alpha')$ is a wandering arc in $U_1$ with endpoints in $\bd U_1\sm K$. Since $U_1$ is a connected component of $U\sm K$ (so $\bd U_1\subset \bd U \cup K$),  this means that $\alpha$ is contained in $U$ and has endpoints in $\bd U$, so it is a wandering cross-cut for $f$ in $U$, contradicting our hypothesis.
\end{proof}

\subsection{Contractible case}\label{sec:contractible}

We begin with a simple observation: 
\begin{proposition} \label{pro:link-bound}  Let $\mc{I}=(f_t)_{t\in [0,1]}$ be an isotopy from the identity to $f\colon S\to S$, and $\hat{\mc{I}}=(\hat{f}_t)_{t\in [0,1]}$ its natural lift from the identity to $\hat{f}$ in the universal covering. If $K\subset \fix(\hat{\mc{I}})$ is a compact set, then the set $K' = \{z\in \fix(\hat{f}): L_{\hat{\mc{I}}}(z,z')\neq 0\text{ for some }z'\in K\}$ is relatively compact. 
\end{proposition}
In other words, the set of fixed points of $\hat{f}$ linking a given compact set $K$ is relatively compact.
\begin{proof}
Due to the compactness of $S$ and the equivariance of the isotopy $\hat{\mc{I}}$, there exists $M>0$ such that for any $z\in \hat{S}$, the arc $(\hat{f}_t(z))_{t\in [0,1]}$ has diameter at most $M$, so if $z'\in K$ and $z\in\fix(\hat{f})$ is such that $d(z,z')>M$ then $L_{\hat{\mc{I}}}(z,z')=0$.
\end{proof}

\begin{proposition} \label{pro:main-contractible} Let $S$ be a closed orientable surface, and $f\colon S\to S$ a homeomorphism homotopic to the identity such that $\fix(f)$ is inessential. Then there exists $M>0$ such that if $\hat{f}\colon \hat{S}\to \hat{S}$ a natural lift of $f$ and $U\subset \hat{S}$ is any open $\hat{f}$-invariant topological disk which projects injectively into an $f$-invariant disk without wandering cross-cuts, then $\diam(U)\leq M$.
\end{proposition}
\begin{proof}
We begin by noting that it suffices to find a bound $M$ for a specific lift $\hat{f}$. Indeed, if $S$ has genus $g>1$, then there is a unique natural lift (see for instance \cite[Lemma 4.3]{epstein}); whereas if $S$ is the torus there are at most finitely many lifts which have a fixed point, and Proposition \ref{pro:nwcc-fix} implies that any $\hat{f}$ admitting an invariant disk $U$ as in the statement must have a fixed point.

We will prove the following:
\begin{claim}\label{claim:td} If $f$ has a Brouwer-Le Calvez pair $(\mc{I}, \mc{F})$ with totally disconnected set of singularities $X$ whose natural lift corresponds to a Brouwer-Le Calvez pair $(\hat{\mc{I}}, \hat{\mc{F}})$ associated to $\hat{f}$, then there exists $M'>0$ such that the following property holds: if $U$ is a $\hat{f}$-invariant topological disk which projects injectively to an $f$-invariant disk $\pi(U)$, and every wandering cross-cut of $\pi(U)$ has an endpoint in $X$, then $\diam(U)\leq M'$.
\end{claim}

Before moving to the proof of the claim, we show how the proposition follows from it.
\begin{proof}[Proof of Proposition \ref{pro:main-contractible} assuming Claim \ref{claim:td}] 
Under the assumptions of the proposition, let $\mc{I}=(f_t)_{t\in [0,1]}$ be the isotopy from the identity to $f$ which lifts to an isotopy $\hat{\mc{I}}$ from the identity to $\hat{f}$. By Proposition \ref{pro:jaulent} there exists a closed set $X_*\subset \fix(f)$ and a Brouwer-Le Calvez pair $((f_t^*)_{t\in [0,1]}, \mc{F}_*)$ associated to $f|_{S\sm X_*}$ with the additional property that for any $z\in S\sm X_*$ the path $(f_t^*(z))_{t\in [0,1]}$ lifts to a path connecting $z$ to $\hat{f}(z)$. 
Let $h\colon S\to S$ and $f'$ be the maps given by Proposition \ref{pro:collapse} applied to $K = \fill(X_*)$, and let $X'=h(K)$, so $X'$ is totally disconnected. 
The isotopy $((f_t^*)_{t\in [0,1]})$ restricted to $S\sm K$ lifts to an isotopy $(\hat{f}_t^*)_{t\in [0,1]}$ in $\hat{S}\sm \pi^{-1}(K)$ from the identity to $\hat{f}_{\hat{S}\sm \pi^{-1}(K)}$, and the foliation $\mc{F}_*$ lifts to a dynamically transverse foliation $\hat{\mc{F}}_*$ without singularities in $\hat{S}\sm \pi^{-1}(K)$.

Let $\hat{h}$ be the natural lift of $h$ (obtained by lifting the homotopy from the identity to $h$). Then $\hat{f}_t'=\hat{h}\hat{f}_t\hat{h}^{-1}$ defines an isotopy $\hat{\mc{I}}'$ from the identity to some map $\hat{f}'$ in $\hat{S}\sm \hat{X}'$, where $\hat{X}'=\hat{h}(\pi^{-1}(K)) = \pi^{-1}(h(K))$. Since $\hat{X}'$ is closed and totally disconnected, the isotopy $\hat{\mc{I}}'$ (and the map $\hat{f}'$) may be extended to $\hat{S}$ by fixing $\hat{X}'$ pointwise. 
The foliation $\hat{\mc{F}}_*$ also induces a foliation without singularities $\hat{\mc{F}}' = \{h(\Gamma):\Gamma\in \hat{\mc{F}}_*\}$ of $\hat{S}\sm\hat{X}'$, and if we regard it as a foliation with singularities of $\hat{S}\sm \hat{X}'$ it is easy to verify that it is dynamically transverse to $\hat{f}'$. Thus we have a Brouwer-Le Calvez pair for $\hat{f}'$, which projects to a Brouwer-Le Calvez pair for $f'$ with singularities exactly at $X'$.

Let $M_1$ be the constant from Corollary \ref{coro:slice}, and suppose that the proposition does not hold, so for any $M$ we may find an $\hat{f}$-invariant open topological disk $\hat{U}$ which projects to an $f$-invariant disk $U$ without wandering cross-cuts and $\diamup(U)> M$. Assume $M>M_1+2M'$ (where $M'$ comes from Claim \ref{claim:td}), and let $U'$ be the $f'$-invariant disk given by Corollary \ref{coro:slice}, which is the filling of some connected component $U_1$ of the image by $h$ of $U\sm K$ and satisfies $\diamup(U')>(M-M_1)/2>M'$. Clearly $\hat{U}_1$ is $\hat{f}$-invariant, so if $\hat{U}'$ is the connected component of $\pi^{-1}(U')$ containing $\hat{U}_1$ it must also be $\hat{f}$-invariant.
Furthermore, every crosscut $C$ of $U'$ with endpoints in $S'\sm X'$ is the image by $h$ of a crosscut of $U$ and therefore is non-wandering (note that this may not be the case if $C$ has endpoints in $X'$).

Summarizing, we have found a new map $f'$ with a Brouwer-Le Calvez pair with singularities at a totally disconnected set $X'$, and a corresponding lift $\hat{f}'$ with an invariant open topological disk $\hat{U}'$ such that $U'=\pi(\hat{U}')$ has no wandering cross-cuts with endpoints in $S'\sm X'$ and $\diam(\hat{U}')>M'$. This contradicts Claim \ref{claim:td}, completing the proof Proposition \ref{pro:main-contractible}.
\end{proof}

It remains to prove Claim \ref{claim:td}. 
We may thus assume that there is a Brouwer-Le Calvez pair $(\mc{I}, \mc{F})$ for $f$ with singularities at the totally disconnected set $X$, and its natural lift $(\hat{\mc{I}}, \hat{\mc{F}})$ for $\hat{f}$ with singularities at $\hat{X}=\fix(\hat{\mc{I}})$.

Since $\hat{X}$ is totally disconnected (hence inessential), we may find a topological disk $Q\subset \hat{S}$ bounded by a loop $\gamma$ disjoint from $\hat{X}$ such that $\pi(Q) = S$ (see Remark \ref{rem:Q}). 
We will show that the number of elements $T\in \deck(\pi)$ such that $U\cap TQ\neq \emptyset$ is bounded by a constant independent of $U$.

For each leaf $\Gamma$ of $\hat{\mc{F}}$, let $V_\Gamma$ be the neighborhood given by Proposition \ref{pro:perto-corta} applied to $(\hat{\mc{I}},\hat{\mc{F}})$. Since $\bd Q$ is compact and disjoint from $\hat{X}$, we may find finitely many leaves $\Gamma_1,\dots,\Gamma_m$ of $\hat{\mc{F}}$ such that $\bd Q\subset \bigcup_{i=1}^m V_{\Gamma_i}.$
Let $E$ be the union of all sets of the form $\ol{\omega}(\Gamma_i)\cap \hat{X}$ or $\ol{\alpha}(\Gamma_i)\cap \hat{X}$ which are compact (note that a priori we do not know that all such sets are compact, but we only take the union of the compact  ones). If $E'$ is the set consisting of all points $z\in \fix(\hat{f})$ such that $L_{\hat{\mc{I}}}(z,z')\neq 0$ for some $z'\in E$, it follows from Proposition \ref{pro:link-bound} that $E'$ is relatively compact. In particular, there is an upper bound $N_0>0$ on the number of deck transformations $T\in \deck(\pi)$ such that $T(E')\cap E'\neq \emptyset$.

We need to note that there exists some fixed point $z_0\in U$. Indeed, since $U$ no wandering cross-cuts with endpoints in $\hat{S}\sm \hat{X}$, it has no wandering cross-cuts at all if seen as a subset of the surface given by the connected component of $(\hat{S}\sm \fix(\hat{f}))\cup U$ which contains $U$, so Proposition \ref{pro:nwcc-fix} tells us that $\hat{f}|_U$ has a fixed point.

Let $K$ be the closure of the set of all $p\in \hat{X}$ such that $L_{\hat{\mc{I}}}(z_0,p)\neq 0$. The set $K$ is compact, since it is contained in the union of all bounded connected components of the loop $(\hat{f}_t(z_0))_{t\in [0,1]}$ (but it could be empty, for instance if $z_0\in \hat{X})$). Since $z_0\in U$, we have from the definition of linking that $L_{\hat{\mc{I}}}(U,p) = 0$ for all $p\in \hat{X}\sm (K\cup U)$ 

Let $\Sigma_i\subset \deck(\pi)$ be the set of all deck transformations $T$ such that $TU\cap V_{\Gamma_i}\neq \emptyset$. Recall that the sets $\{TU\}_{T\in \deck(\pi)}$ are pairwise disjoint since $U$ is a disk and so is $\pi(U)$. For each $T\in \Sigma_i$ we apply Lemma \ref{lem:linking-plane} to $TU$ and study the possible cases. 
Let us denote by $\Sigma_i'$ the set of all $T\in \Sigma_i$ for which one of cases (1a), (2a) or (3) hold in Lemma \ref{lem:linking-plane} applied to $TU$. In other words, if $T\in \Sigma_i'$ then one of these properties holds:
\begin{itemize}
\item [(1a)] $\emptyset\neq \ol{\alpha}(\Gamma_i)\cap X\subset TU$;
\item [(2a)] $\emptyset \neq \ol{\omega}(\Gamma_i)\cap X\subset TU$;
\item [(3)] $\Gamma\cap TU\neq \emptyset$ and each of its connected components is wandering.
\end{itemize}
We claim that $\Sigma'_i$ has at most two elements. To see this, first note that (1a) and (2a) can only hold for one value of $T$ each, since the sets $\{TU\}_{T\in \deck(\pi)}$ are pairwise disjoint. Moreover, if (2a) holds, then $\Gamma(t)$ intersects $TU$ for arbitrarily large values of $t\in \R$, and similarly of case (1a) holds then $\Gamma(-t)$ intersects $TU$ for arbitrarily large values of $t\in \R$. 
On the other hand, if case (3) holds, then each component of $TU\cap \Gamma_i$ is wandering. Moreover, either $TU\cap \Gamma_i=\Gamma_i$, or every component $C$ of $TU\cap \Gamma_i$ is a cross-cut of $TU$. By our hypothesis, the latter implies that one of the endpoints of $C$ is a point $p\in \hat{X}\cap \bd U$. Since $\Gamma_i$ is disjoint from $X$, it follows that $\{p\}$ is either the $\omega$-limit or the $\alpha$-limit of $\Gamma_i$. In any case, there exists $t_0$ such that $\Gamma(t)\in TU$ holds either for all $t>t_0$ or for all $t<t_0$. Each of these two cases can only hold for one value of $T$, and moreover if the first case holds then (2a) cannot hold for any $T$, and similarly if the second case holds then (1a) cannot hold for any $T$. This implies that $\Sigma'_i$ has at most two elements, as claimed.

Now suppose that $T\in \Sigma_i\sm \Sigma_i'$, so case (1b) or (2b) holds for $TU$. Assume (1b) holds, so $\ol{\alpha}(\Gamma_i)\cap \hat{X}$ is compact, nonempty, and for each $p\in \ol{\alpha}(\Gamma_i)\cap \hat{X}$ we have $L_{\hat{\mc{I}}}(TU, p)\neq 0$. This implies that $L_{\hat{\mc{I}}}(Tz_0, p)\neq 0$, and since $p\in E$, the definition of $E'$ implies that $Tz_0\in E'$. An analogous argument shows that if case (2b) holds, then again $Tz_0\in E'$.
Therefore, $Tz_0\in E'$ for each $T\in \Sigma_i\sm \Sigma_i'$. But recalling that there is a bound $N_0$ on the number of $T\in \deck(\pi)$ such that $TE'\cap E'\neq\emptyset$, we conclude that $\Sigma_i$ has at most $N_0+2$ elements, for $i\in\{1, \dots, m\}$.

If for some $T\in \deck(\pi)$ we have that $TU\cap Q\neq \emptyset$, then either $TU\subset Q$ or $TU\cap \bd Q\neq \emptyset$. The first case may only happen for at most $N_1$ choices of $T$, where $N_1$ is the number of elements of $\{T\in \deck(\pi): TQ\cap Q\neq\emptyset\}$, which is finite and independent of $U$. On the other hand, the second case implies $T\in \Sigma_i$ for some $i$, and this can only happen for at most $m(N_0+2)$ different choices of $T$. Therefore we have bounded the number of elements of $\mc{T}:=\{T\in \deck(\pi): U\cap TQ\neq \emptyset\}$ by $N = N_1 + m(N_0+2)$. 

Finally, $U\subset \bigcup_{T\in \mc{T}} TQ$ and so $\diam(U)\leq N\diam(Q) := M'$, concluding the proof of the claim.
\end{proof}

\subsection{Non-contractible case}

\begin{proposition}\label{pro:twist} If $S$ is a closed hyperbolic surface and $f\colon S\to S$ is isotopic to the identity and $\hat{f}$ is its natural lift to $\hat{S}$, then for each deck transformation $T$ there exists a constant $M_T>0$ such that if $U\subset \hat{S}$ is an open topological disk such that $\hat{f}(U) = TU$ and $U$ projects injectively to an $f$-invariant disk without wandering cross-cuts, then $\diam(U)<M_T$.
\end{proposition}

\begin{proof}
For notational convenience, we will prove the statement for $T^{-1}$ instead of $T$. The case $T=\id$ was already considered in Proposition \ref{pro:main-contractible}, so we assume $T\neq \id$.
Identifying $\hat{S}$ with the Poincar\'e disk $\D$ (see Section \ref{sec:hyperbolic}), since $\hat{f}$ is a natural lift it extends to $\ol{\D}$ fixing the unit circle pointwise (we keep the notation $\hat{f}$ for this extension). The deck transformation $T$ also extends to $\ol{\D}$, having two distinct fixed points $p,q\in\bd \D$. Consider the map $g=T\hat{f}$. Note that $g$ commutes with $T$  and coincides with $T$ in $\bd \D$. Let $\Delta = \bd{\D}\sm \{p,q\}$, denote by $\til{\A}$ the closed annulus $\D\cup\Delta/\langle T^2 \rangle$, and let $\tau\colon (\ol{\D}\sm\{p,q\})\to \til{\A}$ be the projection. Let $\til{g}$ be the homeomorphism induced by $g$ on $\til{\A}$ and $\til{T}$ the map induced by $T$ (which is conjugate to a rotation by $1/2$). Note that $\til{g}$ coincides with $\til{T}$ on the boundary $\bd \til{\A} = \tau(\Delta)$. Consider a torus $\til{\T}^2$ obtained by identifying the two boundary circles in a way that if $z'$ is identified with $z$ then $\til{T}z$ is identified with $\til{T}z'$. Then $\til{T}$ and $\til{g}$ induce homeomorphisms $\til{T}'$ and $\til{g}'$ on $\til{\T}^2$ where $\til{T}'$ is a rotation by $1/2$ on the identified circle $\til{\Delta}$ and $\til{g}'$ coincides with $\til{T}'$ on $\til{\Delta}$ (See Figure \ref{fig:deck-together}).

\begin{figure}[ht]
\includegraphics[width=\textwidth]{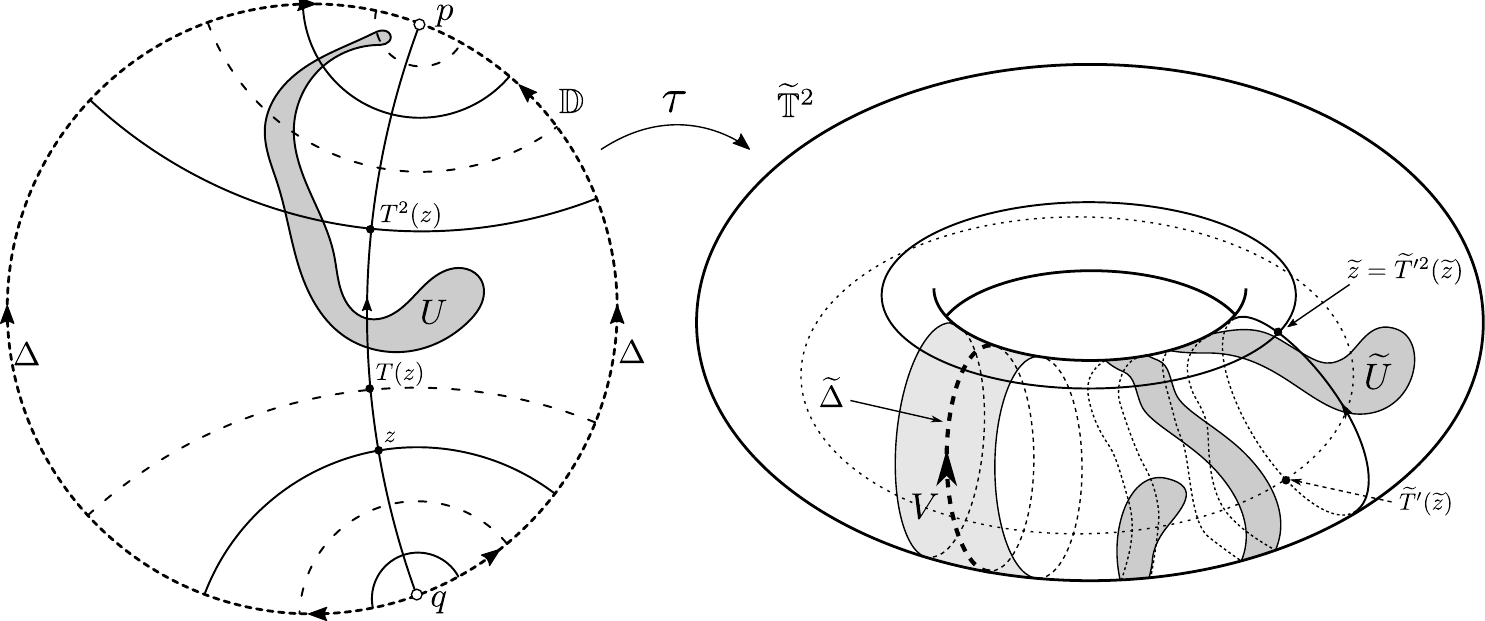}
\caption{The map $\tau\colon \D\to \til{\T}^2\sm \til{\Delta}$}
\label{fig:deck-together}
\end{figure}

We note that $\til{g}'$ is isotopic to the identity in $\til{T}'$: indeed, extending each $f_t$ so that it fixes $\bd \D$ pointwise, if $g_t = T\hat{f}_t$ then $(g_t)_{t\in [0,1]}$ defines an isotopy between $T$ and $g$ on $\D\cup \Delta$, and each $g_t$ commutes with $T^2$ so it also induces an isotopy $(\til{g}_t)$ from $\til{T}$ to $\til{g}$ on $\til{\A}$ such that $\til{g}_t$ coincides with $\til{T}$ on $\bd \til{\A}$. The latter condition implies that $\til{g}_t$ respects the identification of the boundary circles and so it induces a map $\til{g}_t'$ on $\til{\T}^2$ which provides an isotopy from $\til{T}'$ to $\til{g}'$. But $\til{T}'$ is clearly isotopic to the identity, hence so is $\til{g}'$, as claimed.

We also observe that $\fix(\til{g}')$ is inessential. Indeed, since no point of $\til{\Delta}$ is fixed, $\fix(\til{g}') \subset \til{\T}^2\sm \til{\Delta} = \til{\A}\sm \bd \til{\A}$ and then $$\pi(\tau^{-1}(\fix(\til{g}'))) = \pi(\fix(\hat{g})) = \pi(\fix(T\hat{f}))\subset \fix(f).$$
From the fact that $\fix(f)$ is inessential, there is an inessential open set $W\subset S$ containing $\fix(f)$, and it is easy to verify that $\tau(\pi^{-1}(W))$ is inessential in $\til{\A}$ (and contains $\fix(\til{g}')$, so $\fix(\til{g}')$ is inessential as claimed).

Thus the map $\til{g}'$ satisfies the hypotheses of Proposition \ref{pro:main-contractible}, which means that there exists $M'>0$ such that any $\til{g}'$-invariant open topological disk $\til{U}$ without wandering cross-cuts must satisfy $\diamup'(\til{U})\leq M'$, where $\diamup'$ denotes the covering diameter using an appropriate metric on the torus $\til{\T}^2$.

Note that $\til{g}'|_{\til{\Delta}}$ is conjugate to a rotation by $1/2$ (mod $\Z$). A straightforward argument implies that there is a neighborhood of $V$ of $\til{\Delta}$ with the property that if $\gamma$ is any simple arc joining a fixed point of $\til{g}'$ to a point of $V$, then some iterate of $\gamma$ by $\til{g}'$ lifts to the universal covering to an arc of diameter greater than $M$. In particular, since any $\til{g}'$-invariant open topological disk $\til{U}$ without wandering cross-cuts must contain a fixed point (by Proposition \ref{pro:nwcc-fix}), it follows that any such disk must be disjoint from $V$. We may thus find a closed annulus $A_0\subset \til{\T}^2\sm \til{\Delta}$ with the property that any such disk must be contained in $A_0$. 

Since $A_0$ compact and disjoint from $\til{\Delta}$, the metric on $\D$ projects by $\tau$ to a Riemannian metric on $A_0$ under which $\diamup(\til{U}) = \diam(U)$ whenever $\til{U}\subset A_0$ is a topological disk and $U$ is a connected component of $\tau^{-1}(U)$. Due to the compactness of $A_0$, this metric must be equivalent to the metric of $A_0$ regarded as a subset of $\til{\T}^2$. Hence there is a constant $c>0$ such that $\diamup(\til{U})\leq c\diamup'(\til{U})$.

Finally, suppose that $U$ is an open topological disk such that that $\hat{f}(U)=T^{-1}(U)$ and projects to an $f$-invariant topological disk without wandering cross-cuts. Then $U$ is $g$-invariant and has no wandering cross-cuts for $g$, and so $\til{U} = \tau(U)$ is a $\til{g}$-invariant open topological disk. The fact that $\til{U}$ has no wandering cross-cuts needs some care and will be shown below; but assuming that fact, it follows from the previous comments that $\til{U}\subset A_0$ and $\diamup'(\til{U})\leq M'$, hence $\diamup(\til{U})\leq M := cM'$ as we wanted.

We finish the proof showing that $\til{U}$ has no wandering cross-cuts: if $C$ is a cross-cut of $\til{U}$ with endpoints in $\til{\A}\sm \bd \til{\A}$, then $\pi(\tau^{-1}(C))$ is a cross-cut of $U$ (hence nonwandering for $f$), and it follows easily that $C$ is nonwandering for $\til{g}$. But we also need to consider the case where one of the endpoints of $C$ is in $\bd \til{\A}$. Suppose that some such $C$ is wandering. Then the connected component $\hat{C}$ of $\tau^{-1}(C)$ contained in $U$ is a wandering for $g$, and since $C$ has an endpoint in $\bd \til{\A}$, $\hat{C}$ has an endpoint in $\Delta$. Since $g$ coincides with $T$ on $\Delta$, the endpoint of $\hat{C}$ in $\Delta$ is not periodic by $g$, and therefore Proposition \ref{pro:nwcc-nwcs} implies that there is a wandering cross-section $W$ of $U$ for $g$. We claim that $\bd W$ contains some point of $\bd_\D U$. Indeed if this were not the case then $W$ would be a cross-section of $\D$, but since $\pi|_W$ is injective this would be a contradiction: since $S$ is compact, any neighborhood of a point of $\bd \D$ projects onto $S$. Thus $\bd W$ contains some boundary point of $U$ in $\D$, and intersecting $W$ with a small disk around this point we may produce a cross-cut $\hat{C}'$ of $U$ contained in $W$ with endpoints in $\D$. But then $\pi(\hat{C}')$ is a wandering cross-cut for $f$, contradicting our hypothesis.
\end{proof}

\subsection{Proof of Theorem \ref{th:main-disk}}
Recall that $S$ is a closed orientable surface, $f\colon S\to S$ is a homeomorphism homotopic to the identity, and $\fix(f)$ is inessential. Let $U$ be an open $f$-invariant topological disk without wandering cross-cuts. We want to show that $\diamup(U)$ is bounded by a constant independent of $U$. If $S = \T^2$, then there always exists a natural lift $\hat{f}$ of $f$ such that any connected component of $\pi^{-1}(U)$ is $\hat{f}$-invariant, so the claim of the theorem follows immediately from Proposition \ref{pro:main-contractible}. 

Now suppose that $S$ is a hyperbolic surface, and let $\hat{f}$ be a natural lift of $f$. Then there is a constant $N>0$ such that $d(z, \hat{f}(z))\leq N$ for all $z\in \hat{S}$, and so if $T^{-1}\hat{f}$ has a fixed point for some $T\in \deck(\pi)$ then there exists  $z$ such that $d(z, \hat{f}(z)) = d(z, Tz)\leq N$. Since $S$ is compact, only finitely many such $T$'s may exist, so the set $\Sigma = \{T\in \deck(\pi) : \fix(T^{-1}\hat{f})\neq \emptyset\}$ is finite. 
We know from Proposition \ref{pro:nwcc-fix} that there is a fixed point $z_0\in U$. If $\hat{z}_0\in \pi^{-1}(z_0)$, then $\hat{f}(\hat{z}_0)=T\hat{z}_0$ for some $T\in \deck(\pi)$, and so $T\in \Sigma$. If $\hat{U}$ is the connected component of $\pi^{-1}(U)$ containing $\hat{z}_0$, then $\hat{f}(\hat{U})=T\hat{U}$ and so by Proposition \ref{pro:twist} it follows that $\diamup(U)\leq M_T \leq M := \max_{T\in \Sigma} M_T$. \qed

\subsection{A version for surfaces with boundary}\label{sec:main-disk-bd}

For a compact surface with boundary $S$, we say that a compact subset $K\subset S$ is inessential if $K$ is contained in some closed topological disk in $S$. This coincides with the definition for closed surfaces given in Section \ref{sec:essential}, but we remark that some properties stated earlier for inessential sets may fail to hold if $S$ has nonempty boundary (for instance, the complement of an inessential set may be simply connected). Nevertheless, Theorem \ref{th:main-disk} remains valid:

\begin{theorem}\label{th:main-disk-bd} If $S$ is a compact surface with boundary and $f\colon S\to S$ a homeomorphism such that $\fix(f)$ is inessential, then there is a constant $M>0$ such that any open $f$-invariant topological disk $U$ without wandering cross-cuts has covering diameter at most $M$.
\end{theorem}

\begin{proof} We assume that $\chi(S)\leq 0$, otherwise the theorem is trivial. Consider the surface $S'$ obtained by ``doubling'' $S$, \ie $S'$ is obtained from $S$ and a copy $S_0$ of $S$ by joining each boundary component of $S$ to the corresponding boundary component of $S_0$ in an appropriate way so that $S'$ is an orientable surface.
We note that $S\subset S'$ is a filled subset of $S'$, since $S'\sm S$ consists of a unique component homeomorphic to the interior of $S$ (which is not a topological disk because $\chi(S)\leq 0$). 
We may extend $f$ arbitrarily to a homeomorphism $f'\colon S'\to S'$ such that $f'|_S=f$ and $f'$ has finitely many fixed points in $S'\sm S$. This implies that $\fix(f')$ is inessential in $S'$, so Theorem \ref{th:main-disk} applies to $f'$, \ie there exists $M'>0$ such that any open $f'$-invariant topological disk without wandering cross-cuts has covering diameter at most $M'$ in $S'$. 

Let $\pi'\colon \hat{S}'\to S'$ be the universal covering map of $S'$. If $\hat{S}$ is any connected component of $\pi'^{-1}(S)$, then $\hat{S}$ is simply connected (due to the fact that $S$ is filled in $S'$) and covers $S$. Thus $\pi=\pi'|_{\hat{S}}$ is a universal covering map of $S$, and if $U\subset S$ is an open $f$-invariant topological disk without wandering cross-cuts then it satisfies the same hypotheses for $f'$ in $S'$, so $\diamup_{S'}(U)\leq M'$; in particular any connected component $\hat{U}$ of $\pi^{-1}(U)$ satisfies $\diam(\hat{U})\leq M'$. Note that the distance $d$ induced in $\hat{S}$ by the lifted Riemannian metric may not coincide with the restriction to $\hat{S}$ of the distance $d'$ induced by the lifted metric in $\hat{S}'$; however, since the metrics are equivariant, there exists $M>0$ such that if $x,y\in \hat{S}$ and $d'(x,y)\leq M'$ then $d(x,y)\leq M$, and therefore $\diamup_S(U)\leq M$.
\end{proof}

\section{General open invariant sets: proof of Theorem \ref{th:main-gen}}\label{sec:main-gen}

Let $S$ be a closed orientable surface, $f\colon S\to S$ an area-preserving homeomorphism such that $\fix(f)$ is inessential, and $U$ an arbitrary connected $f$-invariant open set. In order to show that $\diamup(U)<\infty$, it suffices to consider the case where $U$ is filled, due to Proposition \ref{pro:diam-fill}; hence we assume that $U=\fill(U)$. This means that $U$ has at most finitely many topological ends, and we may find a surface with boundary $U_0\subset U$ such that each connected component of $U\sm U_0$ is a collar of some topological end. For each topological end of $U$, we may choose a simple loop $\gamma$ in $U\sm U_0$ surrounding the corresponding end close enough so that both $\gamma$ and $f(\gamma)$ are in the same connected component of $U\sm U_0$. If $A\subset U\sm U_0$ is a closed annulus containing $\gamma$ and $f(\gamma)$ in its interior, we may find an area-preserving homeomorphism $h\colon S\to S$ homotopic to the identity, which is the identity outside $A$ and maps $f(\gamma)$ back to $\gamma$ (using an area-preserving variation of Schoenflies' theorem; see for instance \cite{beguin-schonflies}). Thus $hf$ fixes the loop $\gamma$, preserves area, and leaves $U$ invariant. Furthermore, we may assume that $\fix(hf)$ is still inessential. 
Repeating this on each topological end of $U$, if $\mathfrak{p}_1,\dots,\mathfrak{p}_k$ denote these topological ends, we obtain an area-preserving map $g\colon S\to S$ homotopic to the  identity with an inessential set of fixed points, such that $g(U)=U$ and for each $i$ there is an invariant loop $\gamma_i$ bounding a collar $A_i$ of $\mathfrak{p}_i$ in $U$.
We may further assume that the collars $A_1,\dots A_k$ are pairwise disjoint.

For each $i$, choose a rectifiable essential loop $\gamma_i'$ in $A_i$, so $\gamma_i'$ bounds in $U$ a smaller collar $A_i'\subset A_i$ of $\mathfrak{p}_i$, and choose another rectifiable loop $\gamma_i'' \subset A_i\sm \ol{A}_i'$ bounding a collar $A_i''$ of $\mathfrak{p}_i$. Denote by $V_i$ the annulus $A_i\sm \ol{A}_i''$, and let $S_i'$ denote the surface obtained from the connected component of $S\sm V_i$ which contains $A_i''$ (which may be all of $S\sm V_i$ if $\gamma_i''$ is non-separating) by collapsing the boundary component $\gamma_i''$ to a point $\mathfrak{e}_i$ (See Figure \ref{fig:surgery-new}). 
\begin{figure}[ht]
\includegraphics[width=\linewidth]{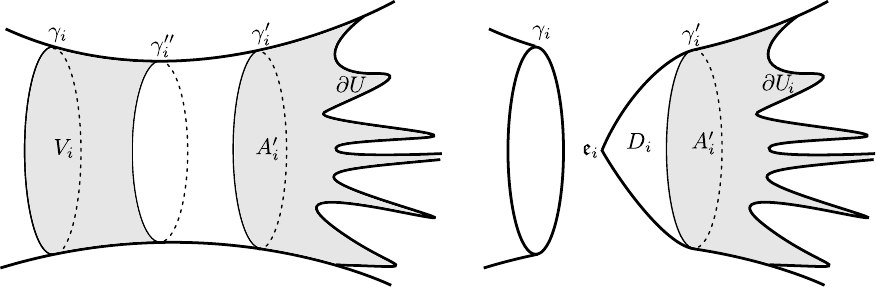}
\caption{Construction in the proof of Theorem \ref{th:main-gen}}
\label{fig:surgery-new}
\end{figure}
Then $U_i:=\mathfrak{e}_i \cup A_i''$ is an open topological disk in $S_i'$, and by Theorem \ref{th:main-disk-bd} it follows that $\diamup_{S_i'}(U_i)<\infty$. Moreover, if $D_i=\mathfrak{e}_i\cup (A_i''\sm A_i')$ then $\ol{D}_i$ is a closed topological disk in $S_i'$, so by Proposition \ref{pro:cover-disk} we have that $\diamup_{S_i'\sm D_i}(A_i')=\diamup_{S_i'\sm D_i}(U_i\sm \ol{D}_i)<\infty$. Since $S_i'\sm D_i$ is a surface with rectifiable boundary in $S$, it follows from Proposition \ref{pro:diam-sub} that $\diamup_S(A_i')\leq \diamup_{S_i'\sm  D_i}(A_i')<\infty$. Finally, we note that $U_0=U\sm \bigcup_{i=1}^k {A_i'}$ is a compact surface with rectifiable boundary in $S$, so $\diamup_S(U_0)<\infty$, and $U=U_0\cup \bigcup_{i=1}^k A_i'$, so by Proposition \ref{pro:diam-partition} 
$$\diamup_S(U)\leq \diamup_S(U_0) + \sum_{i=1}^k\diamup_S(A_i') < \infty,$$
completing the proof of Theorem \ref{th:main-gen}.
\qed

\section{Essential and inessential dynamics}\label{sec:essine}
Let $S$ be a closed orientable surface and $f\colon S\to S$ a homeomorphism.
A point $x\in S$ is said to be \emph{dynamically inessential} if there is a neighborhood $U$ of $x$ such that $\orb_f(U)= \bigcup_{n\in \Z} f^n(U)$ is an inessential subset of $S$. The set of all such points is denoted $\ine(f)$. If $x\notin \ine(f)$, \ie if $\orb_f(U)$ is an essential set for every neighborhood $U$ of $x$, then we call $x$ a \emph{dynamically essential} point. 
We say that $x$ is dynamically \emph{fully essential} if $\orb_f(U)$ is a fully essential set for every neighborhood $U$ of $x$. We denote the set of all dynamically fully essential points by $\fess(f)$. 
Note that whenever $x$ is a nonwandering point, the connected component $U_0$ of $\orb_f(U)$ containing $x$ is a periodic open set. 

\begin{proposition}\label{pro:trans} The following properties hold:
\begin{itemize}
\item[(1)] $\ine(f)$ is open and $f$-invariant;
\item[(2)] $\ess(f)$ and $\fess(f)$ are closed and $f$-invariant;
\item[(3)] $\fess(f)$ is externally syndetically transitive, \ie for any pair of neighborhoods (in $S$) $U,V$ of points of $\fess(f)$, the set $\{n\in \Z : f^n(U)\cap V\}\neq \emptyset$ has uniformly bounded gaps.
\end{itemize}
\end{proposition}
\begin{proof} (1) and (2) follow directly from the definitions. For (3), if $U$ and $V$ are neighborhoods in $S$ of points $x,y\in \fess(f)$, then $\orb_f(U)$ and $\orb_f(V)$ are fully essential open sets, and by choosing an essential loop in each set and using a compactness argument one may find $N>0$ such that $U'=\bigcup_{n=-N}^N f^n(U)$ and $V'=\bigcup_{n=-N}^N f^n(V)$ are fully essential. Since any two fully essential open sets must intersect and for each $k\in \Z$ the set $f^{2Nk}(U')$ is fully essential, it intersects $V'$. This implies that for some $i_k$ with $2N(k-1)\leq i_k\leq 2N(k+1)$ one has $f^{i_k}(U)\cap V\neq \emptyset$, and $|i_{k+1}-i_k|\leq 4N$.
\end{proof}

As a direct consequence of Theorem \ref{th:main-disk} we have the following:
\begin{corollary}\label{cor:island} Let $f\colon S\to S$ be a nonwandering homeomorphism isotopic to the identity, and assume that $\fix(f^n)$ is inessential for all $n>0$. Then every $x\in \ine(f)$ is contained in some periodic open topological disk which is homotopically bounded (by a bound depending only on its period).
\end{corollary}
\begin{proof} 
For each $x\in \ine(f)$ there exists some disk $B$ containing $x$ such that $\orb_f(B)$ is inessential. If $U_0$ is the connected component of $\orb_f(B)$ containing $x$, then from the fact that $\orb_f(B)$ is invariant and $f$ is nonwandering we deduce that $f^n(U_0)=U_0$ for some $n>0$, and so if $U=\fill(U_0)$, we have that $U$ is an open $f^n$-invariant topological disk, so $\diamup(U)<\infty$ by Theorem \ref{th:main-disk}.
\end{proof}

\subsection{Proof of Theorem \ref{th:essine}}

Assume that cases (1)-(3) from the statement of Theorem \ref{th:essine} do not hold. This implies that every open $f$-invariant set $U$ is either inessential or fully essential. Indeed, otherwise $U$ has a connected component which is essential but not fully essential, and since $f$ preserves area $U$ must be periodic, so Theorem \ref{th:main-gen} applied to $f^n$ (where $n$ is the period of $U$) tells us that $U$ is homotopically bounded, contradicting our assumption that case (3) from the statement of Theorem \ref{th:essine} does not hold.

In particular, if $z\in \ess(f)$ then $\orb_f(U)$ is fully essential for any neighborhood $U$ of $z$, so $\fess(f)=\ess(f)$. Moreover, being an open $f$-invariant set, $\ine(f)$ is either fully essential or inessential.

We will show that $\ine(f)$ is inessential. Assume for a contradiction that $\ine(f)$ is fully essential, so ($\ine(f)$ being open) we may find a loop $\gamma\subset \hat{S}$ bounding a region $Q\subset \hat{S}$ such that $\pi(Q)= S$ and $\pi(\gamma)\subset \ine(f)$ (see Remark \ref{rem:Q}). By Corollary \ref{cor:island}, the loop $\gamma$ may be covered by finitely many bounded disks $U_1,\dots, U_m$ such that for each $i$ the set $\pi(U_i)$ is $f$-periodic, and we assume that each $U_i$ intersects $\gamma$. Choose $n>0$ such that $f^n(\pi(U_i))=\pi(U_i)$ for all $i\in\{1,\dots, m\}$, and let $\hat{f}$ be the natural lift of $f$. Then, for each $i$ we may find $T_i\in \deck(\pi)$ such that $\hat{f}^n(U_i) = T_iU_i$. We claim that if $U_i\cap U_j\neq \emptyset$ then $T_i=T_j$. Indeed, if $U_i\cap U_j\neq \emptyset$, then $\hat{f}^{kn}(U_i)\cap \hat{f}^{kn}(U_j) = T_i^k U_j\cap T_j^k U_j \neq \emptyset$. But since both $U_i$ and $U_j$ are bounded, taking any $z\in U_i\cap U_j$ we have that $d(T_i^kz, T_j^kz)$ remains bounded for all $k\in \N$. Since $T_i$ and $T_j$ are hyperbolic isometries, this implies $T_i=T_j$. 

For each $x\in \gamma$ choose the smallest $i$ such that $x\in U_i$ and let $T_x = T_i$. As we just saw, the map $x\mapsto T_x$ is locally constant, and since $\gamma$ is connected conclude that $T_x$ is constant. Since every set $U_i$ intersects $\gamma$, this means that there exists $T\in \deck(\pi)$ such that $\hat{f}^n(U_i)=TU_i$ for all $i\in \{1,\dots, m\}$.
If $W=\fill(\bigcup_{i=1}^m U_i)$ we see that $\hat{f}^n(W)= TW$. Since $S$ is hyperbolic and $f$ is homotopic to the identity, its natural lift $\hat{f}$ has a fixed point, and in particular it has a fixed point in $Q\subset W$. Since $W$ is bounded, we conclude that $T=\id$, and therefore $\hat{f}^{nk}(Q)\subset W$ for all $k\in \Z$, which implies that $d(\hat{f}^{nk}(z),z)\leq \diam(W):=M$ for all $k\in \Z$ and $z\in Q$. Since $\hat{f}$ commutes with elements of $\deck(\pi)$, the same property holds for all $z\in \hat{S}$. Thus $\hat{f}^n$ has uniformly bounded displacement, and since $\hat{f}$ is uniformly continuous we conclude that $\hat{f}$ has uniformly bounded displacement as well. But this contradicts our assumption that case (2) of the theorem does not hold.

Thus $\ine(f)$ is inessential, and therefore $\fess(f)=\ess(f)$ is fully essential. To show that $\fess(f)$ is connected, note that if $C$ is the fully essential connected component of $\fess(f)$ (which is necessarily unique, and therefore invariant) and $U$ is a neighborhood of any point $z\in \fess(f)$, then $\orb_f(U)$ is a fully essential open set and therefore intersects $C$. But since $C$ is invariant, it follows that $C\cap U\neq \emptyset$, and since $C$ is closed and $U$ was an arbitrary neighborhood of $z\in \fess(f)$ we conclude that $z\in C$. Thus $C=\fess(f)$ showing the connectedness.

From the fact that $\fess(f)$ is connected and fully essential we also conclude that $\ine(f)$ is an union of topological disks, which must be periodic since they are permuted and $f$ preserves area. Moreover, by Theorem \ref{th:main-disk} these disks are homotopically bounded by a bound that depends only on the periods of the disks. The fact that $\fess(f)$ is externally syndetically transitive was already established in Proposition \ref{pro:trans}. Therefore we conclude that item (4) of Theorem \ref{th:essine} holds, completing the proof.
\qed

\subsection{The case of the torus}

The following result is a version of Theorem \ref{th:essine} for the case of the torus, which slightly improves \cite[Theorem A]{kt-strictly}:

\begin{theorem}\label{th:torus}
If $f\colon \T^2\to \T^2$ is a nonwandering homeomorphism homotopic to the identity, then one of the following holds:
\begin{itemize}
\item[(1)] There is $k\in \N$ such that $\fix(f^k)$ is fully essential;
\item[(2)] There is $k\in \N$ such that $f^k$ has uniformly bounded displacement;
\item[(3)] There is $k\in \N$ such that $f^k$ has an invariant essential  open annulus;
\item[(4)] $\ess(f)=\fess(f)$ and is connected, fully essential, and externally transitive, while $\ine(f)$ is the union of a family of pairwise disjoint periodic bounded disks (by a bound that depends only on the period of the disks).
\end{itemize}
\end{theorem}

\begin{proof}
The proof is the same of \cite[Theorem A]{kt-strictly}, except that we do not need to assume that $f$ is \emph{non-annular} to apply Theorem \ref{th:main-disk} which allows to unfold the first two cases of \cite[Theorem A]{kt-strictly} into cases (1)-(3). We sketch the idea of the proof: If $\fix(f)$ is essential but not fully essential, then $S\sm \fix(f)$ has an essential connected component whose filling is an $f$-invariant essential annulus, so case (3) holds. Now suppose that $\fix(f)$ is inessential. If $\ine(f)$ is essential but not fully essential, then there is a periodic topological annulus (namely the filling of an essential connected of $\ine(f)$) so case (3) holds. If $\ine(f)$ is fully essential then the exact same argument used in the proof of Theorem \ref{th:essine} (using the fact that each point of $\ine(f)$ belongs to a homotopically bounded periodic disk) implies that $f^k$ has uniformly bounded displacement for some $k$, so case (2) holds. Finally, if $\ine(f)$ is inessential and neither (1)-(3) holds, case (4) follows exactly as in Theorem \ref{th:essine}.
\end{proof}

\subsection{Proof of Corollary \ref{cor:trans}}

Assume $f$ has fully essential dynamics. If there are no periodic homotopically bounded disks, then by Theorem \ref{th:main-gen} we have that $\ess(f)=\fess(f)=S$ and the fact that $f$ is externally syndetically transitive in $\fess(f)$ implies that $f$ is syndetically transitive in the whole space $S$. 

Conversely, suppose for a contradiction that $f$ is transitive but $\ine(f)\neq \emptyset$. If $U$ is a connected component of $\ine(f)$, then $U$ is a periodic bounded disk and its orbit is dense in $S$. Thus there is $k\in \N$ such that $\bigcup_{i=0}^{k-1}f^i(\ol{U}) = S$. If $\mc{U}$ is a connected component of $\pi^{-1}(U)$ and $\hat{f}$ is the natural lift of $f$, then there is $T\in \deck(\pi)$ such that $\hat{f}^k(\mc{U})=T(\mc{U})$, and for all $z\in S$ the set $\pi^{-1}(z)$ intersects $K = \bigcup_{i=0}^{k-1}\hat{f}^i(\ol{\mc{U}})$. Since, being a natural lift, $\hat{f}$ has a fixed point, this means that there is a fixed point of $\hat{f}$ in $K$. Noting that $\hat{f}(K)=TK$ and $K$ is compact, it follows that $T=\id$, so $\hat{f}^k(\mc{U})=\mc{U}$ and for any $z\in K$ and $n\in \Z$ we have $d(\hat{f}^n(z), z)\leq M=\diam(K)$. Since $\hat{f}$ commutes with deck transformations and $\pi(K)=S$, the same property holds for arbitrary $z\in \hat{S}$, so $f$ has uniformly bounded displacement contradicting our assumption that $f$ has fully essential dynamics.
\qed

\section{Rotation sets and fully essential dynamics}
\label{sec:rot}

Throughout this section we assume that $S$ is a hyperbolic closed surface of genus $g>1$, and $\mc I = (f_t)_{t\in [0,1]}$ is an isotopy from the identity to $f=f_1\colon S\to S$. We denote by $\hat{f}\colon \hat{S}\to \hat{S}$ the natural lift of $f$. Given a function $\Phi\colon S\to V$, where $V$ is a vector space, we will abbreviate Birkhoff sums of $\Phi$ as $$\Phi^n_f(x)=\sum_{k=0}^{n-1} \Phi(f^k(x)).$$

Given a loop $\alpha$ in $S$, we denote by $[\alpha]\in H_1(S,\Z)\subset H_1(S,\R)$ its homology class. Recal that $H_1(S,\Z)\simeq \Z^{2g}$ and $H_1(S,\R) \simeq \R^{2g}$ accordingly. 
For any deck transformation $T\in \deck(\pi)$, an arc $\hat{\gamma}_T$ joining any point $x\in \hat{S}$ to $Tx$ projects into a loop $\gamma_T=\pi(\hat{\gamma}_T)$ whose free homotopy class (and in particular its homology class) is determined only by $T$. We denote by $[T]:=[\gamma_T]$ this homology class. 
Recall that $\deck(\pi)$ is isomorphic to the fundamental group $\pi_1(S)$. The kernel of $T\mapsto [T]$ is precisely the commutator group of $\deck(\pi)$, so we may regard $T\mapsto [T]$ as the quotient map from $\pi_1(S)$ to its abelianization. 

We endow $H_1(S,\R)$ with the \emph{stable norm} $\norm{\cdot}$ \cite[\S4]{gromov}, which has the property that for any rectifiable loop $\gamma$
$$\norm{[\gamma]}\leq \length(\gamma).$$
We will frequently use this property in the next subsections.

\subsection{Rotation vectors and rotation set}

Denote by $\mc{I}_x$ the arc $t\mapsto f_t(x)$ joining $x$ to $f(x)$, and for $n\in \N$ define 
$$\mc{I}^n_x = \mc{I}_x*\mc{I}_{f(x)}*\cdots*\cdots \mc{I}_{f^{n-1}(x)},$$
 which is an arc joining $x$ to $f^n(x)$.

Fix a base point $b\in S$ and a family $\mc{A}=\{\gamma_x : x\in S\}$ of rectifiable arcs such that $\gamma_x$ joins $b$ to $x$ and the length of $\gamma_x$ is bounded by a uniform constant $C_\mc{A}$. We note that the map $x\mapsto \gamma_x$ may be chosen to be measurable (see \cite{franks-rot} for details).

For each $x\in S$, define $\alpha_x = \gamma_x*\mc{I}_x*\gamma_{f(x)}^{-1}$ and for $n\in \N$, let $\alpha^n_x = \gamma_x*\mc{I}^n_x*\gamma_{f^n(x)}^{-1}$ which is a loop with base point $b$.  We may then define the (discretized) \emph{homological displacement function} by $\Phi_f(x) = [\alpha_x]$. Note that since $\alpha^n_x$ is homotopic to $\alpha_x*\alpha_{f(x)}*\cdots*\alpha_{f^{n-1}(x)}$, one has
$$[\alpha^n_x] =\sum_{k=0}^{n-1} [\alpha_{f^k(x)}] = \sum_{k=0}^{n-1} \Phi_f(f^k(x)) = \Phi^n_f(x)$$
We remark that in these definitions, the arc $\mc{I}_x^n$ can be replaced by any arc joining $x$ to $f^n(x)$ and homotopic with fixed endpoints to $\mc{I}_x^n$. This implies that $\Phi_f$ depends only $f$ and on the choice of $\mc{A}$, but not on the isotopy $\mc{I}$, since we may replace $\mc{I}_x^n$ by the projection of a geodesic segment in $\hat{S}$ joining $\hat{x}$ to $\hat{f}^n(\hat{x})$ where $\hat{x}\in \pi^{-1}(x)$. This also shows that $\norm{\Phi_f(x)}$ is bounded by $2C_{\mc A} + L$, where $L= \sup\{d(z, \hat{f}(z)) : z\in \hat{S}\}<\infty$.

Of course the map $\Phi_f$ depends on the choices of the basepoint $b$ and the arcs $\gamma_x$. However, given a different basepoint $b'$ and a family $\mc{A}'=\{\gamma_x':x\in S\}$ of rectifiable arcs whose lengths are uniformly bounded  by $C_{\mc A}'$ such that $\gamma_x'$ joins $b'$ to $x$, and defining $\alpha_x'^n$ analogously one has 
\begin{equation}\label{eq:basechange}
[\alpha_x'^n]=[\gamma_x'*\mc{I}^n_x*\gamma_{f^n(x)}'^{-1}]=[\alpha_x^n*\delta^n_x] = [\delta_x^n]+ \Phi^n_f(x),
\end{equation}
where  $\delta^n_x = \gamma_{f^n(x)}*\gamma_{f^n(x)}'^{-1}*\gamma_x'*\gamma_x^{-1}$. Indeed, the loop $\alpha_x'^n$ is freely homotopic to $\mc{I}_x^n*\delta_x^n$.
In particular, if $\Phi_f'(x) = [\alpha_x']$, we have that 
\begin{equation}\label{eq:cohomologous}
\norm{\Phi_f'^n(x)-\Phi_f^n(x)} = \norm{\delta^n_x}\leq M := 2C_{\mc{A}}+2C_{\mc{A}'}.
\end{equation}

If the limit $$\rho(f,x) = \lim_{n\to \infty} \frac{1}{n}\Phi_f^n(x)\in H_1(S,\R)$$ exists, we say that $x$ has a well-defined (homological) rotation vector. In \cite{franks-rot}, Franks defines the rotation set of $f$ as the set of all rotation vectors of this kind. A more general definition given by Pollicott \cite{pollicott} considers all accumulation points of the sequence $\frac{1}{n}\Phi_f^n(x)$. We will give an even more general definition, similar to the Misiurewicz-Ziemian rotation set of a toral homeomorphism \cite{m-z}: The (Misiurewicz-Ziemian) rotation set of $f$ over a set $E\subset S$ is defined as the set $\rotmz(f,E)$ consisting of all limits of the form
\begin{equation}\label{eq:rot-def}
v=\lim_{k\to \infty} \frac{1}{n_k} \Phi_f^{n_k}(x_k)\in H_1(S,\R),
\end{equation}
where $x_k\in E$ and $n_k\to \infty$. The rotation set of $f$ is then defined as $\rotmz(f) = \rotmz(f,S)$.  By (\ref{eq:cohomologous}), the rotation set depends only on $f$, but not on the choice of the isotopy, the basepoint $b$ or the arcs $\gamma_x$. 
Note that our definition coincides with
$$\rotmz(f,E) = \bigcap_{m\geq 0} \ol{\bigcup_{n\geq m} \{\Phi^n_f(x)/n : x\in E\}}$$
which shows that the rotation set is compact and bounded (since $\Phi_f$ is bounded).

Note that, using a computation similar to (\ref{eq:basechange}), if one chooses a rectifiable arc $\beta$ joining $f^n(x)$ to $x$ one has 
\begin{equation}\label{eq:rot-alt}
[\mc{I}^n_x*\beta] = [\gamma_{x}^{-1}*\alpha_{x}^{n}*\gamma_{f^{n}(x)}*\beta]= \Phi_f^{n}(x) + [\gamma_{f^{n}(x)}*\beta*\gamma_{x}^{-1}].
\end{equation}
Thus,
$\snorm{[\mc{I}^{n}_{x}*\beta] - \Phi_f^{n}(x)} \leq 2C_\mc{A}+\length(\beta)$.
As a consequence, an alternate but equivalent definition of rotation vectors and rotation sets is obtained by considering all limits of the form 
$$v=\lim_{k\to \infty} \frac{1}{n_k}[\mc{I}^{n_k}_{x_k}*\beta_k]$$ 
where $x_k\in S$, $n_k\to \infty$ and $\beta_k$ are rectifiable arcs joining $f^{n_k}(x_k)$ to $x_k$ such that $\sup_k \length(\beta_k)<\infty$.

We remark that the same definition applied to $S=\T^2$ is also meaningful, but the resulting rotation set depends on the choice of the isotopy $\mc{I}= (f_t)_{t\in [0,1]}$. However, the rotation set thus obtained is essentially the Misiurewicz-Ziemian rotation set of the natural lift $\hat{f}$ of $f$ associated to $\mc{I}$. More precisely, there is an isomorphism $\R^2 \to H_1(\T^2, \R)$ which maps the basis $\{(1,0), (0,1)\}$ to the homology classes corresponding to the two deck transformations $z\mapsto z+(1,0)$ and $z\mapsto z+(0,1)$, and this isomorphism maps the classical Misiurewicz-Ziemian rotation set to the set $\rotmz(f)$ as defined above (using the isotopy $\mc{I}$). In the case of $\T^2$, it is known that this rotation set is convex \cite{m-z}. This is no longer the case when $S$ is hyperbolic.

The following properties follow easily from the definitions:
\begin{itemize}
\item $\rotmz(f)$ is compact;
\item $\rotmz(f^n) = n\cdot\rotmz(f)$;
\item If $S=\bigcup_{i=1}^m E_i$, then $\rotmz(f) = \bigcup_{i=1}^m \rotmz(f,E_i)$.
\item $0\in \rotmz(f)$;
\end{itemize}
the later being due to the fact that $f$ is homotopic to the identity and $S$ is hyperbolic, so there exists a contractible fixed point (\ie a fixed point $x$ such that $\mc{I}_x$ is a homotopically trivial loop).

\subsection{Mean rotation vectors}\label{sec:rot-meas}

If the arcs $\gamma_x$ in the definition of $\Phi_f$ are chosen adequately, one may guarantee that $x\mapsto \Phi_f(x)$ is not only bounded but also Borel measurable (see \cite{franks-rot}). Denote by $\mc{M}(f)$ the set of invariant Borel probability measures. We may define the mean rotation vector associated to $\mu\in \mc{M}(f)$ by 
$$\rotmeas(f,\mu) = \int \Phi_f d\,\mu \in H_1(S,\R).$$
Since rotation vectors are calculated as limits of Birkhoff averages of $\Phi_f$, the Birkhoff Ergodic Theorem guarantees that $\mu$-almost every point has a well defined rotation vector $\rho(f,x)$, and 
$\rotmeas(f,\mu) = \int \rho(f,x)\,d\mu(x)$, and if $\mu$ is ergodic then $\rho(f,x)=\rotmeas(f,\mu)$ for $\mu$-almost every $x$.
Due to these facts and (\ref{eq:cohomologous}), the mean rotation vector is independent of any choices made in the definitions.
Denote by $\rotmeas(f)$ the set of all mean rotation vectors of elements of $\mc{M}(f)$ and $\roterg(f)$ the corresponding set for ergodic measures. The proof of \cite[Theorem 2.4]{m-z}, without requiring any modifications, implies that
$$\rotmeas(f) = \conv(\roterg(f)) = \conv(\rotmz(f)).$$
In particular, every extremal point of the convex hull of $\rotmz(f)$ is the rotation vector of some ergodic measure (and therefore, it is the rotation vector of some recurrent point).

\subsection{Homotopically bounded sets and bounded homological displacement}\label{sec:displacement}

If $U_0\subset S$ is a connected surface with boundary, we denote by $\cH(U_0)\subset H_1(S,\R)$ the image of the map $H_1(U_0,\R)\to H_1(S,\R)$ induced by the inclusion $U_0\to S$. Note that $\cH(U_0)$ is a vector subspace of $H_1(S,\R)$  and $\cH(\fill(U_0)) = \cH(U_0)$. Moreover, if $U_0$ is not fully essential, then $\cH(U_0)$ is a proper subspace of $H_1(S,\R)$.
If $U_1,\dots U_m$ are the connected components of $S\sm U_0$, we let $\cV(U_0) = \bigcup_{i=0}^m \cH(U_i)$. If $U_0$ is essential but not fully essential, then $\cV(U_0)$ is a union of proper subspaces of $H_1(S,\R)$. 

If $U\subset S$ is a filled connected open set, we may find a surface with boundary $U_0\subset U$ such that $U\sm U_0$ is a finite union of topological annuli which are essential in $S$ (see Section \ref{sec:ends} and Proposition \ref{pro:ends}). We then define $\cV(U) = \cV(U_0)$, which is independent of the choice of $U_0$.
Finally, for an arbitrary connected open set, define $\cV(U) = \cV(\fill(U))$. Note that by Proposition \ref{pro:ends}, the set $S\sm U_0$ has at most $g$ connected components (where $g$ is the genus of $S$). Therefore, $\cV(U)$ is the union of at most $g+1$ proper subspaces of $H_1(S,\R)$ (which are in fact generated by elements of $H_1(S,\Z)$).

For a set $E\subset H_1(S,\R)$ and $r>0$ denote by $B_r(E) = \bigcup_{v\in E} \{w\in H_1(S,\R): \norm{v-w}<r\}$ the $r$-neighborhood of $E$. We begin with a general lemma.

\begin{lemma} \label{lem:desvio-hyp}If there exists an open invariant homotopically bounded set $U$ which is essential but not fully essential, then there is a constant $r>0$ such that $\Phi_f^n(x)\in B_r(\cV(U))$ for all $x\in S$ and $n\in \N$. In particular, $\rotmz(f) \subset \cV(U)$. 
\end{lemma}

\begin{proof}
We may assume that $U$ is filled, replacing it by $\fill(U)$ if necessary.
Since $U$ is homotopically bounded, for each $x\in U$ and $n\in \N$ there exists a rectifiable arc $\beta$ joining $f^{n}(x)$ to $x$ of length at most $\diamup(U)$ such that $\beta$ is homotopic with fixed endpoints to an arc contained in $U$. This implies by (\ref{eq:rot-alt}) that $\snorm{[\mc{I}^n_x*\beta] - \Phi_f^{n}(x)}\leq M_0 = 2C_{\mc{A}}+\diamup(U)$. Since $U$ is invariant and essential, by Lemma \ref{lem:hyp-isotopy} the arc $\mc{I}^n_x$ is homotopic with fixed endpoints to an arc contained in $U$.  Thus $\mc{I}^n_x*\beta$ is homotopic to a loop in $U$, and $\Phi_f^n(x)\in B_{M_0}(\cH(U))$.

Let $U_0$ be a filled essential connected compact surface with boundary such that $U\sm U_0$ is a union of topological annuli. Increasing $U_0$ we may find $U_0'$ with the same properties such that $U_0\subset U_0'$ and $f(U_0)\subset U_0'$. Let $W_1,\dots W_m$ be the connected components of $S\sm U_0$, and $W_1',\dots W_m'$ the corresponding components of $U_0'$ such that $W_i'\subset W_i$. The fact that the topological ends of $U$ are fixed by $f$ (see Proposition \ref{pro:ends}) implies that $W_i'$ contains an essential loop homotopic in $W_i'$ to its own image. In particular $W_i'$ intersects $f(W_i')$ and so $f(W_i')\subset W_i$ as well. By Lemma \ref{lem:hyp-isotopy},  for each $x\in W_i'$ the arc $\mc{I}_x$ is homotopic with fixed endpoints to an arc contained in $W_i$. If $x\in (S\sm U)\cap W_i'$, then $f^n(x)\in W_i$ for all $n$ so $\mc{I}_x^n$ is homotopic with fixed endpoints to an arc in $W_i$, and we may choose an arc $\beta$ of length at most $\diamup(W_i)$ connecting $f^n(x)$ to $x$ and homotopic with fixed endpoints to an arc in $W_i$. This implies that $\Phi_f^n(x)\in B_{r_i}(\cH(W_i))$, where $r_i=2C_{\mc{A}}+\diamup(W_i)$. Thus the statements of the lemma hold setting $r=\max\{r_1,\dots, r_m, M_0\}$.
\end{proof}

The previous lemma does not hold if one drops the condition that $U$ be homotopically bounded, as the example in Figure \ref{fig:rotbad} shows. In the example, the map $f$ is a full rotation along each circle $b_1$ and $b_2$, and there exists an orbit with $\alpha$-limit on $b_1$ and $\omega$-limit on $b_2$. It is easy to show that $\rotmz(f)$ contains the line segment $\{t[b_1]+(1-t)[b_2] : t\in [0,1]\}$, which is not contained in $\mc{V}(U)$. However, the claim about $\rotmz(f)$ holds in the special case that $\bd U$ is fixed pointwise.
\begin{figure}[ht]
\includegraphics[height=3.5cm]{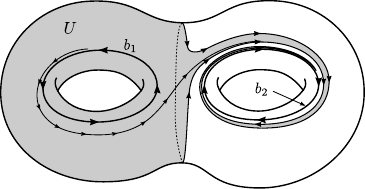}
\caption{$\rotmz(f)$ contains a segment joining $[b_1]$ to $[b_2]$}
\label{fig:rotbad}
\end{figure}

\begin{lemma}\label{lem:desvio-fix} If $U$ is an essential but not fully essential $f$-invariant open connected set and $\bd U\subset \fix(f)$, then $\rotmz(f)\subset \cV(U)$.
\end{lemma}

\begin{proof}

As usual it suffices to consider the case where $U$ is filled. By Proposition \ref{pro:ends}, the boundary of $U$ has at most $2g$ connected components, where $g$ is the genus of $U$. Let $K_1,\dots, K_l$ be the connected components of $\bd U$.

Since $\bd U\subset \fix(f)$, for each $x\in \bd U$ the isotopy arc $\mc{I}_x$ is a loop. The map $\bd U\to H_1(S,\Z)$ defined by $x\mapsto [\mc{I}_x]$ is continuous and therefore constant on each connected component of $\bd U$. For $1\leq i\leq l$, let $v_i$ be the value of $[\mc{I}_x]$ on each set $K_i$. 

Fix $\epsilon>0$ such that the $\epsilon$-neighborhoods of the sets $K_i$ are pairwise disjoint. Let $0<\delta<\epsilon/2$ be such that $d(f(x), f(y))< \epsilon/2$ whenever $d(x,y)<\delta$. Let $U_\delta\subset U$ be a compact connected filled manifold with boundary such that $S\sm U_\delta$ is contained in the $\delta$-neighborhood of $S\sm U$ and $U\sm U_\delta$ is a union of topological annuli. Note that $d(f(x),x)<\epsilon$ for all $x\in U\sm U_\delta$ due to our choice of $\delta$ and the fact that $\bd U\subset \fix(f)$.

We claim that $\{v_1,\dots v_l\}\subset \cH(U)$. Indeed, fix $i\in \{1,\dots,l\}$ and choose a simple arc $\sigma$ joining a point of $U\sm U_\delta$ to some point of $y_0\in K_i$ such that $\sigma\sm\{y_0\} \subset U\sm U_\delta$.
Let $A$ be the connected component of $U\sm U_\delta$ containing $\sigma\sm\{y_0\}$ (so $A$ is a collar of some topological end of $U$).
Since the ends of $U$ are fixed, $A$ contains some smaller collar $A'\subset A$ of the same topological end such that $f(A')\subset A$.  By Lemma \ref{lem:hyp-isotopy}, for any $x\in A'$ the arc $\mc{I}_x$ is isotopic with fixed endpoints in $S$ to a an arc contained in $A$. This implies that, if $\hat{A}$ is a connected component of $\pi^{-1}(A)$ and $\hat{A}'$ is the component of $\pi^{-1}(A')$ in $A$, then $\hat{f}(\hat{A}') \subset \hat{A}$. 
Replacing $A'$ by a smaller annulus we may further assume that $\hat{f}^2(\hat{A}')\subset \hat{A}$. We will also assume $\sigma\subset A'$, by reducing $\sigma$ if necessary.
Let $\hat{\sigma}$ be a lift of $\sigma$ with initial point in $\hat{A}'$, and let $y\in \pi^{-1}(y_0)$ be the endpoint of $\hat{\sigma}$. Then $\hat{\sigma}\sm \{y\}\subset \hat{A}'$, and since $\pi(y) = y_0\in K_i$, we have that $\hat{f}(y)=Ry$ for some deck transformation $R$. Note that $\hat{f}^2(\sigma)\sm \{R^2y\}$ is an arc in $\hat{A}$. Let $\alpha$ be an arc in $\hat{A}$ joining the initial point $x$ of $\hat{\sigma}$ to $\hat{f}^2(x)$, and let $\eta = \hat{\sigma}^{-1}*\alpha*\hat{f}^2(\hat{\sigma})$. See Figure \ref{fig:lema-fix1}. Since $\eta$ joins $y$ to $R^2y$, by Lemma \ref{lem:brouwer-free} we have $\eta\cap R\eta\neq \emptyset$, but since the endpoints of $\eta$ and $R\eta$ are disjoint it follows that the arcs $\eta$ and $R\eta$ with their endpoints removed intersect. Since $\eta$ lies in $\hat{A}$ except for its endpoints, this implies that $\hat{A}\cap R\hat{A}\neq \emptyset$, so $\hat{A}=R\hat{A}$. This means that $[R]\in \cH(A)$ (using the notation introduced at the beginning of Section \ref{sec:rot}). Since $\mc{I}_{y_0}$ lifts to an arc joining $y$ to $Ry$, it follows that $v_i=[\mc{I}_{y_0}]=[R]\in \cH(A)\subset \cH(U)$.

\begin{figure}[ht]
\includegraphics[width=.9\linewidth]{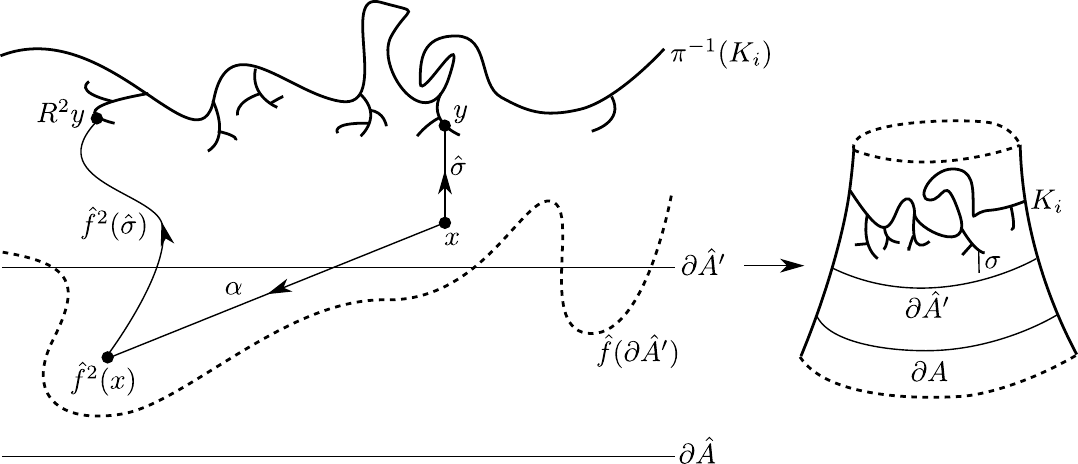}
\caption{Construction of the arc $\eta$.}
\label{fig:lema-fix1}
\end{figure}
Now fix $m\in \N$. Reducing $\delta$ (and increasing $U_\delta$ accordingly), we may assume that $d(f^i(x), f^i(y))\leq \epsilon/2$ whenever $d(x,y)<\delta$ and $0\leq i\leq m$. If $x\in U\sm U_\delta$, we may find $y\in \bd U$ such that $d(x,y)<\delta$. Since $y\in \fix(f)$, this implies that $d(f^m(x),x)<\epsilon$.  Let $i\in \{1,\dots, l\}$ be such that $y\in K_i$. If $\beta_x$ is the geodesic arc joining $f^m(x)$ to $x$ and $t\mapsto \sigma(t)$ is the geodesic arc joining $x$ to $y$ (hence contained in the $\delta$-ball around $y$), then $\mc{I}_{\sigma(t)}^m*\beta_{\sigma(t)}$ is a free homotopy between $\mc{I}_x^m*\beta_x$ and the loop $\mc{I}_y^m$. Hence $[\mc{I}_x^m*\beta_x] = [\mc{I}_y^m] = m[v_i]$, which means by (\ref{eq:rot-alt}) that, assuming $\epsilon<1$, 
$$\norm{\Phi_f^m(x)-m[v_i]}\leq M':=2C_{\mc A} + 1.$$

Suppose that $x\in U$ is such that $f^k(x)\in U\sm U_\delta$ for $0\leq k\leq n$. Recall that we chose $\epsilon$ such that the $\epsilon$-neighborhoods of the sets $K_i$ are pairwise disjoint, so by our choice of $\delta$ this implies that there exists $i\in \{1,\dots,l\}$ such that $d(f^k(x), K_i)<\delta$ for $0\leq k \leq n$. Writing $n=mj+r$ with $0\leq r<m$ and $k\geq 0$, recalling that $x\mapsto \norm{\Phi_f(x)}$ is bounded, we have that 
$$\Phi_f^n(x) - nv_i = \Phi_f^r(x)-rv_i+ \sum_{i=0}^{j-1}\Phi_f^m(f^{mi}(x))-mv_i,$$
which has a norm bounded by $C_m+M'j\leq C_m+M'n/m$, where 
$$C_m=\sup\{\norm{\Phi_f^r(x)-rv_i} :x\in S,\, 1\leq i\leq l,\, 0\leq r \leq m \}<\infty.$$
Since $v_i\in \cH(U)$ for each $i$, we have that $\Phi_f^n(x)\in B_{C_m+M'n/m}(\cH(U)).$

Now suppose that $x\in U$ is such that $x\in U_\delta$ and $f^n(x)\in U_\delta$. Then by Lemma \ref{lem:hyp-isotopy} the arc $\mc{I}^n_x$ is homotopic with fixed endpoints to an arc contained in $U$, and we may find an arc $\beta$ of length at most $\diamup(U_\delta)$ joining $f^n(x)$ to $x$ and isotopic with fixed endpoints to an arc in $U$. Thus $\mc{I}_x^n*\beta$ is homotopic to an arc in $U$ and again by (\ref{eq:rot-alt}), 
$$\norm{\Phi_f^n(x) - \mc{I}_x^n*\beta}\leq M_m :=2C_{\mc{A}}+\diamup(U_\delta).$$
which means that $\Phi_f^n(x) \in B_{M_{m}}(\cH(U))$ (we write $M_m$ to emphasize the dependence of $\delta$ on $m$).

For an arbitrary $x\in U$, let $a$ and $b$ be the smallest and largest positive integers such that $f^a(x)\in U_\delta$ and $f^b(x)\in U_\delta$, respectively. Then applying the previous estimations,
$$\Phi_f^n(x) = \Phi_f^{a}(x) + \Phi_f^{b-a}(f^a(x)) + \Phi_f^{n-b}(f^{b}(x)) \in B_{r_m}(\cH(U)),$$
where 
$$r_m = C_m+a\frac{M'}{m} + M_m+ C_m+(n-b)\frac{M'}{m}\leq 2C_m + M_m + M'\frac{n}{m}.$$
It follows that $\rotmz(f,U) \subset B_{M'/m}(\cH(U))$, and since this holds for any $m\in \N$, we conclude that $\rotmz(f,U)\subset \cH(U)\subset \cV(U)$. On the other hand, for $x\in S\sm U$, the argument used in the last paragraph of the proof of Lemma \ref{lem:desvio-hyp} applies exactly in the same way, implying that there exists $r$ such that $\Phi_f^{n}(x)\in B_r(\cV(U))$ for all $n\in\N$, and therefore $\rotmz(f, S\sm U)\subset \cV(U)$. Thus $\rotmz(f,S)=\rotmz(f, S\sm U)\cup \rotmz(S, U) \subset \cV(U)$.
\end{proof}

Finally we consider the case where $\fix(f)$ is fully essential.

\begin{lemma}\label{lem:rot0} If $\fix(f)$ is fully essential, then $\rotmz(f)=\{0\}$. 
\end{lemma}
\begin{proof}
For each $T\in \deck(\pi)$ let $K_T=\pi(\fix(T^{-1}\hat{f}))$. These sets partition $\fix(f)$ into finitely many compact sets, which are the Nielsen classes of $\hat{f}$ (see \cite{nielsen-book}). 
If $\fix(f)$ is fully essential, then some connected component $K$ of $\fix(f)$ is fully essential, and since it is connected $K\subset K_R$ for some $R$. We claim that $R = \id$.  Suppose on the contrary that $R\neq \id$. Since $\hat{f}$ is a natural lift, it has a fixed point $z_0$. Let $U$ be the connected component of $\hat{S}\sm \pi^{-1}(K)$ containing $z_0$. Let $\sigma$ be an arc joining $z_0$ to a point $z\in \bd U$ such that $\sigma$ lies entirely in $U$ except for its endpoint $z$. Then there is a deck transformation $T\neq \id$ such that $\hat{f}(z)=Tz$, so $\hat{f}^2(\sigma)$ is an arc joining $z_0$ to $T^2z$, and since $\hat{f}(U)=U$ the arc $\hat{f}^2(\sigma)$ is also in $U$ except for its endpoint. Let $\ol{\alpha}=\sigma^{-1}*\hat{f}^2(\sigma)$, which is a compact arc joining $z$ to $T^2 z$, and denote by $\alpha$ the arc $\ol{\alpha}$ with its endpoints removed, so $\alpha\subset U$. Since $T^2\ol{\alpha}$ intersects $\ol{\alpha}$ (at $T^2z$), it follows from classical Brouwer theory (see Lemma \ref{lem:brouwer-free}) that $T\ol{\alpha}$ intersects $\ol{\alpha}$ as well. But the endpoints of $T\ol{\alpha}$ are disjoint from the endpoints of $\ol{\alpha}$, so $\alpha$ intersects $T\alpha$. Hence $U\cap TU\neq \emptyset$, which implies that $U=TU$ and so $\pi(U)$ is essential, contradicting the fact that $K$ is fully essential and disjoint from $\pi(U)$.

Thus $K = \pi(\fix(\hat{f}))$, and this implies that $\pi^{-1}(K)\subset \fix(\hat{f})$, and every connected component of $\hat{S}\sm \pi^{-1}(K)$ is invariant (see \cite{brown-kister}). Since any such component projects injectively to a topological disk, it follows that for any recurrent point $x$ of $f$, every $\hat{x}\in \pi^{-1}(x)$ is recurrent for $\hat{f}$. 

Since $\conv(\rotmz(f))=\conv(\rotmeas(f))$ (see \S\ref{sec:rot-meas}), to conclude that $\rotmz(f)=\{0\}$ it suffices to show that $\rotmeas(f,\mu)=0$ for any ergodic measure $\mu$. But for any ergodic measure, $\mu$-almost every point (in particular some point $x\in S$) is recurrent and has a well-defined rotation vector $\rho(f,x)=\rotmeas(f,\mu)$. If $\hat x\in \pi^{-1}(x)$, then $\hat{x}$ is recurrent for $\hat{f}$, so we may find $n_k\to \infty$ such that $\hat{f}^{n_k}(\hat{x})\to \hat{x}$. Letting $\beta$ be the geodesic segment joining $f^n(x)$ to $x$ we see that $[\mc{I}_x^{n_k}*\beta]=0$ for large enough $k$, and it follows from (\ref{eq:rot-alt}) that $\rho(f,x)=0$, as we wanted to show.
\end{proof}

Finally,

\begin{lemma}\label{lem:rot-bounded}
For any $x\in \hat{S}$ and $n\in \N$,
 $$\norm{\Phi_f^n(\pi(x))} \leq 2C_{\mc{A}}+d(\hat{f}^n(x),x).$$
In particular, if $f$ has uniformly bounded displacement then $\rotmz(f)=\{0\}$.
\end{lemma}
\begin{proof}
It follows immediately from the fact that the arc $\mc{I}^n_{\pi(x)}$ is homotopic with fixed endpoints to a rectifiable arc of length $d(\hat{f}^n(x), x)$, since $\hat{f}$ is a natural lift.
\end{proof}

\begin{theorem}\label{th:rational} Suppose that $f$ is area-preserving. Then either $f$ has a fully essential dynamics, or $\rotmz(f)$ is contained in a union of at most $g+1$ proper rational subspaces of $H_1(S,\R)$.
\end{theorem}

\begin{proof}
If $f$ does not have fully essential dynamics, then one of items (1) to (3) from Theorem \ref{th:essine} holds. In case (1), \ie if $\fix(f^n)$ is essential, then either some connected component $U$ of $S\sm \fix(f^n)$ is essential but not fully essential, or $\fix(f^n)$ is fully essential. Since $\rotmz(f^n) = n\rotmz(f)$, by Lemmas \ref{lem:desvio-fix} and \ref{lem:rot0}, either $\rotmz=\{0\}$ or $\rotmz(f)\subset \cV(U)$, which implies the claims of the theorem.
For case (2), \ie $f$ has uniformly bounded displacement, then Lemma \ref{lem:rot-bounded} implies that $\rotmz(f)=\{0\}$. For case (3), \ie there is homotopically bounded invariant open connected set $U$ which is essential but not fully essential, Lemma \ref{lem:desvio-hyp} implies that $\rotmz(f)\subset \cV(U)$ and again our claims hold.
\end{proof}

\subsection{A uniform property of displacement of fully essential points}\label{sec:shadow}

\begin{lemma}\label{lem:localrot-uniform} If $z\in \fess(f)$, then for any neighborhood $U$ of $z$ there exists $r>0$ such that
$$\{\Phi_f^n(x) : x\in S\} \subset B_r(\{\Phi_f^n(y):y\in U\}).$$
for all $n\in \N$. In particular, $\rotmz(f,U) = \rotmz(f)$.
\end{lemma}
\begin{proof}

Reducing $U$ if necessary, we may assume that $U$ is a small topological disk. Since $z\in \fess(f)$, there exists $m$ such that $W = \bigcup_{i=-m}^m f^i(U)$ is fully essential. This implies (see Remark \ref{rem:Q}) that there is a loop $\gamma \subset \hat{S}$ bounding a closed topological disk $Q$ such that $\pi(\gamma)\subset W$ and $\pi(Q) = S$. 
Let $\hat f$ be the natural lift of $f$. Since the natural lift is a bounded distance away from the identity, there exists a constant $L>0$ such that $d(\hat{f}(w), w) < L$ and $d(\hat{f}^{-1}(w), w) < L$ for all $w\in \hat S$.

Given any $x\in S$ and $n\in \N$, choose $\hat{x}\in \pi^{-1}(Q)$ and let $T\in \deck(\pi)$ be such that $\hat{f}^n(x) \in TQ$. Since $\hat{f}^n(Q)\cap TQ\neq \emptyset$, there are two possibilities: either $\bd \hat{f}^n(Q) \cap TQ\neq \emptyset$ or $\hat{f}^n(Q) \supset TQ$. We claim that in both cases there exists $\hat{y}\in \pi^{-1}(U)$ such that $d(\hat{f}^n(\hat y), T\hat y) < 2mL + \diam(Q)$. 
Assume first that $\bd \hat f^n(Q)\cap TQ\neq \emptyset$. Then there is $\hat{y}_0\in \bd Q$ such that $\hat{f}^n(\hat y_0)\in TQ$, and since $\bd Q\subset \pi^{-1}(W)$ we have that $\hat y := \hat f^{i}(\hat y_0) \in \pi^{-1}(U)$ for some $i$ with $|i|\leq m$. Since $\hat f^{n}(\hat y_0) \in TQ$, it follows that $d(\hat{f}^{n-i}(\hat y), T\hat y_0) \leq \diam(TQ) = \diam(Q)$, and since $|i|\leq m$ it follows that $d(\hat{f}^n(\hat y), T\hat y_0) < mL + \diam(Q)$. Moreover, for a similar reason $d(T\hat y_0, T\hat y) = d(\hat y_0, \hat y) \leq mL$, so we conclude that $d(\hat{f}^n(\hat y), T\hat y) \leq 2mL+\diam(Q)$ as claimed.

Now assume that $\hat{f}^n(Q)\supset TQ$. Since $\pi(TQ) = S$ we know that there is some $w\in \pi^{-1}(f^n(U)) \cap TQ \subset \hat{f}^n(Q)$, so $\hat{y} = \hat{f}^{-n}(w)$ is a point of $\pi^{-1}(U)\cap Q$ such that $\hat{f}^n(\hat y)\in TQ$. Thus $d(\hat{f}^n(\hat y), T\hat y) \leq \diam(Q) \leq 2mL+\diam(Q)$ as claimed.

Hence in any case there exists $\hat{y}\in \pi^{-1}(\hat U)$ such that $d(\hat{f}^n(\hat y), T\hat y) < 2mL + \diam(Q)$. To complete the proof, let $y=\pi(\hat y)\in U$ and note that the arc $\mc{I}_y^n$ lifts to an arc $\hat {\mc{I}}_y^n$ joining $\hat{y}$ to $\hat{f}^n(\hat{y})$. If $\beta_y$ denotes the projection of the geodesic segment joining $\hat f^n(\hat y)$ to $T\hat y$, we see that $[\mc{I}_y^n * \beta_y] = [T]$ and the length of $\beta_y$ is at most $\diam(Q) + 2mL$. Similarly if $\beta_x$ is the projection of the geodesic segment joining $\hat{f}^n(\hat{x})$ to $T\hat{x}$ we have $[\mc{I}_x^n * \beta_x] = [T]$, and since $\hat{x}\in Q$ and $\hat{f}^n(\hat x)\in TQ$ the length of $\beta_x$ is at most $\diam(Q)$. By (\ref{eq:rot-alt}) we have 
$$\norm{[T]-\Phi_f^n(y)} = \norm{[\mc{I}_y^n*\beta_y]-\Phi_f^n(y)}\leq 2C_{\mc A} + \length(\beta_y)\leq 2C_{\mc A}+\diam(Q) + 2mL,$$
and a similar property holds with $x$ in place of $y$, so setting $r = 2(2C_{\mc A} + \diam(Q) + 2mL)$ we have 
$$\norm{\Phi_f^n(y) - \Phi_f^n(x)} < r.$$
Since the choice of $r$ does not depend on $x$, this completes the proof.
\end{proof}

\begin{lemma}\label{lem:diffusion} If $U\subset S$ is a topological disk containing $x\in \fess(f)$, then $$\liminf_{n\to \infty} \diamup(f^n(U))/n \geq \diam(\rotmz(f)).$$
\end{lemma}
\begin{proof}

We begin with an observation which follows from (\ref{eq:rot-alt}): if $M=\diamup(S)+2C_{\mc{A}}$, then for every $\hat{x}\in \hat{S}$ and $n\in \N$ there exists $T\in \deck(\pi)$ such that $[T]=\Phi^n_f(x)$ and $d(\hat{f}^n(\hat{x}), T\hat{x})<M$. 

We will assume that $\diamup(U)<M$, by reducing $U$ if necessary. We may further assume that $\rotmz(f)\neq\{0\}$, since otherwise there is nothing to be done. Since $\rotmz(f)$ is compact and bounded, the diameter of $\rotmz(f)$ is realized as the distance of two extremal points $v_1 \neq v_2$ of its convex hull. As explained in Section \ref{sec:rot-meas}, such extremal points are realized by ergodic measures, and in particular by points. Hence there exist $x_1, x_2\in S$ such that $\rho(f,x_i) = v_i$. By Lemma \ref{lem:localrot-uniform}, for each $n\in \N$ we may find $y_1, y_2\in U$ such that $\Phi^n_f(y_i)\in B_r(\Phi^n_f(x_i))$. Fix a connected component $\hat{U}$ of $\pi^{-1}(U)$, and let $\hat{y}_i$ be the lift of $y_i$ in $\hat{U}$. By the observation at the beginning of the proof, there exist deck transformations $T_i$ such that $[T_i]= \Phi^n_f(y_i)$ and $d(\hat{f}^n(\hat{y}_i), T_i\hat{y}_i)\leq M$. 
Thus, 
$$d(T_1^{-1}T_2\hat{y}_2, \hat{y}_1) = d(T_2\hat{y}_2, T_1\hat{y}_1) \leq d(\hat{f}^n(\hat y_2), \hat{f}^n(\hat y_1))+2M\leq \diam(\hat{f}^n(\hat{U}))+2M.$$
Since $d(\hat y_1, \hat y_2) < \diam(\hat U) = \diamup(U) < M$, this means that $$d(T_1^{-1}T_2\hat y_2, \hat y_2) \leq \diam(\hat f^n(\hat U)) + 3M.$$
Let $\eta_n$ be the projection into $S$ of the geodesic segment joining a point $\hat{y}_2$ to $T_2T_1^{-1}\hat{y}_2$ in $\hat{S}$. Then 
$$\diamup(f^n(U)) = \diam(\hat{f}^n(\hat{U})) \geq \length(\eta_n) -3M.$$
In addition,
 $$[\eta_n] = [T_2T_1^{-1}] = [T_2]-[T_1] = \Phi^n_f(y_2)-\Phi^n_f(y_1)\in B_{2r}(\Phi^n_f(x_2) -\Phi^n_f(x_1)).$$ 
Thus 
\begin{equation}\label{eq:diff1}
\lim_{n\to \infty} \snorm{[\eta_n]}/n = \snorm{\rho(f,x_2)-\rho(f,x_1)}=\diam(\rotmz(f))>0.
\end{equation}
Since we are using the stable norm and $\eta_n$ is the shortest rectifiable loop with basepoint $y_2$ in its homology class (and $\snorm{[\eta_n]}\to \infty$),
\begin{equation}\label{eq:diff2}
\lim_{n\to \infty} \snorm{[\eta_n]}/\length(\eta_n)= 1
\end{equation}
(see \cite[Lemma 4.20$\frac{1}{2}\text{bis}_+$]{gromov}, noting that changing the basepoint amounts to adding a uniform constant to the loop length). 
Thus we conclude that 
$$\frac{\diamup(f^n(U))}{n} \geq \frac{\length(\eta_n)-3M}{n} = \frac{\length(\eta_n)}{\snorm{[\eta_n]}}\cdot\frac{\snorm{[\eta_n]}}{n} - \frac{3M}{n}$$
which, using (\ref{eq:diff1}) and (\ref{eq:diff2}), implies our claim.
\end{proof}

\subsection{Proof of Theorem \ref{th:interior-essential}}
For the sake of clarity we restate the theorem here, with some additional detail. Recall that throughout this section $S$ is a closed hyperbolic surface and $f\colon S\to S$ is a homeomorphism homotopic to the identity.

\begin{theorem}\label{th:rational-all}
If $f$ is area-preserving and the rotation set of $f$ is not contained in the union of $g+1$ rational proper subspaces of $H_1(S,\R)$, then $f$ has fully essential dynamics and in addition:
\begin{itemize}
\item[(1)] $\fess(f)=\ess(f)$, which is a fully essential continuum, and $\ine(f)$ is a disjoint union of periodic bounded disks;
\item[(2)] $\fess(f)$ is externally syndetically transitive and sensitive on initial conditions;
\item[(3)] If $U$ is any neighborhood of  $x\in \fess(f)$, then $\rotmz(f,U)=\rotmz(f)$ and 
$$\liminf_{n\to \infty}\diamup(f^n(U))/n \geq \diam(\rotmz(f)) > 0.$$
\end{itemize}
\end{theorem}
\begin{proof}
That $f$ has fully essential dynamics follows from Theorem \ref{th:rational}. Part (1) and the first claim of part (2) follow from Theorem \ref{th:essine}. Part (3) follows from Lemma \ref{lem:localrot-uniform} and Lemma \ref{lem:diffusion}. Note that $\diam(\rotmz(f))>0$ because, $S$ being hyperbolic, $\rotmz(f)$ contains $0$ and by our assumptions it contains some nonzero vector as well. It remains to prove the sensitive dependence on initial conditions, but this follows from part (3), noting that diameter of $U$ in $S$ coincides with $\diamup(U)$ whenever $U$ is a sufficiently small topological disk.
\end{proof}

\subsection{Periodic points and measures: proof of Theorem \ref{th:reali-ess}}\label{sec:reali-ess}

For a fixed point $f(z)=z$, the rotation vector $\rho(f,z)$ exists and is an integer vector, meaning that $\rho(f,z)\in H_1(S,\Z)\subset H_1(S,\R)$.  If $f^n(z) = z$ is a periodic point, then $\rho(f,z)$ also exists and is rational, \ie $\rho(f,z)\in H_1(S,\Q)\subset H_1(S,\R)$.  

A periodic point $z$ of $f$ of is of type $(n,T)$ for $n\in \N$ and $T\in \deck(\pi)$ if some $\hat{z}\in \pi^{-1}(z)$ satisfies $\hat{f}^n(\hat{z}) = T\hat{z}$. The set of all periodic points of type $(n,T)$ is a Nielsen class for $f^n$. Two points in the same Nielsen class are said to be Nielsen equivalent.
It is easy to see from the definitions that if $z$ is a periodic point of type $(n,T)$, then $\rho(f,z)=[T]/n$. 
We omit the proof of the following lemma, which is equally straightforward:
\begin{lemma}\label{lem:ine-rational}
If $U\subset S$ is a homotopically bounded invariant topological disk, then $\rotmz(f,U)$ consists of a single element of $H_1(S,\Z)$.
\end{lemma}

\begin{proof}[Proof of Theorem \ref{th:reali-ess}]
For Part (2) note that from the observations from Section \ref{sec:rot-meas}, for an ergodic measure $\mu$ we have that $\mu$-almost every point $z$ is recurrent and has a well-defined rotation vector $\rho(f,z)$ equal to $\rotmeas(f,\mu)$, so if the support of $\mu$ intersects $\ine(f)$ (which is a union of periodic disks), Lemma \ref{lem:ine-rational} implies that $\rotmeas(f^n,\mu)\in H_1(S,\Z)$ for some $n$, so $\rotmeas(f,\mu)=\rotmeas(f^n,\mu)/n\in H_1(S,\Q)$.

 For part (1), it suffices to show that for any $T\in \deck(\pi)$ such that there is a periodic point of type $(n,T)$, there exists a periodic point of type $(n,T)$ in $\fess(f)$.
Suppose for a contradiction that there is at least one periodic point of type $(n,T)$, but no such point lies in $\fess(f)$. Then $E_T=\pi(\fix(T^{-1}\hat{f}^n))$ is disjoint from $\ess(f)$, hence it is contained in $\ine(f)$. Let $g=f^n$ and $\hat{g}=\hat{f}^n$. Since $\ine(f)$ is $g$-invariant and $E_T\subset \fix(g)$, each $z\in E_T$ belongs to some $g$-invariant open topological disk $D$, which is the connected component of $\ine(f)$ containing $z$. Since $E_T$ is compact, there exist finitely many invariant connected components $D_1,\dots, D_m$ of $\ine(f)$ such that $E_T\subset D_1\cup \cdots \cup D_m$ and each $D_i$ intersects $E_T$. For each $i$, let $\hat{D}_i$ be a lift of $D_i$ intersecting $\fix(T^{-1}\hat{g})$, so $G(\hat{D}_i)=\hat{D}_i$ where $G=T^{-1}\hat{g}$.  Then $\bd \hat{D}_i\subset \pi^{-1}(\ess(f))$, so by our assumption there are no fixed points of $G$ in $\bd \hat{D}_i$. Since $g|_{D_i}$ is area-preserving and $\pi|_{\hat{D}_i}$ is injective, $G|_{\hat{D}_i}$ is nonwandering. Moreover, $\hat{D}_i$ is bounded and has no fixed points of $G$ in its boundary, so by \cite[Lemma 4.2]{koro} the fixed point index of $G$ in $\hat{D}_i$ is $1$. This implies that the fixed point index of $g$ in $D_i$ is $1$ for each $i\in \{1,\dots,m\}$, which in turn means that the Nielsen class $E_T$ of $g$ has index $m>0$. 

But since $g$ is isotopic to the identity, by the Lefschetz-Nielsen theory the index of the Nielsen class $E_T$ is zero except when $T=\id$, in which case it is the Euler characteristic $\chi(S)\leq 0$ (see for instance \cite[Chapter 4]{nielsen-book}). In either case, we arrive to a contradiction, completing the proof of the first claim.
\end{proof}

\section{A final example}\label{sec:example}

Given a nonwandering homeomorphism $f\colon S\to S$ such that $\fix(f^n)$ is inessential for all $n\in \N$, it follows from Theorem \ref{th:main-disk} that for every periodic open topological disk $U$ is such that $\diamup(U)$ is bounded by a constant depending on the period of $U$. Let us describe an example to show that in general one may not expect this bound to be independent of the period of $U$. To obtain such an example, it suffices to prove the following:
\begin{claim*}
If $\alpha$ is an irrational number and $M>0$ is given, there exists an area-preserving diffeomorphism $f\colon \T^1\times [0,1]\to \T^1\times [0,1]$ arbitrarily $C^1$-close to the rigid rotation $R_\alpha\colon (x,y)\mapsto (x+\alpha, y)$ such that $f$ coincides with $R_\alpha$ on the two boundary circles, the fixed point set of $f^n$ is inessential for each $n\in \N$, and there is an $f$-periodic open topological disk $U$ in the interior of the annulus such that $\diamup(U)>M$.
\end{claim*}
Indeed, one may obtain an example with arbitrarily large periodic disks by fixing an irrational number $\alpha$ and defining on each annulus $A_n = \T^1\times [1/n, 1/(n+1)]$ a map as in the claim above, using $M=n$ and choosing the map $1/n$-close to the identity. This defines a homeomorphism on $\T^1\times (0,1]$ which extends to the circle $\T^1\times \{0\}$ as the rigid rotation by $\alpha$, and which has the required properties (to get an example on any surface we may do the above construction on some essential annulus and extend outside the annulus to a map with finitely many fixed points of each period)

Before proving the claim, we state a lemma which can be obtained by a straightforward adaptation of the proof of Proposition 2.2 from \cite{centralizer}.  

\begin{lemma} For every $\epsilon>0$ there exists $\delta>0$ such that, if $g\colon \T^1 \to \T^1$ is a diffeomorphism of the circle which is $\delta$-close to the identity in the $C^1$-topology,  there exists an area preserving diffeomorphism $f\colon \T^1\times [0,1]\to \T^1\times [0,1]$ which is $\epsilon$-close to the identity in the $C^1$-topology satisfying $f(x,1)= (g(x),1)$ and $f(x, 0)= (x,0)$ for all $x\in \T^1$.
\end{lemma}

\begin{remark}\label{rem:gluing} If $g$ is $\delta$-close to some rigid rotation $R_\theta$ of $\T^1$, then applying the above lemma to $R_\theta^{-1}g$ one concludes that there exists an area-preserving diffeomorphism $f$ of the annulus which is $\epsilon$-close to the rotation by $\theta$ such that $f(x,1) = (g(x), 1)$ and $f(x,0) = (R_\theta(x), 0)$ for all $x\in \T^1$.
\end{remark}

To prove the claim, let $(\phi_t)_{t\in \R}$ denote a topological flow on the closed annulus, lifting to a flow $(\hat\phi_t)_{t\in \R}$ of $\R\times [0,1]$ such that:
\begin{itemize}
\item $\phi_t$ is area-preserving for each $t$;
\item The unit square $\hat{U}_0=(0,1)^2$ is $\hat{\phi}_t$-invariant;
\item There are finitely many singularities and no essential closed orbit.
\end{itemize}

Let $A_i = \T^1\times [i, i+1]$, and, given $\epsilon>0$, let $\delta<\epsilon/2$ be as in Remark \ref{rem:gluing}.
Fix a rational number $p/q$ such that $|\alpha-p/q|<\delta/2$ and consider
$$H(x,y) = (x/q, y) + ((M+1)\sin(2\pi y), 0).$$
The projection of the topological disk $\hat U = H(\hat U_0)$ onto the first coordinate has diameter greater than $M$, so in particular $\diam(\hat U)>M$. Consider the maps $T\colon(x,y)\mapsto (x+1,y)$ and $\hat{R}_{p/q}\colon (x,y)\mapsto (x+p/q, y)$. Note that $R_{p/q}H = HT$. Letting
$$\hat{f}_t =  R_{p/q}H\hat \phi_t H^{-1} = HT\hat \phi_t H^{-1}.$$
it is easy to verify that $\hat{f}_t^{q} = T^{p}\hat{f}_{qt}$,  so $\hat{f}_t^{q}(\hat{U}) = T^{p} \hat{U}$. Denote by $f_t$ the map induced by $\hat{f}_t$ on $\T^1\times \R$, and by $U$ the projection of $\hat{U}$. Note that $U$ is $f_t$-periodic.

Define $f$ on $A_0$ as $f=f_t$, where $t$ is chosen small enough so as to guarantee that $f_t$ is $\delta/2$-close to the rotation by $p/q$ on $A_0$ in the $C^1$ topology. This implies that the restrictions of $f$ to the two boundary circles of $A_0$ are $\delta$-close to the rotation by $\alpha$ in the $C^1$-topology, so using Remark \ref{rem:gluing} on the annuli $A_{-1}$ and $A_1$ we may extend $f$ to an area-preserving $C^1$-diffeomorphism of $\T^1\times [-1,2]$ which is $\epsilon$-close to the rotation by $\alpha$ and which coincides with that rotation on the two boundary circles $\T^1\times \{-1\}$ and $\T^1\times \{2\}$. Moreover, inside the annulus $A_0$ the set of fixed points of $f^n$ is inessential for each $n>0$, and there is an open periodic topological disk $U$ which lifts to a disk of diameter at least $M$. After a small perturbation of $f$ supported on the interior of the complement of $A_0$ we may assume that the set of fixed points of $f^n$ outside of $A_0$ is totally disconnected for each $n$, so rescaling the annulus vertically we obtain a map which satisfies the conditions of the claim. \qed

\subsection*{Acknowledgements}

The authors would like to thank Carolina Puppo and the anonymous referee for the useful corrections that helped improve this paper.

\bibliographystyle{koro} 
\bibliography{essine}

\providecommand{\bysame}{\leavevmode\hbox to3em{\hrulefill}\thinspace}
\providecommand{\MR}{\relax\ifhmode\unskip\space\fi MR }
\providecommand{\MRhref}[2]{%
  \href{http://www.ams.org/mathscinet-getitem?mr=#1}{#2}
}
\providecommand{\href}[2]{#2}
\begin{thebibliography}{KLCN15}

\bibitem[{Add}15]{zanata-disk}
S.~{Addas-Zanata}, \emph{Area-preserving diffeomorphisms of the torus whose
  rotation sets have non-empty interior}, Ergodic Theory and Dynamical Systems
  \textbf{35} (2015), 1--33.

\bibitem[AZ02]{zanata-vertical}
S.~Addas-Zanata, \emph{On the existence of a new type of periodic and
  quasi-periodic orbits for twist maps of the torus}, Nonlinearity \textbf{15}
  (2002), no.~5, 1399--1416. \MR{1925420 (2004b:37076)}

\bibitem[BCLR06]{beguin-schonflies}
F.~B{\'e}guin, S.~Crovisier, and F.~Le~Roux, \emph{Pseudo-rotations of the open
  annulus}, Bull. Braz. Math. Soc. (N.S.) \textbf{37} (2006), no.~2, 275--306.
  \MR{2266384 (2008b:37074)}

\bibitem[BCW08]{centralizer}
C.~Bonatti, S.~Crovisier, and A.~Wilkinson, \emph{$c^1$-generic conservative
  diffeomorphisms have trivial centralizer}, Journal of Modern Dynamics
  \textbf{2} (2008), no.~2, 359--373.

\bibitem[BK84]{brown-kister}
M.~Brown and J.~M. Kister, \emph{Invariance of complementary domains of a fixed
  point set}, Proc. Amer. Math. Soc. \textbf{91} (1984), no.~3, 503--504.
  \MR{744656 (86c:57014)}

\bibitem[Bro85]{brown}
M.~Brown, \emph{Homeomorphisms of two-dimensional manifolds}, Houston J. Math.
  \textbf{11} (1985), no.~4, 455--469. \MR{837985 (87g:57020)}

\bibitem[CB88]{casson}
A.~J. Casson and S.~A. Bleiler, \emph{Automorphisms of surfaces after {N}ielsen
  and {T}hurston}, London Mathematical Society Student Texts, vol.~9, Cambridge
  University Press, Cambridge, 1988. \MR{964685 (89k:57025)}

\bibitem[D{\'a}v13]{davalos}
P.~D{\'a}valos, \emph{On torus homeomorphisms whose rotation set is an
  interval}, Math. Z. \textbf{275} (2013), no.~3-4, 1005--1045. \MR{3127045}

\bibitem[Eps66]{epstein}
D.~B.~A. Epstein, \emph{Curves on {$2$}-manifolds and isotopies}, Acta Math.
  \textbf{115} (1966), 83--107. \MR{0214087 (35 \#4938)}

\bibitem[FM12]{farb-margalit}
B.~Farb and D.~Margalit, \emph{A primer on mapping class groups}, Princeton
  Mathematical Series, vol.~49, Princeton University Press, Princeton, NJ,
  2012. \MR{2850125 (2012h:57032)}

\bibitem[Fra88]{franks-gen}
J.~Franks, \emph{Generalizations of the {P}oincar\'e-{B}irkhoff theorem}, Ann.
  of Math. (2) \textbf{128} (1988), no.~1, 139--151. \MR{951509 (89m:54052)}

\bibitem[Fra89]{franks-reali}
\bysame, \emph{Realizing rotation vectors for torus homeomorphisms}, Trans.
  Amer. Math. Soc. \textbf{311} (1989), no.~1, 107--115. \MR{958891
  (89k:58239)}

\bibitem[Fra96]{franks-rot}
\bysame, \emph{Rotation vectors and fixed points of area preserving surface
  diffeomorphisms}, Trans. Amer. Math. Soc. \textbf{348} (1996), no.~7,
  2637--2662. \MR{1325916 (96i:58143)}

\bibitem[GKT14]{kt-annular}
N.~Guelman, A.~Koropecki, and F.~A. Tal, \emph{A characterization of annularity
  for area-preserving toral homeomorphisms}, Math. Z. \textbf{276} (2014),
  no.~3-4, 673--689. \MR{3175156}

\bibitem[GKT15]{kt-transitive}
N.~{Guelman}, A.~{Koropecki}, and F.~A. {Tal}, \emph{Rotation sets with
  non-empty interior and transitivity in the universal covering}, Ergodic
  Theory and Dynamical Systems \textbf{35} (2015), no.~3, 883--894.

\bibitem[Gro07]{gromov}
M.~Gromov, \emph{Metric structures for {R}iemannian and non-{R}iemannian
  spaces}, english ed., Modern Birkh\"auser Classics, Birkh\"auser Boston,
  Inc., Boston, MA, 2007, Based on the 1981 French original, With appendices by
  M. Katz, P. Pansu and S. Semmes, Translated from the French by Sean Michael
  Bates. \MR{2307192 (2007k:53049)}

\bibitem[Ham66]{hamstrom}
M.-E. Hamstrom, \emph{Homotopy groups of the space of homeomorphisms on a
  {$2$}-manifold}, Illinois J. Math. \textbf{10} (1966), 563--573. \MR{0202140
  (34 \#2014)}

\bibitem[Han90]{handel-annulus}
M.~Handel, \emph{The rotation set of a homeomorphism of the annulus is closed},
  Comm. Math. Phys. \textbf{127} (1990), no.~2, 339--349. \MR{1037109
  (91a:58102)}

\bibitem[J{\"a}g11]{jager-elliptic}
T.~J{\"a}ger, \emph{Elliptic stars in a chaotic night}, J. Lond. Math. Soc. (2)
  \textbf{84} (2011), no.~3, 595--611. \MR{2855792 (2012k:37102)}

\bibitem[{Jau}14]{jaulent}
O.~{Jaulent}, \emph{{Existence d'un feuilletage positivement transverse \`a un
  hom\'eomorphisme de surface}}, {Annales de l'institut Fourier} \textbf{64}
  (2014), no.~4, 1441--1476.

\bibitem[JM06]{nielsen-book}
J.~Jezierski and W.~Marzantowicz, \emph{Homotopy methods in topological fixed
  and periodic points theory}, Topological Fixed Point Theory and Its
  Applications, vol.~3, Springer, Dordrecht, 2006. \MR{2189944 (2006i:55003)}

\bibitem[KLCN15]{kln}
A.~Koropecki, P.~Le~Calvez, and M.~Nassiri, \emph{Prime ends rotation numbers
  and periodic points}, Duke Math. J. \textbf{164} (2015), no.~3, 403--472.
  \MR{3314477}

\bibitem[Kor10]{koro}
A.~Koropecki, \emph{Aperiodic invariant continua for surface homeomorphisms},
  Math. Z. \textbf{266} (2010), no.~1, 229--236. \MR{2670681 (2011j:37075)}

\bibitem[KT14a]{kt-example}
A.~Koropecki and F.~A. Tal, \emph{Area-preserving irrotational diffeomorphisms
  of the torus with sublinear diffusion}, Proc. Amer. Math. Soc. \textbf{142}
  (2014), no.~10, 3483--3490. \MR{3238423}

\bibitem[KT14b]{kt-pseudo}
\bysame, \emph{Bounded and unbounded behavior for area-preserving rational
  pseudo-rotations}, Proc. Lond. Math. Soc. (3) \textbf{109} (2014), no.~3,
  785--822. \MR{3260294}

\bibitem[KT14c]{kt-strictly}
\bysame, \emph{Strictly toral dynamics}, Invent. Math. \textbf{196} (2014),
  no.~2, 339--381. \MR{3193751}

\bibitem[LC05]{lecalvez-equivariant}
P.~Le~Calvez, \emph{Une version feuillet\'ee \'equivariante du th\'eor\`eme de
  translation de {B}rouwer}, Publ. Math. Inst. Hautes \'Etudes Sci. (2005),
  no.~102, 1--98. \MR{2217051 (2007m:37100)}

\bibitem[LM91]{llibre-mackay}
J.~Llibre and R.~S. MacKay, \emph{Rotation vectors and entropy for
  homeomorphisms of the torus isotopic to the identity}, Ergodic Theory Dynam.
  Systems \textbf{11} (1991), no.~1, 115--128. \MR{1101087 (92b:58184)}

\bibitem[MZ89]{m-z}
M.~Misiurewicz and K.~Ziemian, \emph{Rotation sets for maps of tori}, J. London
  Math. Soc. (2) \textbf{40} (1989), no.~3, 490--506. \MR{1053617 (91f:58052)}

\bibitem[Pol92]{pollicott}
M.~Pollicott, \emph{Rotation sets for homeomorphisms and homology}, Trans.
  Amer. Math. Soc. \textbf{331} (1992), no.~2, 881--894. \MR{1094554
  (92h:58116)}

\bibitem[Ric63]{richards}
I.~Richards, \emph{On the classification of noncompact surfaces}, Trans. Amer.
  Math. Soc. \textbf{106} (1963), 259--269. \MR{0143186 (26 \#746)}

\bibitem[{Tal}15]{tal2}
F.~A. {Tal}, \emph{On non-contractible periodic orbits for surface
  homeomorphisms}, Ergodic Theory and Dynamical Systems \textbf{FirstView}
  (2015), 1--12.

\bibitem[Whi33]{flow-foliation}
H.~Whitney, \emph{Regular families of curves}, Ann. of Math. (2) \textbf{34}
  (1933), no.~2, 244--270. \MR{1503106}

\end{thebibliography}

\end{document}